\documentclass[10pt,oneside]{amsart}

\usepackage[
paper=a4paper,
headsep=20pt,text={136mm,213mm},centering,includehead
]{geometry}

\usepackage[bookmarks]{hyperref}
\usepackage{amssymb,amsxtra,dsfont}
\usepackage{graphicx}
\usepackage[all]{xy}
\usepackage{pb-diagram,pb-xy}
\usepackage{color}
\usepackage{pinlabel}
\usepackage{float}
\usepackage{array,calc,booktabs}
\usepackage{enumitem}
\usepackage[
  textwidth=1.2in,
  backgroundcolor=yellow,
  bordercolor=orange,
  textsize=small
]{todonotes}

\iftrue
\makeatletter
\def\@settitle{%
  \vspace*{-0pt}
  \begin{flushleft}%
    \LARGE\bfseries
    \strut\@title\strut
  \end{flushleft}%
}
\def\@setauthors{%
  \begingroup
  \def\thanks{\protect\thanks@warning}%
  \trivlist
  \raggedright
  \large \@topsep27\p@\relax
  \advance\@topsep by -\baselineskip
  \item\relax
  \author@andify\authors
  \def\\{\protect\linebreak}%
  \authors
  \ifx\@empty\contribs
  \else
    ,\penalty-3 \space \@setcontribs
    \@closetoccontribs
  \fi
  \normalfont
  \endtrivlist
  \endgroup
}
\def\@setaddresses{\par
  \nobreak \begingroup
  \small\raggedright
  \def\author##1{\nobreak\addvspace\smallskipamount}%
  \def\\{\unskip, \ignorespaces}%
  \interlinepenalty\@M
  \def\address##1##2{\begingroup
    \par\addvspace\bigskipamount\noindent
    \@ifnotempty{##1}{(\ignorespaces##1\unskip) }%
    {\ignorespaces##2}\par\endgroup}%
  \def\curraddr##1##2{\begingroup
    \@ifnotempty{##2}{\nobreak\noindent\curraddrname
      \@ifnotempty{##1}{, \ignorespaces##1\unskip}\/:\space
      ##2\par}\endgroup}%
  \def\email##1##2{\begingroup
    \@ifnotempty{##2}{\nobreak\noindent E-mail address%
      \@ifnotempty{##1}{, \ignorespaces##1\unskip}\/:\space
      \ttfamily##2\par}\endgroup}%
  \def\urladdr##1##2{\begingroup
    \def~{\char`\~}%
    \@ifnotempty{##2}{\nobreak\noindent\urladdrname
      \@ifnotempty{##1}{, \ignorespaces##1\unskip}\/:\space
      \ttfamily##2\par}\endgroup}%
  \addresses
  \endgroup
  \global\let\addresses=\@empty
}
\def\@setabstracta{%
    \ifvoid\abstractbox
  \else
    \skip@17pt \advance\skip@-\lastskip
    \advance\skip@-\baselineskip \vskip\skip@
    \box\abstractbox
    \prevdepth\z@ 
    \vskip-15pt
  \fi
}
\renewenvironment{abstract}{%
  \ifx\maketitle\relax
    \ClassWarning{\@classname}{Abstract should precede
      \protect\maketitle\space in AMS document classes; reported}%
  \fi
  \global\setbox\abstractbox=\vtop \bgroup
    \normalfont\small
    \list{}{\labelwidth\z@
      \leftmargin0pc \rightmargin\leftmargin
      \listparindent\normalparindent \itemindent\z@
      \parsep\z@ \@plus\p@
      
    }%
    \item[\hskip\labelsep\bfseries\abstractname.]%
}{%
  \endlist\egroup
  \ifx\@setabstract\relax \@setabstracta \fi
}

\def\ps@headings{\ps@empty
  \def\@evenhead{%
    \setTrue{runhead}%
    \normalfont\scriptsize
    \rlap{\thepage}\hfill
    \def\thanks{\protect\thanks@warning}%
    \leftmark{}{}}%
  \def\@oddhead{%
    \setTrue{runhead}%
    \normalfont\scriptsize
    \def\thanks{\protect\thanks@warning}%
    \rightmark{}{}\hfill \llap{\thepage}}%
  \let\@mkboth\markboth
}\ps@headings

\def\section{\@startsection{section}{1}%
  \z@{-1.4\linespacing\@plus-.5\linespacing}{.8\linespacing}%
  {\normalfont\bfseries\Large}}
\def\subsection{\@startsection{subsection}{2}%
  \z@{-.8\linespacing\@plus-.3\linespacing}{.5\linespacing\@plus.2\linespacing}%
  {\normalfont\bfseries\large}}
\def\subsubsection{\@startsection{subsubsection}{3}%
  \z@{.7\linespacing\@plus.2\linespacing}{-1.5ex}%
  {\normalfont\bfseries}}
\def\@secnumfont{\bfseries}

\renewcommand\contentsnamefont{\bfseries}
\def\@starttoc#1#2{\begingroup
  \setTrue{#1}%
  \par\removelastskip\vskip\z@skip
  \@startsection{}\@M\z@{\linespacing\@plus\linespacing}%
    {.5\linespacing}{
      \contentsnamefont}{#2}%
  \ifx\contentsname#2%
  \else \addcontentsline{toc}{section}{#2}\fi
  \makeatletter
  \@input{\jobname.#1}%
  \if@filesw
    \@xp\newwrite\csname tf@#1\endcsname
    \immediate\@xp\openout\csname tf@#1\endcsname \jobname.#1\relax
  \fi
  \global\@nobreakfalse \endgroup
  \addvspace{32\p@\@plus14\p@}%
  \let\tableofcontents\relax
}
\def\contentsname{Contents}
\def\l@section{\@tocline{2}{.5ex}{0mm}{5pc}{}}
\def\l@subsection{\@tocline{2}{0pt}{2em}{5pc}{}}
\makeatother
\fi 

\def\to{\mathchoice{\longrightarrow}{\rightarrow}{\rightarrow}{\rightarrow}}
\makeatletter
\newcommand{\shortxra}[2][]{\ext@arrow 0359\rightarrowfill@{#1}{#2}}
\def\longrightarrowfill@{\arrowfill@\relbar\relbar\longrightarrow}
\newcommand{\longxra}[2][]{\ext@arrow 0359\longrightarrowfill@{#1}{#2}}

\def\addtagsub#1{\let\oldtf=\tagform@\def\tagform@##1{\oldtf{##1}\hbox{$_{#1}$}}}

\makeatother

\def\otimesover#1{\mathbin{\mathop{\otimes}_{#1}}}

\makeatletter
\def\Nopagebreak{\@nobreaktrue\nopagebreak}

\newtheoremstyle{theorem-giventitle}
        {}{}              
        {\itshape}                      
        {}                              
        {\bfseries}                     
        {.}                             
        {\thm@headsep}                             
        {\thmnote{\bfseries#3}}
\newtheoremstyle{theorem-givenlabel}
        {}{}              
        {\itshape}                      
        {}                              
        {\bfseries}                     
        {.}                             
        {\thm@headsep}                             
        {\thmname{#1}~\thmnumber{#3}\setcurrentlabel{#3}}

\newtheoremstyle{definition-giventitle}
        {}{}              
        {}                      
        {}                              
        {\bfseries}                     
        {.}                             
        {\thm@headsep}                             
        {\thmnote{\bfseries#3}}
\def\setcurrentlabel#1{\gdef\@currentlabel{#1}}

\makeatother

\newtheorem{theorem}{Theorem}[section]
\newtheorem{theoremalpha}{Theorem}
 
\newtheorem{proposition}[theorem]{Proposition}
\newtheorem{corollary}[theorem]{Corollary}
\newtheorem{lemma}[theorem]{Lemma}

\newtheorem{assertion}{Assertion}

\theoremstyle{definition}
\newtheorem{definition}[theorem]{Definition}

\newtheorem{remark}[theorem]{Remark}

\theoremstyle{theorem-giventitle}
\newtheorem{theorem-named}{}
\theoremstyle{theorem-givenlabel}
\newtheorem{theorem-labeled}{Theorem}

\theoremstyle{definition-giventitle}
\newtheorem{definition-named}{}
\newtheorem{step-named}{}

\numberwithin{equation}{section}

\newenvironment{describe}{\begin{proof}}{\end{proof}}

\def\Z{\mathbb{Z}}
\def\Q{\mathbb{Q}}
\def\R{\mathbb{R}}
\def\C{\mathbb{C}}

\def\F{\mathcal{F}}

\def\cE{\mathcal{E}}
\def\cG{\mathcal{G}}
\def\cR{\mathcal{R}}
\def\cP{\mathcal{P}}
\def\cW{\mathcal{W}}

\def\Ker{\operatorname{Ker}}

\def\Im{\operatorname{Im}}

\def\sign{\operatorname{sign}}
\def\rank{\operatorname{rank}}

\def\Bl{B\ell}

\def\Lt{L^2}
\def\rhot{\rho^{(2)}}

\def\mathbinover#1#2{\mathbin{\mathop{#1}\limits_{#2}}}
\def\amalgover#1{\mathbinover{\amalg}{#1}}
\def\cupover#1{\mathbinover{\cup}{#1}}

\def\setminus{\smallsetminus}

\begin{document}

\vspace*{0mm}

\title%
{Unknotted gropes, Whitney towers, and doubly slicing knots}

\author{Jae Choon Cha}
\address{
  Department of Mathematics\\
  POSTECH\\
  Pohang 37673\\
  Republic of Korea
  \quad -- and --\linebreak
  School of Mathematics\\
  Korea Institute for Advanced Study \\
  Seoul 02455\\
  Republic of Korea
}
\email{jccha@postech.ac.kr}

\author{Taehee Kim}
\address{
  Department of Mathematics\\
  Konkuk University \\
  Seoul 05029\\
  Republic of Korea
}
\email{tkim@konkuk.ac.kr}

\def\subjclassname{\textup{2010} Mathematics Subject Classification}
\expandafter\let\csname subjclassname@1991\endcsname=\subjclassname
\expandafter\let\csname subjclassname@2000\endcsname=\subjclassname
\subjclass{%
}


\begin{abstract}
  We study the structure of the exteriors of gropes and Whitney towers
  in dimension~4, focusing on their fundamental groups.  In particular
  we introduce a notion of unknottedness of gropes and Whitney towers
  in the 4-sphere.  We prove that various modifications of gropes and
  Whitney towers preserve the unknottedness and do not enlarge the
  fundamental group.  We exhibit handlebody structures of the
  exteriors of gropes and Whitney towers constructed by earlier
  methods of Cochran, Teichner, Horn, and the first author, and use
  them to construct examples of unknotted gropes and Whitney towers.
  As an application, we introduce geometric bi-filtrations of knots
  which approximate the double sliceness in terms of unknotted gropes
  and Whitney towers.  We prove that the bi-filtrations do not
  stabilize at any stage.
\end{abstract}

\maketitle


\section{Introduction}

In 4-dimensional topology, the disk embedding problem is of primary
importance.  Since Freedman's 1983 ICM proceedings
paper~\cite{Freedman:1984-1}, \emph{capped gropes} have been used as
the most essential ingredient for disk embedding, especially for the
non-simply connected case.  The book of Freedman and
Quinn~\cite{Freedman-Quinn:1990-1} presents beautiful grope-based
treatments of foundational results in dimension~4.  It is also natural
to consider a closely related notion of \emph{Whitney towers} together
with gropes.  Cochran, Orr, and Teichner introduced a framework of the
study of knot concordance, which can be viewed as the local case of
the general disk embedding problem, in terms of symmetric gropes and
Whitney towers~\cite{Cochran-Orr-Teichner:1999-1}.  For links, Conant,
Schneiderman, and Teichner developed a theory of asymmetric grope
concordance and Whitney tower concordance, as summarized
in~\cite{Conant-Schneiderman-Teichner:2011-1}.

To obtain flat embedded disks from capped gropes, the present
technology (e.g.\ see \cite{Freedman-Quinn:1990-1,
  Freedman-Teichner:1995-1, Krushkal-Quinn:2000-1, Cha-Powell:2014-1})
requires that the involved \emph{fundamental group} is, very roughly
speaking, ``not too large.''  Such a group is often called
\emph{good}.  The key question, which is still left open, is whether
all groups are good.  The best known result is that subexponential
groups and their iterated extensions and direct limits are good, due
to Freedman and Teichner~\cite{Freedman-Teichner:1995-1} and Krushkal
and Quinn~\cite{Krushkal-Quinn:2000-1}.

\subsubsection*{Unknotted gropes and Whitney towers}

In this paper, motivated by the above, we begin to study the structure
of the exteriors of gropes and Whitney towers in 4-manifolds, focusing
on their fundmantal groups.  We also present an application to double
slicing of knots.

The following definition formulates the case of the smallest possible
fundamental group.

\begin{definition}
  A sphere-like capped grope, or Whitney tower, in the 4-sphere is
  \emph{$\pi_1$-unknotted} if its complement has infinite cyclic
  fundamental group.
\end{definition}

Precise definitions of sphere-like capped gropes and Whitney towers
are given in
Section~\ref{subsection:preliminaries-whitney-towers-gropes}.  (For
those who are familiar with the notion of properly immersed capped
gropes in \cite{Freedman-Quinn:1990-1}, we remark that we consider a more general class of capped gropes in 4-manifolds, for which a cap is allowed to intersect body surfaces as well as caps, similarly to,
e.g.~\cite{Conant-Schneiderman-Teichner:2012-3}.)  An infinite cyclic
fundamental group is the smallest possible in the sense that the
fundamental group of every sphere-like capped grope/Whitney tower
exterior in $S^4$ has an infinite cyclic quotient (see
Remark~\ref{remark:motivation-to-unknotted-definition}).

The following observation justifies our terminology: Freedman showed
that a flat 2-sphere $S$ embedded in $S^4$ is unknotted in the
classical sense if and only if $\pi_1(S^4\setminus S)$ is infinite
cyclic~\cite{Freedman:1984-1}.  Thus an embedeed 2-sphere is unknotted
if and only if it is $\pi_1$-unknotted as a capped grope or Whitney
tower.  Furthermore, if one performs finger moves on a 2-sphere
embedded in $S^4$ and obtains a Whitney tower consisting of the
resulting immersed 2-sphere together with the additional Whitney disks
introduced by the finger moves, then the Whitney tower is
$\pi_1$-unknotted if and only if the original embedded sphere is
unknotted (see Lemma~\ref{lemma:finger-move-fundamental-group}).

We note that if a sphere-like properly immersed capped grope $G$ in
$S^4$ has height at least 1.5 and is unknotted, then for any 2-disk
$D$ contained in the base surface, the disk embedding result
in~\cite[Theorem~3.4]{Cha-Powell:2014-1} tells us that the subgrope
$\overline{G\setminus D}$ can be replaced with a flat embedded disk.
Thus the grope $G$ can be modified, relative to~$D$, to a flat
embedded 2-sphere.

\subsubsection*{Modifications of Whitney towers and gropes}

It turns out that various fundamental operations on capped gropes and
Whitney towers do not enlarge the fundamental group of the exterior,
and consequently preserve the $\pi_1$-unknottedness.  For instance, in
Sections~\ref{section:Unknotted Whitney towers and gropes}
and~\ref{section:schneiderman-transformation}, we show that it is the
case for the following operations:

\begin{enumerate}
\item Regular homotopy of Whitney towers
  (Definition~\ref{definition:whitney-tower-regular-isotopy} and
  Lemma~\ref{lemma:finger-move-fundamental-group})
\item Taking a subtower of a Whitney tower by removing Whitney disks
  (Proposition~\ref{proposition:unknottedness-of-subtower} and
  Corollary~\ref{corollary:unknotted-subtower-of-lower-heights})
\item Whitney disk splitting (Proposition~\ref{proposition:tower-splitting-unknottedness})
\item Symmetric and asymmetric contraction of a capped grope
  (Proposition~\ref{proposition:unknottedness-of-contraction})
\item Pushing an intersection down in a capped grope (Proposition~\ref{proposition:pushing-intersection-down-unknottedness})
\item Krushkal's grope splitting
  (Proposition~\ref{proposition:grope-splitting-unknotedness})
\end{enumerate}

In~\cite{Schneiderman:2006-1}, Schneiderman presented fundamental
procedures which transform a capped grope to a Whitney tower in an
arbitrary 4-manifold, and vice versa.  Concerning this, we prove the
following result on the exteriors:

\begin{theoremalpha}
  \label{theorem:schneiderman-transformation-intro}
  Schneiderman's transformation~\cite{Schneiderman:2006-1} from a
  Whitney tower to a capped grope does not enlarge the fundamental
  group of the exterior, and for sphere-like ones in the 4-sphere, it
  preserves the $\pi_1$-unknottedness.  The same holds for
  Schneiderman's transformation from a capped grope to a Whitney
  tower.
\end{theoremalpha}

In Section~\ref{section:schneiderman-transformation}, we state and
prove a refined version of
Theorem~\ref{theorem:schneiderman-transformation-intro} as
Theorem~\ref{theorem:grope-whitney-tower-transformation}.

\subsubsection*{Analysis of grope constructions}

Cochran and Teichner~\cite{Cochran-Teichner:2003-1},
Horn~\cite{Horn:2010-1}, and Cha~\cite{Cha:2012-1} developed methods
to produce annulus-like symmetric gropes in $S^3\times I$ cobounded by
knots in $S^3\times 0$ and $S^3\times 1$, which can be viewed as
approximations of ordinary concordance.  Together with obstructions
from $L^2$-signatures, these constructions have been used as key
ingredients in revealing the rich structure of the grope and Whitney
tower theory in the context of concordance (e.g.\ see
\cite{Horn:2011-1,Cha-Powell:2013-1,Jang:2017-1}).  Briefly, in these
methods, the two knots in $S^3\times 0$ and $S^3\times 1$ are related
by (iterated) satellite constructions, and annular gropes cobounded by
them are constructed by first finding simpler gropes and surfaces in
3-space for companions and patterns of the satellite constructions,
and then stacking them in 4-space.  A precise formulation of this
procedure is described in Sections~\ref{subsection:product-of-gropes}
and~\ref{subsection:handle-structure-grope-exterior} (see
Definitions~\ref{definition:satellite-capped-grope}--\ref{definition:composition-satellite-grope-and-grope-concordance}
and~\ref{definition:capped-gropes-in-3D}).

We present a handle decomposition of the exterior of a grope in
$S^3\times I$ obtained by these methods.  In particular, we show the
following:

\begin{theoremalpha}
  \label{theorem:composition-grope-handle-decomposition-intro}
  Suppose $G$ is an annular capped grope in $S^3\times I$ cobounded by
  $K\subset S^3\times 0$ and $K'\subset S^3\times 1$, which is
  constructed by the methods of Cochran and
  Teichner~\cite{Cochran-Teichner:2003-1}, Horn~\cite{Horn:2010-1},
  and Cha~\cite{Cha:2012-1} described in
  Sections~\ref{subsection:product-of-gropes}
  and~\ref{subsection:handle-structure-grope-exterior}.  Let $E_G$ and
  $E_K$ be the exteriors of $G$ and~$K$.  Then $E_G$ has a relative
  handle decomposition with 2-handles only:
  $E_G \cong (E_K\times I) \cup (\text{$2$-handles})$.
\end{theoremalpha}

We state and prove a full version of
Theorem~\ref{theorem:composition-grope-handle-decomposition-intro} in
Section~\ref{subsection:handle-structure-grope-exterior} (see in
particular
Theorem~\ref{theorem:composition-grope-handle-decomposition}).

One can apply Schneiderman's method~\cite{Schneiderman:2006-1} to
convert a capped grope constructed by the methods
of~\cite{Cochran-Teichner:2003-1, Horn:2010-1,Cha:2012-1} to a Whitney
tower of the same height. Combining
Theorem~\ref{theorem:composition-grope-handle-decomposition-intro}
with our analysis of Schneiderman's method (which is used to prove
Theorem~\ref{theorem:schneiderman-transformation-intro}), it turns out
that the exterior of the resulting Whitney tower has a similar handle
decomposition with 2-handles only.

\subsubsection*{Application to knot double slicing}

We use the notion of $\pi_1$-unknottedness and the above handle
structure results to study double slicing of knots in terms of gropes
and Whitney towers.

Recall that a knot $K$ in $S^3$ is \emph{doubly slice} if there is an
unknotted flat 2-sphere in $S^4$ which intersects the standard
$S^3\subset S^4$ at~$K$.  That is, $K$ is a slice in $S^3$ of an
unknotted 2-sphere in~$S^4$.  Inspired from this, we define a knot $K$
in $S^3$ to be \emph{height $(m,n)$ grope slice} if it is a slice of a
$\pi_1$-unknotted sphere-like capped grope $G$ in $S^4$ such that the
intersections of $G$ with the upper and lower hemispheres of $S^4$ have
height $m$ and $n$ respectively.  Here $m$ and $n$ are nonnegative
half integers.  See Definition~\ref{definition:Whitney tower
  bi-filtration}.  We also define a \emph{height $(m,n)$ Whitney slice
  knot} similarly by replacing the grope $G$ above by a Whitney tower.

We denote by $\cG_{m,n}$ and $\cW_{m,n}$ the collection of height
$(m,n)$ grope slice knots and height $(m,n)$ Whitney slice knots,
respectively.  Each of $\cG_{m,n}$ and $\cW_{m,n}$ is a submonoid of
the monoid of knots under connected sum, and $\{\cG_{m,n}\}$ and
$\{\cW_{m,n}\}$ are decending (non-increasing) bi-filtrations.  It
turns out that a height $(m,n)$ grope slice knot is height $(m,n)$
Whitney slice, that is, $\cG_{m,n}\subset \cW_{m,n}$ (see
Theorem~\ref{theorem:grope-bifiltration-contained-in-whitney-bifiltration}).

A doubly slice knot is height $(m,n)$ grope slice and height $(m,n)$
Whitney slice for all $m, n$, and therefore lies in the intersection
of all $\cG_{m,n}$ and~$\cW_{m,n}$.

It turns out that previously known obstructions to knots being doubly
slice are obstructions to lying in low height terms of our
bi-filtrations.  In Section~\ref{section:bi-filtrations and classical
  obstructions}, we show that the double sliceness obstructions by
Gilmer-Livingston, Friedl, and Livingston-Meier
\cite{Gilmer-Livingston:1983-1,Friedl:2003-4,Livingston-Meier:2015-1}
vanish for knots in $\cG_{4,4}$ or~$\cW_{4,4}$; it is proven in
stronger forms in
Propositions~\ref{proposition:doubly-2-solvability-and-GL-obstructions},
\ref{proposition:doubly-1.5-solvability-and-friedl-obstruction}, and
\ref{proposition:doubly-1.5-solvability-and-livingston-meier-obstruction}.

Our main result on the geometric bi-filtrations $\{\cG_{m,n}\}$ and
$\{\cW_{m,n}\}$ is the following.

\begin{theoremalpha}\label{thmalpha:nontriviality-of-bi-filtrations}
  Let $m,n\ge 3$ be integers.
\begin{enumerate}
\item There exists a family of slice knots $\{J^i\}_{i=1,2,\ldots}$ in
  $\cG_{m,m}$ whose arbitrary nontrivial linear combination
  $\# a_i J^i$ ($a_i\in\Z)$ is not in~$\cW_{m.5,m.5}$.
\item There exists a family of knots $\{J^i_{m,n}\}_{i=1,2,\ldots}$ in
  $\cG_{m,n}$ whose arbitrary nontrivial linear combination
  $\# a_i J^i_{m,n}$ ($a_i\in\Z)$ is not in
  $\cW_{m.5,n}\cup \cW_{m,n.5}$.
\end{enumerate}
\end{theoremalpha}

Theorem~\ref{thmalpha:nontriviality-of-bi-filtrations} is proven in a
stronger form in Theorem~\ref{theorem:infinite-rank}.  Since
$\cG_{m,n}\subset \cW_{m,n}$, it follows from
Theorem~\ref{thmalpha:nontriviality-of-bi-filtrations} that both
bi-filtrations $\{\cG_{m,n}\}$ and $\{\cW_{m,n}\}$ are highly
nontrivial.

In the proof of
Theorem~\ref{thmalpha:nontriviality-of-bi-filtrations}, we give the
knots $J^i$ and $J^i_{m,n}$ using iterated satellite constructions.
The main proof consists of two parts: the existence of a
$\pi_1$-unknotted capped grope of the given height, and the
non-existence of a $\pi_1$-unknotted Whitney tower of larger height.

For the existence, briefly speaking, we employ the methods
of~\cite{Cochran-Teichner:2003-1, Horn:2010-1,Cha:2012-1} to produce
capped gropes which slice our knots.  We prove that they are
$\pi_1$-unknotted by using handle structures given by
Theorem~\ref{theorem:composition-grope-handle-decomposition-intro}\@.
Details can be found in
Section~\ref{subsection:construction-of-examples}.

For the non-existence, we first relate our geometric filtrations to
\emph{the solvable bi-filtration}, which was defined by the second
author~\cite{Kim:2006-1} as a double slicing analog of
Cochran-Orr-Teichner's solvable
filtration~\cite{Cochran-Orr-Teichner:1999-1}.  Briefly, a knot $K$ is
\emph{$(m,n)$-solvable} if there are 4-manifolds $U$ and $V$ bounded
by the zero surgery manifold $M(K)$, each of which approximates a
slice disk exterior in terms of a certain duality property over the
group ring $\Z[\pi_1(-)/\pi_1(-)^{(n)}]$, where $\pi_1(-)^{(n)}$
designates the $n$th derived subgroup, such that the union
$U\cup_{M(K)}V$ has infinite cyclic fundamental group.  See
Definition~\ref{definition:(m,n)-solution} for a precise description.
In Section~\ref{subsection: (m,n)-solvable knot}, we prove that a knot
is $(m,n)$-solvable whenever it is height $(m+2,n+2)$ Whitney slice,
that is, $\cW_{m+2,n+2} \subset \F_{m,n}$ (see
Theorem~\ref{theorem:whitney-tower-to-solution-fundamental-group}).
Then, in
Sections~\ref{subsection:amenable-signature}--\ref{subsection:proofs-of-lemma-and-proposition},
we use the \emph{amenable signature theorem}
of~\cite{Cha-Orr:2009-1,Cha:2010-1} and combine it with the ideas
of~\cite{Kim:2006-1} to extract obstructions to being $(m,n)$-solvable
from the von Neumann-Cheeger-Gromov $\rhot$-invariants of the zero
surgery manifolds.

A corollary of (the above outlined proof of)
Theorem~\ref{thmalpha:nontriviality-of-bi-filtrations} is that its
analog holds for the solvable bi-filtration $\{\F_{m,n}\}$ (in place
of both $\{\cG_{m,n}\}$ and~$\{\cW_{m,n}\}$) as well.  This
consequence, which provides infinitely many linearly independent knots
in each stage of the solvable bi-filtration modulo the next stages,
generalizes the main result of~\cite{Kim:2006-1} which gives
nontrivial knots in each stage of the solvable bi-filtration modulo
the next stages.

\subsection*{Acknowledgements}

The authors thank Patrick Orson for his helpful comments on the first
draft version of this paper.  The first named author was supported by
NRF grant 2011-0030044.  The second named author was supported by
NRF grant 2011-0030044 and NRF grant 2015R1D1A1A01056634.

\section{Unknotted Whitney towers and gropes}
\label{section:Unknotted Whitney towers and gropes}

In this section we define the notion of unknotted Whitney towers and
gropes in terms of the fundamental group, and discuss some basic
properties.

\subsection{Preliminaries: Whitney towers and gropes}
\label{subsection:preliminaries-whitney-towers-gropes}

We begin by recalling the definitions of Whitney towers and gropes.
Readers familiar with them may skip to
Section~\ref{subsection:unknotted-towers}, after reading our
conventions given after Definition~\ref{definition:grope}.

In the following definitions, we will assume that the readers are
familiar with the notion of (framed) immersions of surfaces and
Whitney disks; for instance see \cite[\S 1.2, \S
1.4]{Freedman-Quinn:1990-1}.

In this paper, we will define and use \emph{framed} gropes and Whitney
towers only.  In particular a Whitney disk will always be framed as
follows: for a Whitney disk $D$ pairing two intersections of sheets
$A$ and $B$, the restriction of the unique framing of $D$ on
$\partial D$ is equal to the framing determined by the tangential
direction of $A$ on $A\cap \partial D$ and the common normal direction
of $B$ and $D$ on $B\cap \partial D$.

In this paper, surfaces are always orientable.

\begin{definition}[Whitney tower]
  \label{definition:whitney-tower}
  Suppose $S$ is a surface.  An \emph{$S$-like Whitney tower} in a
  4-manifold $W$ is a 2-complex defined inductively as follows.  A
  properly immersed surface
  $(S,\partial S) \looparrowright (W,\partial W)$ is an $S$-like
  Whitney tower in $W$.  Suppose $T$ is an $S$-like Whitney tower and
  $D$ is a framed immersed Whitney disk pairing two intersections of
  opposite signs between two sheets in $T$, where the interior of $D$
  is allowed to transversely intersect the interior of surfaces/disks of $T$
  but is disjoint from the boundary of any surface in~$T$.  Then $T$
  together with $D$ is an $S$-like Whitney tower.  We say that $S$ is
  the \emph{base surface} of~$T$, and $S$ \emph{supports}~$T$.  The
  \emph{boundary} of $T$ is that of the base surface of~$T$.
\end{definition}

We will mostly use sphere-like or disk-like Whitney towers.

We often consider \emph{symmetric Whitney towers}, which are defined
to be a Whitney tower with a well-defined height in the following
sense.  This was introduced by Cochran, Orr and
Teichner~\cite{Cochran-Orr-Teichner:1999-1}.  We remark that the
notion of \emph{order} is defined for any Whitney towers as in, for
instance, \cite{Conant-Schneiderman-Teichner:2012-2}.

\begin{definition}[Height of a Whitney tower]
  \label{definition:height-of-Whitney-tower}
  Suppose $T$ is a Whitney tower.  For surfaces/disks and
  intersections in $T$, the \emph{height} is defined inductively as
  follows.  The height of the base surface of $T$ is one.  If $p$ is
  an intersection of two surfaces/disks of the same height $k$ in $T$, then
  we say $p$ has height~$k$.  If a Whitney disk in $T$ pairs up two
  intersections of the same height $k$, we say the disk has
  height~$k+1$.  (Note that the height may not be defined for some
  intersections and for some Whitney disks.)

  Let $n$ be a positive integer.  We say that $T$ is a \emph{Whitney
    tower of height~$n$} if the following hold: (i) for all
  surfaces/disks in $T$ and their intersections, the height is defined
  and not greater than~$n$, and (ii) all intersections of height $<n$
  are paired up by Whitney disks in~$T$.
  
  We say that $T$ is a \emph{Whitney tower of height $n.5$} if the
  following hold: (i) for all surfaces/disks in $T$, the height is
  defined and not greater than~$n+1$,
  (ii) all intersections between surfaces/disks of height $\le n$ have
  well-defined height, and are paired up by Whitney disks in~$T$, and
  (iii) each height $n+1$ Whitney disk does not meet surfaces/disks of
  height $<n$ (but allowed to meet surfaces/disks of height $n$
  and~$n+1$).
\end{definition}

An immersed surface can be viewed as a Whitney tower of height one.

\begin{definition}[Grope]
  \label{definition:grope}
  A \emph{capped surface} is a surface $\Sigma$ together with disks
  attached along $2g$ standard symplectic basis curves on~$\Sigma$,
  where $g$ is the genus of~$\Sigma$~\cite[\S
  2.1]{Freedman-Quinn:1990-1}.  The disks are called \emph{caps}.
  Suppose $S$ is a surface.  A \emph{model $S$-like capped grope} is a
  2-complex defined inductively as follows.  A 2-complex obtained from
  $S$ by replacing finitely many disjointly embedded disk in $S$ by
  capped surfaces is a model $S$-like capped grope.  Its \emph{base
    surface} is defined to be the surface we obtain by modifying~$S$.
  If $G$ is a model $S$-like capped grope, then a 2-complex obtained
  by replacing a cap with a capped surface is a model $S$-like capped
  grope.  Its caps are the unmodified caps of $G$ together with the
  caps of the attached capped surface.  Its base surface is defined to
  be that of~$G$.  The \emph{body} of a model capped grope $G$ is
  defined to be $G$ with all caps removed.  The boundary of $G$ is the
  boundary of the base surface.

  Note that a model capped grope $G$ admits a standard embedding
  in~$\R^3$.  Composing it with $\R^3 \hookrightarrow \R^4$, $G$
  embeds in $\R^4$.  Take a regular neighborhood of $G\subset \R^4$,
  and then possibly introduce finitely many plumbings between
  caps/surfaces.  An embedding of the result into a 4-manifold $W$ is
  called an \emph{immersed capped grope} in~$W$.
\end{definition}

In this paper we assume that each intersection in an immersed capped
grope always involves a cap.  That is, there is no intersection
between two sheets of body surfaces, while caps are allowed to meet
caps and body surfaces.

We remark that this differs from the notion of a \emph{properly}
immersed capped grope~\cite{Freedman-Quinn:1990-1}, which is defined
to be an immersed capped grope whose intersections are always between
caps.  In our case, every intersection can be changed to intersections
of a cap and the base surface, by ``pushing intersections down'' as
described in~\cite[Section~2.5]{Freedman-Quinn:1990-1} (see also
Section~\ref{subsection:modification-capped-gropes-unknottedness} of
this paper).

\begin{definition}[Height of a capped grope]
  \label{definition:height-of-grope}
  Suppose $G$ is a (model) capped grope.  The \emph{height} of
  surfaces/caps in $G$ is defined inductively as follows.  The height
  of the base surface of $G$ is one.  The height of a surface/cap in
  $G$ is $k+1$ if its boundary is attached to a surface of height~$k$.
  If all caps of $G$ have height $n+1$, then we say that $G$ is a
  \emph{capped grope of height $n$}.  If, for each dual pair of curves
  on the base surface, a height $n$ capped grope is attached along one
  of them and a height $n-1$ capped grope along the other, then we say
  that $G$ is a \emph{capped grope of height~$n.5$}.
\end{definition}

\subsection{Unknotted Whitney towers}
\label{subsection:unknotted-towers}

\begin{definition}
  \label{definition:unknotted-whitney-tower}
  A (union-of-spheres)-like Whitney tower $T$ in $S^4$ is
  \emph{$\pi_1$-unknotted} if $\pi_1(S^4 \setminus T)$ is a free
  group.  (If it is the case, the rank of the free group is
  automatically equal to the number of the base spheres.)  In
  particular, a sphere-like Whitney tower $T$ in $S^4$ is
  \emph{$\pi_1$-unknotted} if $\pi_1(S^4 \setminus T)\cong\Z$.
\end{definition}

\begin{remark}
  \label{remark:motivation-to-unknotted-definition}
  \begin{enumerate}
  \item For every sphere-like Whitney tower $T$ in $S^4$, the
    abelianization $H_1(S^4\setminus~T)$ of $\pi_1(S^4\setminus T)$ is
    equal to~$\Z$ by Alexander duality.  In this sense, for the
    sphere-like case, $T$ is $\pi_1$-unknotted if and only if the
    complement has the ``smallest'' fundamental group.
  \item Definition~\ref{definition:unknotted-whitney-tower} is a
    Whitney tower generalization of the standard notion of an
    unknotted 2-sphere embedded in the 4-sphere, in light of the
    following well-known result~\cite[\S
    11.7A]{Freedman-Quinn:1990-1}: \emph{a locally flat 2-sphere $S$
      embedded in $S^4$ is unknotted if and only if $\pi_1(S^4
      \setminus S) \cong \Z$}.  In other words, $S$ is unknotted as a
    sphere if and only if $S$ is $\pi_1$-unknotted as a Whitney tower.
  \end{enumerate}
\end{remark}

Recall that a \emph{finger move} \cite[\S 1.5]{Freedman-Quinn:1990-1}
for surfaces introduces new intersections paired by an embedded
Whitney disk; the finger move is reversed by a Whitney move across the
disk.  For a Whitney tower, we will consider a finger move which
introduces new intersections of surfaces/disks of the tower, performed
along an arc whose interior is disjoint from the tower.  As our
convention, the newly introduced Whitney disk is added to the
resulting Whitney tower.  Similarly, when we apply a Whitney move
across an embedded Whitney disk in a Whitney tower, we remove the
Whitney disk we used from the tower.

\begin{definition}
  \label{definition:whitney-tower-regular-isotopy}
  A \emph{regular homotopy} for a Whitney tower is a finite sequence
  of ambient isotopy, finger moves, and Whitney moves (across an
  embedded Whitney disk).
\end{definition}

For brevity, we use the following terminology.  If $A$, $B$, and $C$
are subspaces of $X$ satisfying $A\cup B\subset C$ and
$f\colon\pi_1(X\setminus A)\to \pi_1(X\setminus B)$ is a homomorphism
such that $fi_* = j_*$ for the inclusions
$i\colon X\setminus C\to X\setminus A$ and
$j\colon X\setminus C \to X\setminus B$, then we say that $f$ \emph{is
  supported by~$C$}, or $f$ \emph{extends the identity on
  $X\setminus C$}.

For a subcomplex $K$ of a manifold $X$, we denote its regular
neighborhood by $\nu(K)$, and the exterior $\overline{X\setminus \nu(K)}$ by
$E_K$.  When $X$ is a topological
manifold, $\nu(K)$ and $E_K$ are defined if a neighborhood of $K$
admits a PL/smooth structure with respect to which $K$ is PL/smooth.
In this paper it is always the case.  Since we will use the following
standard fact repeatedly, we state it as a lemma.

\begin{lemma}
  \label{lemma:subcomplex-exterior-handle-decomposition}
  Suppose $(K,L)$ is a subcomplex pair embedded in an $n$-manifold~$X$
  with $L = K\cap \partial X$.  If $K=L\cup e_1 \cup \cdots \cup e_r$
  with $e_i$ a $p_i$-cell, then
  $X = E_K \cup h_1 \cup \cdots \cup h_r$ with $h_i$ an
  $(n-p_i)$-handle.
\end{lemma}

\begin{proof}
  Let $K_0=L$ and $K_i = L\cup e_1 \cup \cdots \cup e_i$.  Then
  $E_{K_r} = E_K$ and $E_{K_{i-1}}$ is obtained from $E_{K_i}$ by
  filling in it with $\nu(e_i) = D^{p_i} \times D^{n-p_i}$.  Since
  $\nu(e_i) \cap E_{K_i} = D^{p_i} \times \partial D^{n-p_i}$,
  $\nu(e_i)$ is an $(n-p_i)$-handle.
\end{proof}

\begin{lemma}[Regular homotopy preserves unknottedness]
  \label{lemma:finger-move-fundamental-group}
  Suppose $T'$ is a Whitney tower obtained from another Whitney tower
  $T$ by a finger move in a 4-manifold~$W$.  Then the exterior $E_T$
  is obtained by attaching a $3$-handle to the exterior~$E_{T'}$, and
  thus there is an isomorphism
  $\pi_1(W\setminus T) \cong \pi_1(W\setminus T')$ supported by a
  regular neighborhood of the trace of the finger move.  Consequently,
  regular homotopy preserves the $\pi_1$-unknotedness of Whitney
  towers in~$S^4$.
\end{lemma}

\begin{proof}
  Consider a finger move which converts a tower $T$ to another tower
  $T'$.  Suppose this is performed along an arc $\gamma$ joining two
  interior points of surfaces/disks of~$T$.  See
  Figure~\ref{figure:finger-move}.  Then
  $E_{T'} \cong E_{T\cup\gamma} = \overline{E_T \setminus
    \nu(\gamma)}$, that is, $E_{T'}$ is homeomorphic to the exterior
  of $\gamma$ in~$E_T$.  Since $\gamma$ is a 1-cell,
  $E_T \cong E_{T'} \cup(\text{$3$-handle})$ by
  Lemma~\ref{lemma:subcomplex-exterior-handle-decomposition}.  The
  other conclusions follow from this.
\end{proof}

\begin{figure}[H]
  \labellist
  \pinlabel{$\gamma$} at 57 35
  \pinlabel{$T$} at 20 30
  \pinlabel{$T'$} at 205 30
  \endlabellist
  \centering
  \includegraphics{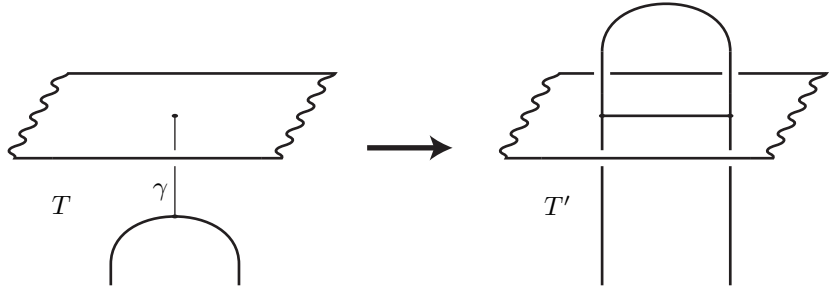}
  \caption{A finger move performed along~$\gamma$.}
  \label{figure:finger-move}
\end{figure}

\begin{proposition}[Unknottedness of subtowers]
  \label{proposition:unknottedness-of-subtower}
  Suppose $T$ is a Whitney tower in a 4-manifold $W$, and let $T'$ be
  a Whitney tower which is obtained from $T$ by removing some Whitney
  disks.  Then $E_{T'} \cong E_T \cup ($handles of index $\ge 2)$, and
  consequently the inclusion induces an epimorphism of
  $\pi_1(W\setminus T)$ onto $\pi_1(W\setminus T' )$.  In addition, if
  $W=S^4$ and $T$ is sphere-like and $\pi_1$-unknotted, then $T'$ is
  $\pi_1$-unknotted.
\end{proposition}

\begin{proof}
  It suffices to consider the case that $T'$ is obtained by deleting a
  Whitney disk~$D$.  Note that $D$ is immersed and may intersect other
  surfaces in general.  View $E_{T'}\cap D$ as a 2-complex embedded
  in~$E_{T'}$.  Since $E_T$ is equal to the exterior of $E_{T'}\cap D$
  in $E_{T'}$, it follows that $E_{T'} \cong E_T \cup ($handles of
  index $\ge 2)$ by
  Lemma~\ref{lemma:subcomplex-exterior-handle-decomposition}.
  Therefore the inclusion-induced homomorphism
  $\pi_1(W\setminus T) \to \pi_1(W\setminus T')$ is surjective.  Finally, if $W=S^4$ and $T$
  is a sphere-like Whitney tower which is $\pi_1$-unknotted, then
  since $\pi_1(S^4\setminus T)\cong \Z\cong H_1(S^4\setminus T')$, it follows that
  $\pi_1(S^4\setminus T') \cong \Z$.
\end{proof}

An immediate corollary is that we can lower the height of a Whitney
tower without losing $\pi_1$-unknottedness.

\begin{corollary}
  \label{corollary:unknotted-subtower-of-lower-heights}
  Suppose $S$ is an immersed 2-sphere in $S^4$ supporting a
  $\pi_1$-unknotted Whitney tower of height $h\in \tfrac12\Z$.  Then
  for any $h'<h$, $S$ supports a $\pi_1$-unknotted Whitney tower of
  height~$h'$.
\end{corollary}

\begin{proof}
  Let $T$ be the given Whitney tower of height~$h$.  When
  $h'=n\in \Z$, remove Whitney disks of height $>n$.  By
  Proposition~\ref{proposition:unknottedness-of-subtower}, the
  resulting tower is $\pi_1$-unknotted, and has height $n$. When
  $h'=n.5$, remove Whitney disks of height $>n+1$.  This produces a
  $\pi_1$-unknotted Whitney tower of height $n+1$, and by definition
  it is also a tower of height~$n.5$.
\end{proof}

\subsection{Unknotted capped gropes}
\label{subsection:unknotted-gropes}

As in the Whitney tower case, we define the notion of unknottedness of
a capped grope in terms of the fundamental group:

\begin{definition}
  \label{definition:unknotted-capped-grope}
  A (union-of-spheres)-like capped grope $G$ immersed in $S^4$ is
  \emph{$\pi_1$-unknotted} if $\pi_1(S^4\setminus G)$ is a free group.
  In particular, a sphere-like capped grope $G$ in $S^4$ is
  \emph{$\pi_1$-unknotted} if $\pi_1(S^4\setminus G) \cong \Z$.
\end{definition}

Analogs of Proposition~\ref{proposition:unknottedness-of-subtower} and
its Corollary~\ref{corollary:unknotted-subtower-of-lower-heights} for
capped gropes can be formulated in terms of contraction.  Suppose
$D_1$ and $D_2$ are dual caps of a capped grope $G$, that is, they are
attached to the same body surface, say $\Sigma$, along curves
intersecting at a single point.  Following~\cite[\S
2.3]{Freedman-Quinn:1990-1}, take two parallel copies of $D_1$ and
$D_2$ and attach them to a square neighborhood of the point
$\partial D_1\cap \partial D_2$ in $\Sigma$ to obtain a disk.  Cut
$\Sigma$ along $\partial D_1$ and $\partial D_2$, and then attach this
disk.  This gives a surface with genus one less than that of~$\Sigma$.
See the left hand side of Figure~\ref{figure:contraction}.  We call
this \emph{symmetric contraction}.  Applying this to all the dual
pairs of caps attached to a top stage surface, we can replace the top
stage surface together with its caps by a disk, which may be viewed as
a new cap.

There is an ``asymmetric'' version of the above.
Following~\cite{Krushkal-Quinn:2000-1}, forget one of the caps, say
$D_2$, cut $\Sigma$ along $\partial D_1$, and attach two parallel
copies of $D_1$ to obtain a new surface with genus one less than that
of~$\Sigma$.  See the right hand side of
Figure~\ref{figure:contraction}.  We call this operation
\emph{asymmetric contraction}.  Similarly to the above, a top stage
surface of a capped grope can be changed to a cap by applying this
repeatedly.

\begin{figure}[H]
  \labellist
  \footnotesize
  \pinlabel{$D_1$} at 186 95
  \pinlabel{$D_2$} at 186 62
  \pinlabel{symmetric} at 118 45
  \pinlabel{asymmetric} at 254 45
  \endlabellist
  \includegraphics{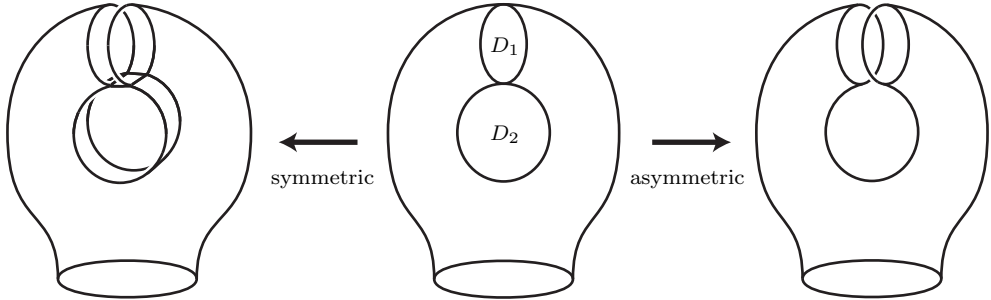}
  \caption{Symmetric and asymmetric contraction.}
  \label{figure:contraction}
\end{figure}

\begin{proposition}[Unknottedness of contraction]
  \label{proposition:unknottedness-of-contraction}
  Suppose $G'$ is a capped grope obtained from another capped grope
  $G$ in a 4-manifold $W$ by either asymmetric or symmetric
  contraction.  Then $E_{G'} \cong E_G \cup (\text{$2$-handles})$.
  Consequently, there is an epimorphism of $\pi_1(W\setminus G)$ onto
  $\pi_1(W\setminus G')$ supported by a regular neighborhood of~$G$.
  In addition, if $W=S^4$ and $G$ is sphere-like and
  $\pi_1$-unknotted, then $G'$ is $\pi_1$-unknotted.
\end{proposition}

\begin{proof}
  We assert that $E_G$ is homeomorphic to the exterior $E_L$ of a
  2-complex $L$ obtained by attaching $k$ 2-cells to $G'$, where $k$
  is the number of intersections of the caps $D_1\cup D_2$ with other
  surfaces/caps.  All the conclusions follow from this, since
  $E_{G'} \cong E_L \cup (\text{2-handles})$ by
  Lemma~\ref{lemma:subcomplex-exterior-handle-decomposition}.

  To prove the assertion for the symmetric contraction case, attach to
  $G'$ $k$ 2-cells which are shown in
  Figure~\ref{figure:contraction-exterior} as hatched rectangles.
  Denote the resulting complex by~$L$.  Briefly speaking, the picture
  tells us that $L$ has the same regular neighborhood as that of~$G$,
  and consequently the exterior of $L$ in $W$ is
  homeomorphic to that of~$G$.

  \begin{figure}[H]
    \includegraphics{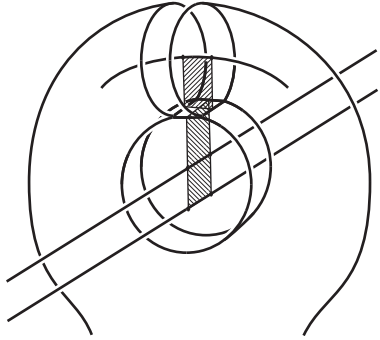}
    \caption{A 2-complex $L$ with the same exterior as the initial
      capped grope~$G$.  This is the case of $k=3$, that is, there are
      three arcs representing other surfaces which intersect the caps.
      The three hatched rectangles are additional 2-cells attached to
      the contraction.}
    \label{figure:contraction-exterior}
  \end{figure} 

  In what follows we will present a formal approach for the last
  sentence, since it will also be useful in later sections.  It is
  best described using the language of PL topology since we need to
  deal with complexes which are not manifolds and their regular
  neighborhoods.

  \begin{definition}
    \label{definition:cellular-expansion}
    Suppose $K$ is a subcomplex of a simplicial complex~$X$.  If
    $\Delta$ is a subcomplex of $X$ such that $(\Delta,\Delta\cap K)
    \cong (D^n,$ an embedded $(n-1)$-ball in $\partial D^n)$ for
    some~$n\ge 1$, then we say that $K\cup \Delta$ is obtained from
    $K$ by \emph{elementary cellular expansion} in~$X$ (and $K$ is
    obtained from $K\cup\Delta$ by \emph{elementary cellular
      collapse}).  We say that $K$ \emph{expands cellularly} to $K'$
    in $X$ (and $K'$ \emph{collapses cellularly} to $K$) if there is a
    sequence of elementary cellular expansions in $X$ transforming $K$
    to~$K'$.
  \end{definition}

  When the ambient complex $X$ is clearly understood, we often omit
  ``in~$X$.''

  The following is a standard fact.  For instance see
  \cite[Chapter~3]{Rourke-Sanderson:1972-1}.

  \begin{lemma}
    \label{lemma:cellular-expansion-exterior}
    If $K$ expands cellularly to $K'$ in a manifold $X$, then the
    regular neighborhoods $\nu(K)$ and $\nu(K')$ in $X$ are isotopic
    in~$X$.  Consequently the exteriors $E_K$ and $E_{K'}$ are
    homeomorphic.
  \end{lemma}

  We remark that Lemma~\ref{lemma:cellular-expansion-exterior} applies
  to subcomplexes in a triangulable codimension zero submanifold of a
  topological manifold, for instace, in a regular neighborhood of an
  immersed capped grope.

  Now, returning to the proof of
  Proposition~\ref{proposition:unknottedness-of-contraction}, observe
  that the hatched rectangles in
  Figure~\ref{figure:contraction-exterior} lie in thickened caps of
  $G$, and each thickened cap of $G$ cut along the hatched rectangles
  is a 3-cell.  It follows that $L$ expands cellularly to $G$ with
  thickened caps.  By Lemma~\ref{lemma:cellular-expansion-exterior},
  it follows that $E_L$ is homeomorphic to~$E_G$.

  A similar argument can be carried out for the asymmetric case as
  well.
\end{proof}

\begin{corollary}
  \label{corollary:unknotted-grope-of-lower-heights}
  Suppose $G$ is a $\pi_1$-unknotted sphere-like capped grope of
  height $h$ in $S^4$, $h\in \frac12\Z$.  Then for any $h'<h$, there
  is a $\pi_1$-unknotted capped grope of height $h'$ which is obtained
  from $G$ by symmetric/asymmetric contraction.
\end{corollary}

\begin{proof}
  Contract top stage surfaces repeatedly until the capped grope has
  the desired height.  The result is $\pi_1$-unknotted by
  Proposition~\ref{proposition:unknottedness-of-contraction}.
\end{proof}

\section{Transformation between unknotted Whitney towers and
  gropes}
\label{section:schneiderman-transformation}

In \cite{Schneiderman:2006-1}, Schneiderman presents fundamental
constructions which convert an immersed $S$-like capped grope to an
$S$-like Whitney tower immersed in a neighborhood of the capped grope,
and vice versa.  Furthermore, he shows that the corresponding capped
gropes and Whitney towers have exactly the same intersection data
(which are described precisely in terms of uni-trivalent trees).  As a
corollary he shows that a capped grope of height $h$ in a 4-manifold
can be transformed to a Whitney tower of height $h$ with the same
boundary \cite[Corollary~2]{Schneiderman:2006-1}.

In this section we show that Schneiderman's method preserves
unknottedness, as stated below:

\begin{theorem}
  \label{theorem:grope-whitney-tower-transformation}
  \leavevmode\Nopagebreak
  \begin{enumerate}
  \item If a capped grope $G$ is obtained from a Whitney tower $T$ in
    a 4-manifold $W$ by Schneiderman's construction, then
    $E_G \cong E_T \cup ($handles of index $\ge 2)$, and thus there is
    an epimorphism of $\pi_1(W\setminus T)$ onto $\pi_1(W\setminus G)$
    supported by a regular neighborhood of\/~$T$.  Consequently, if $T$
    is sphere-like and $\pi_1$-unknotted in $W=S^4$, then $G$ is
    $\pi_1$-unknotted.
  \item If a Whitney tower $T$ is obtained from a capped grope $G$ in
    a 4-manifold $W$ by Schneiderman's construction, then
    $E_T \cong E_G \cup ($handles of index $\ge 2)$, and thus there is
    an epimorphism of $\pi_1(W\setminus G)$ onto $\pi_1(W\setminus T)$
    supported by a regular neighborhood of\/~$G$.  Consequently, if $G$
    is sphere-like and $\pi_1$-unknotted in $W=S^4$, then $T$ is
    $\pi_1$-unknotted.
  \end{enumerate}
\end{theorem}

Applying Theorem~\ref{theorem:grope-whitney-tower-transformation} to a
grope of height $h$, we obtain the following:

\begin{corollary}
  \label{corollary:unknotted-tower-from-grope-height-h}
  Suppose $G$ is a $\pi_1$-unknotted sphere-like capped grope of
  height $h$ in $S^4$.  Then, in a regular neighborhood of $G$, there
  is a $\pi_1$-unknotted sphere-like Whitney tower of height~$h$ with the same boundary.
\end{corollary}

\begin{remark}
  \label{remark:base-surface-transformation}
  If $P$ is a planar surface contained in the base surface of a
  Whitney tower $T$ and $P$ is disjoint from non-base Whitney disks,
  then Schneiderman's construction produces a capped grope $G$ such
  that the base surface of $G$ contains $P$ and $G\setminus P$ is
  contained in a regular neighborhood of $T\setminus P$.  Conversely,
  if $P$ is a planar surface contained in the base surface of a capped
  grope $G$ and $P$ is disjoint from non-base surfaces and caps, then
  for the Whitney tower $T$ obtained by Schneiderman's construction,
  the base surface of $T$ contains $P$ and $T\setminus P$ is contained
  in a regular neighborhood of $G\setminus P$.  In particular, this
  holds for the height $h$ Whitney tower obtained in
  Corollary~\ref{corollary:unknotted-tower-from-grope-height-h}.
\end{remark}

\subsection{Modifications of Whitney towers and unknottedness}
\label{subsection:modification-whitney-towers-unknottedness}

In the proof of
Theorem~\ref{theorem:grope-whitney-tower-transformation}, we will use
that $\pi_1$-unknottedness is preserved under certain modifications of
Whitney towers and capped gropes, which are discussed in this and next
subsections.

\begin{describe}[Splitting.] In~\cite[\S 3.7, \S
  3.8]{Schneiderman:2006-1}, a splitting procedure was introduced to
  separate intersections in a Whitney tower.  Suppose $D$ is a Whitney
  disk between two sheets $A$ and $B$, and $\beta$ is a properly
  embedded arc in $D$ which joins $A\cap \partial D$ to
  $B\cap \partial D$ and avoids the intersection of $D$ with other
  surfaces/disks.  Then a finger move on $A$ along $\beta$ introduces
  two additional intersections between $A$ and $B$, and divides $D$
  into two new Whitney disks.  See
  Figure~\ref{figure:whitney-disk-splitting} (for now ignore the
  disk~$\Delta$).  Repeated applying this splitting procedure, one
  eventually obtains a Whitney tower each of whose Whitney disks is
  embedded and has either a single unpaired intersection or two
  intersections paired by another Whitney disk.
\end{describe}

\begin{figure}[H]
  \labellist
  \footnotesize
  \pinlabel{$A$} at 10 45
  \pinlabel{$B$} at 31 3
  \pinlabel{$\Delta$} at 283 86
  \endlabellist
  \includegraphics{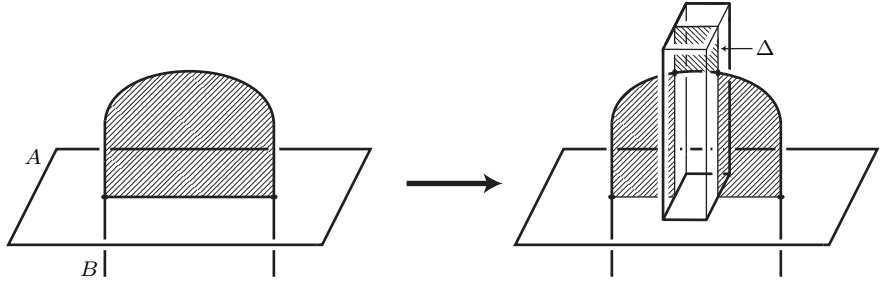}
  \caption{Splitting of a Whitney disk.}
  \label{figure:whitney-disk-splitting}
\end{figure}

\begin{proposition}
  \label{proposition:tower-splitting-unknottedness}
  Suppose a Whitney tower $T'$ is obtained from another Whitney tower
  $T$ by splitting in a 4-manifold~$W$.  Then
  $E_{T'} \cong E_T\cup(\text{$2$-handles})$.  Consequently there is an
  epimorphism of $\pi_1(W\setminus T)$ onto $\pi_1(W\setminus T')$
  supported by a regular neighborhood of~$T$.  In addition, splitting
  changes a $\pi_1$-unknotted sphere-like Whitney tower to a
  $\pi_1$-unknotted Whitney tower in~$S^4$.
\end{proposition}

\begin{proof}
  Let $\Delta$ be the Whitney disk shown in
  Figure~\ref{figure:whitney-disk-splitting}, which pairs the two new
  intersections of $T'$ introduced by the finger move.  Let $L$ be the
  union of $T'$ and the 3-dimensional trace of the finger move.  It is
  easily seen that $T$ expands cellularly to the 3-complex~$L$.  Also,
  the 2-complex $T' \cup \Delta$ expands cellularly to~$L$: first
  thicken $\Delta$ to fill in the top part of the trace of the finger
  move in Figure~\ref{figure:whitney-disk-splitting}, and then stretch
  it down to fill in the remaining part of the trace.  By
  Lemma~\ref{lemma:cellular-expansion-exterior}, it follows that the
  exterior $E_T$ is homeomorphic to~$E_{T'\cup \Delta}$.  By
  Lemma~\ref{lemma:subcomplex-exterior-handle-decomposition},
  $E_{T'} \cong E_{T'\cup \Delta}\cup(\text{2-handles}) \cong\
  E_T\cup(\text{2-handles})$.
\end{proof}

\begin{describe}[Tri-sheet move.] The following modification, which we
  call a \emph{tri-sheet move} in this paper, was introduced
  in~\cite[Lemma~3.6]{Schneiderman:2006-1}.  Suppose $A$, $B$, and $C$
  are surfaces/disks in a Whitney tower.  Suppose $D$ is a Whitney
  disk which pairs two intersections between $A$ and $B$, and $p$ is
  an intersection between $D$ and $C$, as illustrated in the left hand
  side of Figure~\ref{figure:tri-sheet-move}.  We assume that $p$ is
  the only intersection of~$D$.  (By splitting, we may assume this for
  any unpaired intersection~$p$.)  If we denote a Whitney disk between
  $X$ and $Y$ by $(X,Y)$ and denote an intersection between $X$ and
  $Y$ by $\langle X, Y\rangle$, then we may write $p = \langle
  (A,B),C\rangle$.  Apply an isotopy of $B$ in a regular neighborhood
  of $D$ and replace $D$ by another Whitney disk $D'$ as shown in the
  right hand side of Figure~\ref{figure:tri-sheet-move}.  This replaces
  $p=\langle (A,B),C\rangle$ by a new intersection $q=\langle
  A,(B,C)\rangle$.
\end{describe}

\begin{figure}[H]
  \labellist
  \footnotesize
  \pinlabel{$B$} at 3 8
  \pinlabel{$A$} at 37 3
  \pinlabel{$C$} at 15 68
  \pinlabel{$D$} at 85 69
  \pinlabel{$p$} at 75 80
  \pinlabel{$B$} at 186 8
  \pinlabel{$A$} at 220 3
  \pinlabel{$C$} at 198 68
  \pinlabel{$D'$} at 310 110
  \pinlabel{$q$} at 255 66
  \endlabellist
  \includegraphics{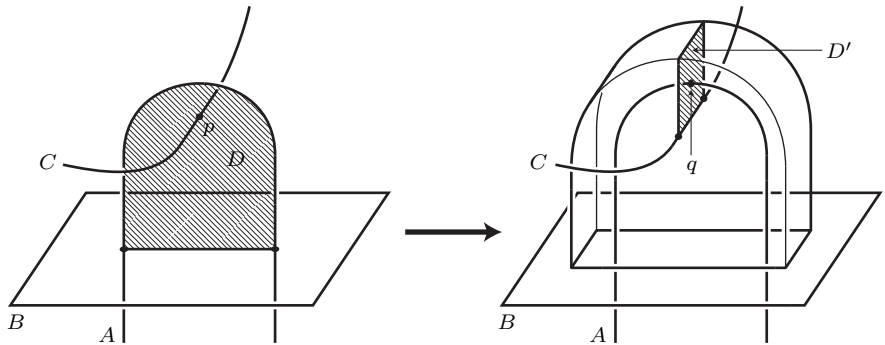}
  \caption{Tri-sheet move.}
  \label{figure:tri-sheet-move}
\end{figure}

\begin{proposition}
  \label{proposition:tri-sheet-move-unknottedness}
  If a Whitney tower $T'$ is obtained from another Whitney tower $T$
  by a tri-sheet move in a 4-manifold~$W$, then $E_{T'} \cong E_T$,
  and thus $\pi_1(W\setminus T')\cong \pi_1(W\setminus T)$.
  Consequently a tri-sheet move changes a $\pi_1$-unknotted
  (union-of-spheres)-like Whitney tower to a $\pi_1$-unknotted Whitney
  tower in~$S^4$.
\end{proposition}

\begin{proof}
  From Figure~\ref{figure:tri-sheet-move} it is seen that $T'$
  cellularly expands to $T$ with thickened~$D$.  By
  Lemma~\ref{lemma:cellular-expansion-exterior}, it follows that
  $E_{T'}$ is homeomorphic to~$E_T$.
\end{proof}

\subsection{Modification of capped gropes and unknottedness}
\label{subsection:modification-capped-gropes-unknottedness}

To prove Theorem~\ref{theorem:grope-whitney-tower-transformation}, we
also need the following observations on some modifications of capped
gropes.

\begin{describe}[Pushing an intersection down.]  Following \cite[\S
  2.5]{Freedman-Quinn:1990-1}, an intersection of a cap $D$ and
  another non-base surface/cap $S$ in an immersed capped grope can be
  changed to two intersections of $D$ with the surface $S'$ to which
  $S$ is attached, by performing a finger move to push $D$ off $S$ and
  through~$S'$.  See Figure~\ref{figure:pushing-intersection-down}
  (for now ignore the hatched rectangle~$\Delta$).  Repeatedly applying
  this, we may assume that each cap intersects only the base surfaces.
\end{describe}

\begin{figure}[H]
  \labellist
  \footnotesize
  \pinlabel{$D$} at 10 44
  \pinlabel{$S$} at 79 120
  \pinlabel{$S'$} at 13 24
  \pinlabel{$D$} at 199 44
  \pinlabel{$S$} at 268 120
  \pinlabel{$S'$} at 202 24
  \pinlabel{$\Delta$} at 222 13
  \endlabellist
  \includegraphics{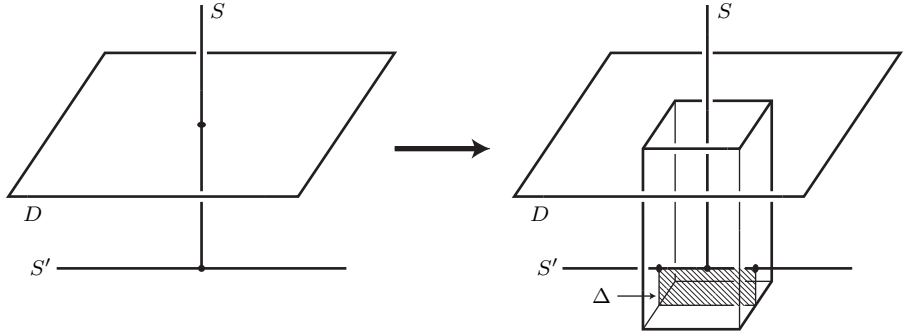}
  \caption{Pushing an intersection down.}
  \label{figure:pushing-intersection-down}
\end{figure}

\begin{proposition}
  \label{proposition:pushing-intersection-down-unknottedness}
  If a capped grope $G'$ is obtained from another capped grope $G$ by
  pushing an intersection down in a 4-manifold $W$, then
  $E_{G'} \cong E_{G}\cup(\text{$2$-handle})$, and thus there is an
  epimorphism of $\pi_1(W\setminus G)$ onto $\pi_1(W\setminus G')$
  supported by a regular neighborhood of~$G$.  Consequently, by
  pushing an intersection down, a $\pi_1$-unknotted sphere-like capped
  grope is changed to a $\pi_1$-unknotted capped grope in~$S^4$.
\end{proposition}

\begin{proof}
  Let $\Delta$ be the disk shown in the right hand side of
  Figure~\ref{figure:pushing-intersection-down}, and let $L$ be the
  union of $G$ and the trace of the finger move performed.  It is seen
  from Figure~\ref{figure:pushing-intersection-down} that both $G$ and
  $G'\cup \Delta$ expand cellularly to the 3-complex~$L$.  By
  Lemma~\ref{lemma:cellular-expansion-exterior}, it follows that
  $E_{G'\cup \Delta}$ is homeomorphic to~$E_G$.  By
  Lemma~\ref{lemma:subcomplex-exterior-handle-decomposition},
  $E_{G'} \cong E_{G'\cup\Delta}\cup(\text{2-handle}) \cong
  E_{G}\cup(\text{2-handle})$.
\end{proof}

\begin{describe}[Grope splitting.]  In \cite{Krushkal:2000-1},
  Krushkal introduced an operation which splits caps and body surfaces
  of a grope.  The cap splitting is described as follows.  Suppose
  $D_1$ and $D_2$ are dual caps attached to a body surface $S$ of a
  capped grope immersed in a 4-manifold.  Choose an embedded arc
  $\alpha$ in $D_1$, which is disjoint from intersection points and
  joins $\partial D_1\cap D_2$ and another point in~$\partial D_1$.
  Perform tubing (surgery) on $S$ along~$\alpha$.  The cap $D_1$ is
  divided into two disks which can be used as caps for the new surface
  obtained, and two parallel copies of $D_2$ can be used as their dual
  caps.  See Figure~\ref{figure:splitting-cap} (ignore the hatched
  rectangles $\Delta_i$ for now).  

  \begin{figure}[H]
    \labellist
    \footnotesize
    \pinlabel{$D_1$} at 68 5
    \pinlabel{$D_2$} at 88 102
    \pinlabel{$S$} at 118 8
    \pinlabel{$\Delta_1$} at 214 112
    \pinlabel{$\Delta_2$} at 214 102
    \endlabellist
    \includegraphics{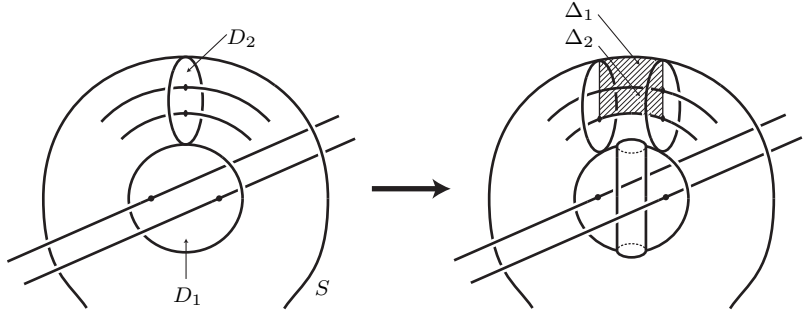}
    \caption{Splitting of a cap~$D_1$.}
    \label{figure:splitting-cap}
  \end{figure}

  Repeatedly applying this operation, one may assume that each cap has
  no self intersection and has at most one intersection point.  (We
  may further assume that each cap has exactly one intersection point
  since we can apply contraction to remove a cap which has no
  intersection points.)  Similarly to the cap case, one can split body
  surfaces: if $\Sigma_1$ and $\Sigma_2$ are dual surfaces attached to
  a previous stage surface $S$ in a capped grope $G$, then apply
  tubing to $S$ along an arc on $\Sigma_1$ and attach two parallel
  copies of the subgrope supported by~$\Sigma_2$.  If $\Sigma_1$ has
  genus greater than one, then by tubing along an appropriate arc,
  $\Sigma_1$ splits into two surfaces with genera less than that
  of~$\Sigma_1$.  Iterating this, we may assume that each non-base
  surface has genus one.  We call such a capped grope a \emph{dyadic
    capped grope}.
\end{describe}

In what follows we assume that a cap of a capped grope is embedded and
can intersect the base surface only, by pushing intersections down if
necessary.

\begin{proposition}
  \label{proposition:grope-splitting-unknotedness}
  If a capped grope $G'$ is obtained from another capped grope $G$ by
  grope splitting in a 4-manifold $W$, then
  $E_{G'} \cong E_G \cup ($handles of index $\ge2)$, and thus there is
  an epimorphism of $\pi_1(W\setminus G)$ onto $\pi_1(W\setminus G')$
  supported by a regular neighborhood of~$G$.  Consequently, grope
  splitting changes a $\pi_1$-unknotted sphere-like capped grope to a
  $\pi_1$-unknotted capped grope in~$S^4$.
\end{proposition}

\begin{proof}
  We assert that there is a 2-complex $L$ in $W$ which contains $G'$
  as a subcomplex and expands cellularly to a 3-complex $L'$ which
  collapses cellularly to~$G$.  From the assertion it follows that
  $E_G$ is homeomorphic to $E_L$ by
  Lemma~\ref{lemma:cellular-expansion-exterior}.  Since $L$ is
  obtained from $G'$ by attaching cells of dimension $\le 2$,
  $E_{G'} \cong E_L \cup ($handles of index $\ge2)$.  This completes
  the proof.

  The remaining part of this proof is devoted to showing the
  assertion.

  The idea is easier to see for the case of cap splitting.  If we
  split $G$ using a cap $D_1$ whose dual cap $D_2$ intersects other
  surfaces $k$ times, then it can be seen from
  Figure~\ref{figure:splitting-cap} (illustrated for $k=2$) that there
  are $k$ 2-cells $\Delta_1,\ldots,\Delta_k$ shown as hatched
  rectangles such that the 2-complex $L:=G'\cup(\bigcup \Delta_i)$
  expands cellularly to $L':=(G$ with thickened $D_2)\cup($solid
  tube$)$: first expand the disk $\bigcup \Delta_i$ to fill in the
  interior of thickened $D_2$, and then stretch down its bottom to
  fill in the solid tube.  It is obvious that the 3-complex $L'$
  collapses cellularly to~$G$.

  The body surface splitting case could also be understood in a
  similar way, but this would require a more complicated picture which
  does not fit into 3-space.  Instead, in what follows we present a
  more formal approach for both cap and body surface cases.

  Suppose $\Sigma_1$ and $\Sigma_2$ are dual surfaces (or caps) in $G$
  attached to the same body surface $S$, and $G'$ is obtained by
  splitting~$\Sigma_1$.  Let $G_i$ be the capped subgropes supported
  by~$\Sigma_i$ for $i=1,2$.  (If $\Sigma_i$ is a cap, then $G_i$ is
  $\Sigma_i$ itself.)  Each $G_i$ is embedded since the caps are
  embedded and can intersect the base surface of $G$  only.  Let
  $V\cong I^3$ be the solid tube used to perform tubing on~$S$.  See
  Figure~\ref{figure:grope-splitting-expansion}.  Recall that $G'$ has
  two parallel copies of~$G_2$ instead of~$G_2$. Choose a thickening
  $G_2\times I$ in such a way that the parallel copies are $G_2\times
  0$ and $G_2\times 1$ and $(G_2\times I) \cap V$ is a rectangle $R$
  of the form $($an interval in $\partial G_2)\times I$.  In
  Figure~\ref{figure:grope-splitting-expansion}, $R$ is shown as a
  hatched rectangle.

  \begin{figure}[H]
    \labellist
    \footnotesize
    \pinlabel{$S$} at 30 200
    \pinlabel{$G_2\times I$} at 70 140
    \pinlabel{$V\cong I^3$} at 170 58
    \pinlabel{$R = (G_2\times I) \cap V$} at 240 182
    \endlabellist
    \includegraphics[scale=.95]{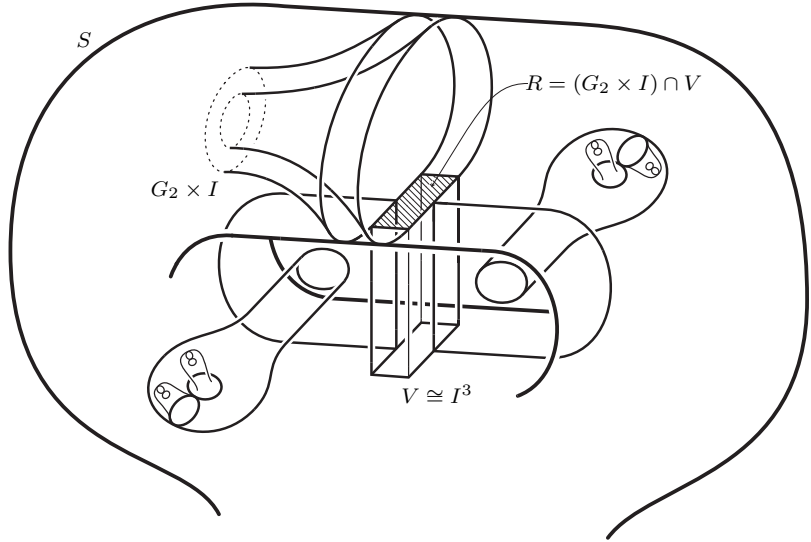}
    \caption{The solid tube and thickened subgrope used for grope
      splitting.}
    \label{figure:grope-splitting-expansion}
  \end{figure}

  We choose a tree $K$ embedded in $G_2$ as follows.  For brevity, for
  a body surface or cap $S_1$ dual to another body surface or cap
  $S_2$, we call the point $\partial S_1 \cap \partial S_2$ the
  \emph{basepoint} of $S_1$ (and~$S_2$).  For the base surface of
  $G_2$, choose a basepoint in~$\partial G_2$.  For each intersection
  $p$ of a cap $D$ in $G_2$ and a sheet in $G$, choose an arc on $D$
  which joins $p$ to the basepoint of~$D$.  For each basepoint $q$ of
  a surface/cap attached to a body surface $S$ in $G_2$, choose an arc
  on $S$ joining $q$ to the basepoint of~$S$.  We assume that the
  interiors of the arcs chosen above are pairwise disjoint and are
  disjoint from the boundary of any surface and cap.  Finally choose
  two arcs in $\partial G_2$ whose intersection is the basepoint of
  the base surface of~$G_2$.  The union of all the above arcs is a
  tree $K$ in~$G_2$ since $G_2$ is embedded.  We regard the above arcs
  as an edge of $K$, and their endpoints as vertices of~$K$.  See
  Figure~\ref{figure:tree-in-dual-grope} for an example.  In addition,
  we may assume that $\overline{\partial G_2\setminus K} \times I$ is
  equal to the rectangle~$R$.

  \begin{figure}[H]
    \labellist
    \footnotesize
    \pinlabel{$K$} at 100 30
    \pinlabel{$G_2$} at 5 100
    \endlabellist
    \includegraphics[scale=.95]{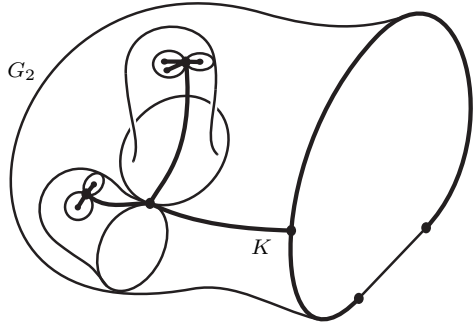}
    \caption{A tree $K$ which expands cellularly to~$G_2$.}
    \label{figure:tree-in-dual-grope}
  \end{figure}

  \begin{assertion}
    \label{assertion:expansion-of-tree-in-grope}
    The tree $K$ expands cellularly to~$G_2$.
  \end{assertion}

  Actually it is seen from Figure~\ref{figure:tree-in-dual-grope} that
  one can expand $K$ to fill in the caps, expand further to fill in
  the top stage body surfaces, and continue to the lower stages to
  eventually fill in the whole grope~$G$.  To make it rigorous, we
  will prove a generalized statement that for any capped subgrope $H$
  in $G_2$, $K\cap H$ expands cellularly to~$H$, by an induction on
  the number $n$ of surfaces/caps in~$H$.
  Assertion~\ref{assertion:expansion-of-tree-in-grope} is the case
  when $H=G_2$.  If $n=1$, then $H$ is a cap, and since $H$ cut along
  $K\cap H$ is a disk, $K\cap H$ expands cellularly to~$H$.  Suppose
  $n>1$, that is, $H$ consists of a body surface $S$ and subgropes
  $H_i$ attached to~$S$.  By induction, $H_i\cap K$ expands cellularly
  to~$H_i$.  Thus it suffices to show that $(\bigcup \partial H_i)
  \cup (K\cap S) \subset S$ expands cellularly to~$S$.  It is true
  since $S$ cut along $(\bigcup \partial H_i) \cup (K\cap S)$ is a
  disk.  This completes the induction and hence the proof of Aseertion~\ref{assertion:expansion-of-tree-in-grope}.

  Let $A$ be the intersection of $G_2\times I$ with sheets in the
  initial capped grope~$G$.  We may assume that each component of $A$
  is a straight arc of the form $v\times I \subset G_2\times I$ for
  some vertex $v$ of~$K$, since the intersection of a cap of $G_2$
  with a sheet is always a vertex of~$K$.

  Let $L_0 := (G_2\times\{0,1\}) \cup (K\times I)$.  Observe that $L_0$
  is the 2-complex obtained from the 2-complex $(G_2\times\{0,1\})
  \cup A$ by attaching 1-cells of the form $v\times I$ for vertices
  $v$ of $K$ not intersecting any sheet, and attaching 2-cells of the
  form $e\times I$ for edges $e$ of~$K$.

  \begin{assertion}
    \label{assertion:expansion-to-thickened-subgrope}
    The 2-complex $L_0$ expands cellularly to $G_2 \times I$.
  \end{assertion}

  Assertion~\ref{assertion:expansion-to-thickened-subgrope} follows
  from Assertion~\ref{assertion:expansion-of-tree-in-grope} by
  applying the following, which must be regarded as a known fact:

  \begin{lemma}
    \label{lemma:expansion-in-product}
    If a subcomplex $K$ of $X$ expands cellularly to $X$, then the
    subcomplex $(X\times\{0,1\}) \cup (K\times [0,1])$ expands
    cellularly to $X\times[0,1]$.
  \end{lemma}

  \begin{proof}
    For an elementary cellular expansion of a complex $K$ in $X$
    across an $n$-disk $\Delta$ such that $\partial_0\Delta :=
    \Delta\cap K$ is an $(n-1)$-disk embedded in $\partial \Delta$,
    $\Delta\times [0,1]$ is an $(n+1)$-disk in $X\times I$ such that
    \[
      (\Delta\times [0,1]) \cap \big((X\times\{0,1\})\cup
      (K\times[0,1])\big) = (\Delta\times\{0,1\}) \cup
      (\partial_0\Delta \times[0,1])
    \]
    is an $n$-disk embedded in $\partial(\Delta\times[0,1])$.  Thus
    there is an elementary cellular expansion of
    $(X\times\{0,1\})\cup (K\times[0,1])$ across $\Delta\times [0,1]$.
    By repeatedly applying this, the desired conclusion is obtained.
  \end{proof}

  Let $G''$ be $G'$ with $G_2\times \{0,1\}$ removed.  Let
  $L:= G''\cup L_0$ and $L':=G'' \cup (G_2\times I) \cup V$.  We will
  verify that $L$ and $L'$ satisfy the properties promised at the
  beginning of the proof.  Obviously $L'$ collapses cellularly to~$G$;
  see Figure~\ref{figure:grope-splitting-expansion}.  Since
  $G_2\times\{0,1\} \subset L_0$, $L$ is a 2-complex containing $G'$
  as a subcomplex.  Since
  $R = \overline{\partial G_2\setminus K} \times I$, we have
  \[
    (G_2\times I) \cap G'' = \overline{(\partial G_2\times I)
      \setminus R} = (K\cap \partial G_2)\times I= L_0\cap G''\subset L_0.
  \]
  See Figures~\ref{figure:grope-splitting-expansion}
  and~\ref{figure:tree-in-dual-grope}.  It follows that the cellular
  expansion of $L_0$ to $G_2\times I$ in
  Assertion~\ref{assertion:expansion-to-thickened-subgrope} is indeed
  a cellular expansion of $L=G''\cup L_0$ to $G''\cup (G_2\times I)$.
  Since
  \[
  \big(G''\cup (G_2\times I)\big) \cap V =
  \partial V\setminus(\text{bottom face of }V)
  \]
  is a disk (see Figure~\ref{figure:grope-splitting-expansion}), there
  is a cellular expansion of $G''\cup (G_2\times I)$ to $L'$ across
  the cube~$V$.  Consequently $L$ expands cellularly to~$L'$.  This
  completes the proof.
\end{proof}

\subsection{Proof of
  Theorem~\ref{theorem:grope-whitney-tower-transformation}}

\subsubsection*{From Whitney towers to capped gropes}

Suppose $T$ is a Whitney tower in a 4-manifold~$W$.  In
\cite[Section~5.1]{Schneiderman:2006-1}, a capped grope $G$ is
constructed from $T$ by first applying Whitney tower splitting and
then applying a sequence of tri-sheet moves and the following
operation:

\begin{describe}[Tubing on a cap.]
  Suppose a dyadic capped grope immersed in a 4-manifold is given.
  Suppose no cap has self intersections while a cap may intersect base
  surfaces.  Suppose the union of caps and base surfaces supports a
  Whitney tower disjoint from non-base body surfaces.  We regard the
  union of the capped grope and the Whitney tower as a 2-complex,
  say~$P$.  (In particular a \emph{split grope subtower} introduced in
  \cite[Section~4.2]{Schneiderman:2006-1} is such a 2-complex; we will
  not need its precise definition.)  Suppose $D$ is a cap of the
  capped grope which has two intersections paired by a Whitney
  disk~$\Delta$.  Perform tubing on $D$ along
  $\partial\Delta \setminus D$ and attach, as new caps, a meridional
  disk of the tube and $\Delta$ with a collar neighborhood removed.
  See Figure~\ref{figure:tubing-on-cap}.  This gives us a new
  2-complex, say $P'$, which consist of a capped grope and a Whitney
  tower.
\end{describe}

\begin{figure}[H]
  \labellist
  \footnotesize
  \pinlabel{$D$} at 90 54
  \pinlabel{$\Delta$} at 79 95
  \pinlabel{new caps} at 144 86
  \endlabellist
  \includegraphics{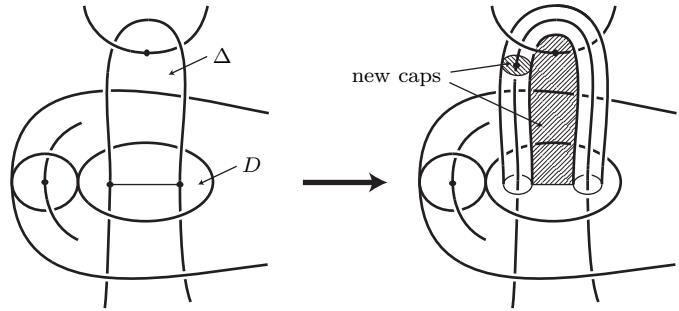}
  \caption{Tubing on a cap $D$ which has two intersections paired by a
    Whitney disk~$\Delta$.}
  \label{figure:tubing-on-cap}
\end{figure}

Observe that the above tubing operation preserves the homeomorphism
type of the 2-complex exterior.  For, from
Figure~\ref{figure:tubing-on-cap} it is seen that the new 2-complex
$P'$ expands cellularly to the union $P''$ of the solid tube and the
initial 2-complex $P$, and then $P''$ collapses cellularly to~$P$.  By
Lemma~\ref{lemma:cellular-expansion-exterior} the exterior is
preserved.

From this and
Propositions~\ref{proposition:tower-splitting-unknottedness}
and~\ref{proposition:tri-sheet-move-unknottedness}, it follows that
$\pi_1(W\setminus G)$ is a quotient of $\pi_1(W\setminus T)$.  This
proves Theorem~\ref{theorem:grope-whitney-tower-transformation}(1).

\subsubsection*{From capped gropes to Whitney towers}

Suppose $G$ is a capped grope in a 4-manifold~$W$.  In
\cite[Section~5.2]{Schneiderman:2006-1}, a Whitney tower $T$ is
constructed from $G$ by first pushing intersections down and applying
grope splitting, and then applying a sequence of tri-sheet moves and
the following operation:

\begin{describe}[Surgery along a cap]
  Suppose $P$ is a 2-complex consisting of a dyadic capped grope and a
  Whitney tower supported by caps and base surfaces as in the case of
  tubing on a cap.  Suppose $D_1$ and $D_2$ are dual caps attached to
  the same body surface $S$, and $D_1$ intersects exactly one sheet
  $A$ of the Whitney tower.  Replace the subgrope $S\cup D_1\cup D_2$
  by a cap obtained by ambient surgery on $S$ using~$D_1$.  Note that
  this new cap has two intersections with the sheet~$A$, which are
  paired by a Whitney disk obtained by attaching to $D_2$ a band
  contained in $\nu(D_1)$, as shown in
  Figure~\ref{figure:surgery-along-cap}.  Add this disk to the Whitney
  tower.  This gives us a new 2-complex $P'$ which consists of the
  modified capped grope and the modified Whitney tower.
\end{describe}

Although we start with a capped grope, this surgery operation produces
a hybrid 2-complex of a capped grope and a nontrivial Whitney tower.
Thus the use of a tri-sheet move, which is defined for a Whitney
tower, makes sense in the above conversion process.

\begin{figure}[H]
  \labellist
  \footnotesize
  \pinlabel{$A$} at 32 68
  \pinlabel{$D_1$} at 8 6
  \pinlabel{$D_2$} at 83 28
  \pinlabel{new Whitney disk} at 260 28
  \endlabellist
  \includegraphics{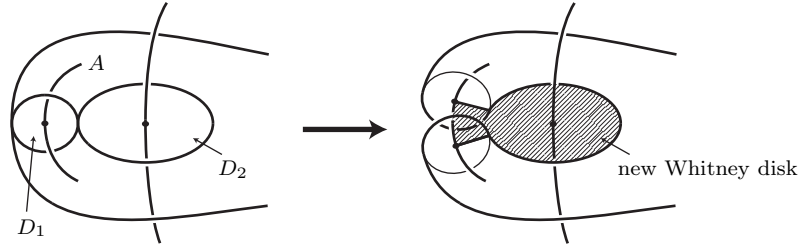}
  \caption{Surgery along a cap which intersects a sheet.}
  \label{figure:surgery-along-cap}
\end{figure}

From Figure~\ref{figure:surgery-along-cap} it is seen that the
resulting 2-complex $P'$ expands cellularly to $P$ with the cap $D_1$
thickened, which then collapses cellularly to~$P$.  By
Lemma~\ref{lemma:cellular-expansion-exterior}, the homeomorphism type
of the 2-complex exterior is preserved by the above surgery operation.

From this and
Propositions~\ref{proposition:pushing-intersection-down-unknottedness}
and~\ref{proposition:grope-splitting-unknotedness}, it follows that
$\pi_1(W\setminus T)$ is a quotient of $\pi_1(W\setminus G)$.  This
proves Theorem~\ref{theorem:grope-whitney-tower-transformation}(2).

\section{Handle decomposition of grope exteriors}
\label{section:handle-decomposition-grope-exterior}

In the literature constructions of gropes in a 4-manifold bounded by
knots were developed for the study of concordance and related
4-dimensional equivalence relations.  For instance see
\cite{Cochran-Teichner:2003-1, Horn:2010-1, Cha:2012-1,
  Cha-Powell:2013-1}.  In this section we study the handle
decomposition of the exterior of such gropes.  We will use this in the
next sections to construct certain explicit unknotted gropes and
Whitney towers in~$S^4$.

\subsection{Satellite capped gropes and capped grope concordances}
\label{subsection:product-of-gropes}

We begin by presenting some definitions, which appeared or are
influenced by earlier papers \cite{Cha:2012-1,
  Cochran-Teichner:2003-1, Horn:2010-1}.  From now on all immersed
capped gropes are assumed to satisfy that each cap can intersect the
base surface only.  (Push intersections down as described in
Section~\ref{subsection:modification-capped-gropes-unknottedness} if
necessary.)

\begin{definition}[Satellite capped
  grope~{\cite[Definition~4.2]{Cha:2012-1}}]
  \label{definition:satellite-capped-grope}
  Suppose $L$ is a link in $S^3$ and $\alpha$ is an unknotted simple
  closed curve in $S^3$ which is disjoint from~$L$.  Let $E_\alpha$ be
  the exterior of $\alpha\subset S^3$.  Let $\lambda_\alpha$ be a
  zero linking longitude in~$\partial E_\alpha$.  A \emph{satellite
    capped grope} for $(L,\alpha)$ is a disk-like capped grope $G$
  immersed in the 4-manifold $E_\alpha\times I$ such that $G$ is
  bounded by $\lambda_\alpha\times 0$, the body of $G$ is disjoint
  from $L\times I$, and the caps are transverse to $L\times I$.
\end{definition}

If $(L,\alpha)$ is as above, $L$ may be viewed as a pattern in
$E_\alpha \cong D^2\times S^1$ for a satellite construction: for a
knot $K$, by attaching $(E_\alpha,L)$ to $E_K$ along an orientation
reversing homeomorphism $\partial E_\alpha \cong \partial E_K$ which
identifies a meridian and a zero linking longitude of $\alpha$ with a
zero linking longitude and a meridian of $K$ respectively, we obtain a
satellite link in $E_\alpha\cup_\partial E_K \cong S^3$.  We denote
this link by $L(\alpha, K)$ or $L(\alpha;K)$.

\begin{definition}[Product of satellite capped gropes]
  \label{definition:composition-satellite-capped-gropes}
  Suppose $G\subset E_\alpha\times I$ and $H\subset E_\eta\times I$
  are satellite capped gropes for $(L,\alpha)$ and $(K,\eta)$
  respectively.  View $\eta$ as a curve in $E_K \subset E_{L(\alpha,
    K)}$.  The \emph{product} $G\cdot H$ of $G$ and $H$ is a satellite
  capped grope for $(L(\alpha,K),\eta)$ described below.  Note $E_\eta
  = E_{\eta\cup K} \cup E_\alpha$.  By an isotopy of caps of $H$ if
  necessary, we may assume that $H\cap (\partial E_{\eta\cup K}\times
  I)$ consists of circles of the form $\mu_K\times t$ where $\mu_K$ is
  a meridian of $K$ and $0<t<1$.  For each $\mu_K\times t$, take a
  copy $G_t$ of $G$ in $E_\alpha \times [t, t+\epsilon] \cong
  E_\alpha\times I$ and attach $G_t$ to $H\cap (E_{\eta\cup K}\times
  I)$ along $\mu_K\times t = \partial G_t$.  This gives a desired
  satellite capped grope for $(L(\alpha, K),\eta)$ in $E_\eta \times I
  = (E_{\eta\cup K} \cup E_\alpha) \times
  I$~\cite[Proposition~4.3]{Cha:2012-1}.
\end{definition}

The product $G\cdot H$ is obtained by removing disjoint disks embedded
in caps of $H$ and then fill in it with copies of~$G$.  It follows
that if $G$ and $H$ have height $h$ and $k$ respectively, then
$G\cdot H$ has height $h+k$.

\begin{definition}[Capped grope concordance]
  \label{definintion:grope-concordance}
  A \emph{capped grope concordance} between two links $L$ and $L'$ is
  a (union-of-annuli)-like capped grope immersed in $S^3\times I$ such
  that the $i$th component of the base surface is cobounded by that of
  $L\times 0$ and that of $L'\times 1$.
\end{definition}

In the literature, a grope concordance without caps is often
considered.  A grope concordance can be promoted to a capped grope
concordance since $S^3\times I$ is simply connected.

The following is motivated by
\cite[Theorem~3.8]{Cochran-Teichner:2003-1},
\cite[Theorem~3.4]{Horn:2010-1}.

\begin{definition}[Product of a satellite capped grope and a capped
  grope concordance]
  \label{definition:composition-satellite-grope-and-grope-concordance}
  Suppose $G$ is a satellite capped grope for $(L,\alpha)$ and $H$ is
  a capped grope concordance between two knots $J$ and~$J'$.  The
  \emph{product} $G\cdot H$ is a capped grope concordance between
  $L(\alpha,J)$ and $L(\alpha,J')$ described below.  Fix closed
  intervals $U\subset S^1$, $V\subset (0,1)\subset I$, and regard $H$
  as the annulus $S^1\times I$ with the disk $U\times V$ replaced by a
  disk-like capped grope $B$ with $\partial B=\partial (U\times V)$ :
  \[
  H=\overline{(S^1\times I) \setminus (U\times V)} \cupover{\partial
    (U\times V)=\partial B} B \subset S^3\times I.
  \]
  Here caps of $B$ may be plumbed with $(S^1\times I) \setminus
  (U\times V)$.  A regular neighborhood of $H$ can be written as
  \[
  \nu(H) = \overline{(S^1\times D^2\times I) \setminus (U\times
    D^2\times V)} \cupover{S^1\times D^2} \nu(B) \subset S^3\times I
  \]
  with plumbings performed.  By an isotopy of $L\subset
  E_\alpha=S^1\times D^2$, we may assume $L\cap (U\times D^2)=U\times
  \{p_1,\ldots,p_r\}$ for some $p_i\in D^2$.  Then $L\times I\subset
  S^1\times D^2\times I$ intersects the 4-ball $U\times D^2\times V$
  at disks $U\times \{p_1,\ldots,p_r\}\times V$.  Choose $r$ parallel
  copies $B_1,\ldots, B_r$ of $B$ such that $\partial B_i
  = \partial(U\times p_i\times V)$, and consider the capped grope
  \[
  \big((L\times I) \setminus U\times \{p_1,\ldots,p_r\}\times V\big)
  \cupover{\partial(U\times p_i\times V)=\partial B_i} \Big(\bigcup_i
  B_i\Big)
  \]
  in~$\nu(G)$.  By isotopy, we may assume that caps of $B_i$ intersect
  $S^1\times D^2\times I$ at disks of the form $z_j\times D^2\times
  t_j$ where $z_j\in S^1$, $t_j\in (0,1)\setminus V$.  Choose
  sufficiently small $\epsilon>0$, and replace each disk $z_j\times
  D^2\times t_j$ in caps of $B_i$ with a copy of the satellite capped
  grope $G$ in $S^1\times D^2\times[t_j,t_j+\epsilon] \cong
  E_\alpha\times I$.  This gives a promised capped grope concordance
  $G\cdot H$ between $L(\alpha,J)$ and $L(\alpha,J')$.
\end{definition}

Similarly to the case of
Definition~\ref{definition:composition-satellite-capped-gropes}, if
$G$ and $H$ have height $h$ and $k$ respectively, then the product
$G\cdot H$ in
Definition~\ref{definition:composition-satellite-grope-and-grope-concordance}
has height $h+k$.

\begin{remark}
  \label{remark:associativity-of-product}
  The product operations defined above are associative.  Precisely, if
  $G_i$ is a satellite capped grope for $(K_i, \alpha_i)$ for
  $i=1,2,3$, then $(G_1\cdot G_2)\cdot G_3$ and
  $G_1\cdot (G_2\cdot G_3)$ are isotopic satellite capped gropes for
  \[
    ((K_1(\alpha_1,K_2))(\alpha_2,K_3),\alpha_3) \approx
    (K_1(\alpha_1, K_2(\alpha_2,K_3)),\alpha_3).
  \]
  Also, if $H$ is a capped grope concordance between $J$ and $J'$,
  then $(G_1\cdot G_2)\cdot H$ and $G_1\cdot (G_2\cdot H)$ are
  isotopic capped grope concordances between
  \[
    (K_1(\alpha_1,K_2))(\alpha_2,J) \approx
    K_1(\alpha_1,K_2(\alpha_2,J))
  \]
  and
  \[
    (K_1(\alpha_1,K_2))(\alpha_2,J') \approx
    K_1(\alpha_1,K_2(\alpha_2,J')).
  \]
  The proofs are straightforward.  Since we do not use this in this
  paper, we omit details.
\end{remark}

\subsection{Handle structure of grope exteriors}
\label{subsection:handle-structure-grope-exterior}

In the literature, satellite capped gropes and grope concordances in a 4-manifold are
often obtained by pushing capped gropes in the boundary 3-manifold.  For instance see
\cite{Cochran-Teichner:2003-1,Horn:2010-1,Cha:2012-1,Cha-Powell:2013-1}.
In this subsection we will investigate handle decomposition of the
4-dimensional exteriors of such capped gropes and their iterated
products obtained by the constructions described in
Definitions~\ref{subsection:product-of-gropes}
and~\ref{definition:composition-satellite-capped-gropes}.

To state the result rigorously, we use the following definitions.

\begin{definition}
  \label{definition:capped-gropes-in-3D}
  \begin{enumerate}
  \item We say that $G$ is a \emph{3D satellite capped grope}
    for~$(L,\alpha)$ if $G$ is a capped grope embedded in $S^3$
    bounded by an unknotted circle $\alpha$, and $L$ is a link in
    $S^3$ disjoint from the body of $G$ and transverse to caps of~$G$.
    By pushing $G\subset S^3=S^3\times0$ into $S^3\times (0,1)$, a
    satellite capped grope for $(L,\alpha)$ is obtained.
  \item We say that $G$ is a \emph{3D capped grope concordance}
    between two links $L$ and $L'$ in $S^3$ if the following hold: (i)
    $L'$ is obtained from $L$ by replacing disjoint arcs
    $\gamma_1,\ldots,\gamma_k \subset L$ with arcs
    $\gamma'_1,\ldots,\gamma'_k \subset S^3$ such that
    $\partial\gamma_i = \partial\gamma'_i$, that is, letting
    $C:=L\setminus (\gamma_1\cup\cdots\cup\gamma_k)$,
    $L'=C\cup \gamma'_1 \cup \cdots\cup \gamma'_k$; (ii) $G$ is an
    embedded (union-of-disks)-like capped grope in $S^3$ bounded by
    the circles $\gamma_i\cup\gamma'_i$ such that the body is disjoint
    from $C$ and the caps are transverse to~$C$.  Note that the capped
    grope
    \[
      (L\times[0,\tfrac12]) \cup (G\times\tfrac12) \cup
      (L'\times[\tfrac12,1])
    \]
    is a capped grope concordance in $S^3\times I$ between $L$
    and~$L'$.  We say that it is obtained by pushing $G$ into
    $S^3\times I$.
  \end{enumerate}
\end{definition}

\begin{theorem}
  \label{theorem:composition-grope-handle-decomposition}
  Suppose $H$ is a capped grope concordance which is obtained by
  pushing a 3D capped grope concordance between two links $J$ and $J'$
  into $S^3\times I$.  For $i=1,\ldots, n$, suppose $G_i$ is a
  satellite capped grope obtained by pushing a 3D satellite capped
  grope for~$(K_i,\alpha_i)$ into $S^3\times I$, where $K_i$ is a
  knot.  Let $G$ be the iterated product
  \[
    G = G_1\cdot G_2\cdots G_n \cdot H
  \]
  which is a capped grope concordance between
  \[
    K:=K_1(\alpha_1, \cdots, K_n(\alpha_n,J)\cdots)
    \quad\text{and}\quad K':=K_1(\alpha_1, \cdots,
    K_n(\alpha_n,J')\cdots).
  \]
  Then the exterior $E_G$ in $S^3\times I$ has a handle decomposition
  \[
    E_G \cong (E_K \times I) \cup (\text{$2$-handles}).
  \]
\end{theorem}

In the proof of
Theorem~\ref{theorem:composition-grope-handle-decomposition}, we will
show that the conclusion holds for any choice of parenthesization for
the product $ G_1\cdots G_n \cdot H$.  (Indeed, the product is well
defined up to isotopy by
Remark~\ref{remark:associativity-of-product}.)

Our approach may be compared with the standard methods in embedded
Morse theory, which is used to construct a handle decomposition of the
exterior of an \emph{embedded submanifold} in $M\times I$ from the
critical points of the submanifolds.  We will present an analog for
the gropes in
Theorem~\ref{theorem:composition-grope-handle-decomposition}, which
are not submanifolds but 2-complexes.

\subsubsection*{Near a surface stage (critical level of type $A$)}

First we consider a surface stage of a capped grope together with a
collar neighborhood of the boundary of next stages attached to it.
This is explicitly described as follows.  Suppose $\Sigma$ is a
surface of genus $g$ with connected nonempty boundary, which is
embedded in a 3-manifold~$M$.  Suppose
$\alpha_1,\beta_1,\ldots,\alpha_g,\beta_g$ are standard symplectic
basis curves on $\Sigma$, that is, they are simple closed curves such
that any two of them are disjoint except $\alpha_i \cap \beta_i=\{$one
point$\}$.  Choose a bicollar $\Sigma\times[-\epsilon,\epsilon]$ of
$\Sigma=\Sigma\times 0$ in $M$, and let
\begin{alignat*}{2}
  \Sigma' & := \Sigma\cup \big(\bigcup \alpha_i\times[0,\epsilon]
  \big) \cup \big(\bigcup \beta_i\times[-\epsilon,0]\big) &&\subset M,
  \\
  \partial_-\Sigma' &:= \partial\Sigma &&\subset \Sigma',
  \\
  \partial_+\Sigma' &:= \big(\bigcup \alpha_i\times \epsilon \big) \cup
  \big(\bigcup \beta_i\times(-\epsilon)\big) &&\subset \Sigma'.  
\end{alignat*}
Now, in the 4-manifold $M\times I$, let
\[
A := (\partial_-\Sigma'\times[0,\tfrac12]) \cup (\Sigma'\times \tfrac12)
\cup (\partial_+\Sigma'\times[\tfrac12,1]).
\]
Let $E_A = M\times I \setminus \nu(A)$ and $E_{\partial_-\Sigma'} =
M\setminus \nu(\partial_-\Sigma')$ be the exteriors.

\begin{proposition}
  \label{proposition:surface-stage-regnbd-complement}
  The 4-manifold $E_A$ has a handle decomposition
  \[
    E_A \cong (E_{\partial_-\Sigma'}\times I) \cup \big(\text{$(2g-1)$
      $2$-handles}\big).
  \]
\end{proposition}

\begin{proof}
  Choose a regular neighborhood $Y:=\nu(\Sigma')$ in $M$ and choose
  regular neighborhoods $Y_+ := \nu(\partial_+\Sigma')$ and
  $Y_- := \nu(\partial_-\Sigma')$ lying in the interior of~$Y$.  The
  exterior $E_A$ is homeomorphic to
  \[
    \overline{M\times I \setminus \big(Y_-\times[0,\tfrac13] \cup
      Y\times[\tfrac13,\tfrac23] \cup Y_+\times[\tfrac23,1]\big)}.
  \]
  Fix a point $*$ on~$\partial\Sigma$.  Choose arcs $\gamma_i$
  ($i=1,\ldots,g$) on $\Sigma$ joining $*$ to $\alpha_i\cap\beta_i$ in
  such a way that $\Sigma$ collapses cellularly (and thus is homotopy
  equivalent) to $K:=\bigcup_i(\alpha_i\cup\beta_i\cup\gamma_i)$.  See
  Figure~\ref{figure:surface-level-reg-nbhd}.  Since $\Sigma'$
  collapses cellularly to $\Sigma$, $Y=\nu(\Sigma')$ is isotopic to a
  regular neighborhood of~$K$.

  \begin{figure}[H]
    \begin{tikzpicture}[x=1bp,y=1bp]
      \small
      \node [anchor=south west, inner sep=0mm] {\includegraphics{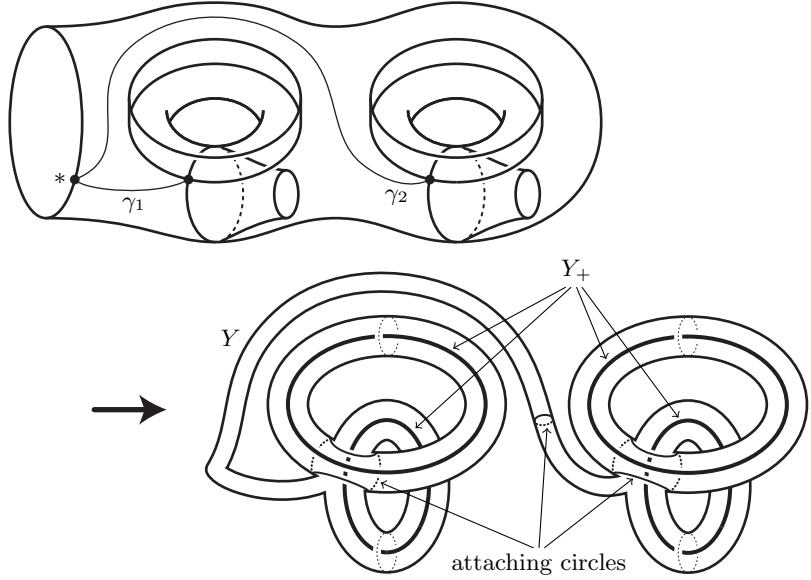}};
      \node [left] at (25,150) {$*$};
      \node [below] at (47,146) {$\gamma_1$};
      \node [below] at (145,149) {$\gamma_2$};
      \node [left] at (90,90) {$Y$};
      \draw [->] (208,109)--(166,86);
      \draw [->] (210,109)--(152,56);
      \draw [->] (212,109)--(223,82);
      \draw [->] (214,109)--(250,60);
      \node [above] at (212,108) {$Y_+$};
      \draw [->] (196,11)--(139,35.5);
      \draw [->] (198,11)--(201,56);
      \draw [->] (200,11)--(233,37);
      \node [below] at (198,12) {attaching circles};
    \end{tikzpicture}
    \caption{A regular neighborhood of the 2-complex~$\Sigma'$.}
    \label{figure:surface-level-reg-nbhd}
  \end{figure}

  From this it is seen that $\overline{Y\setminus Y_+}$ is obtained by
  attaching $(2g-1)$ 2-handles to a collar of
  $\partial Y = \overline{\partial Y\setminus \partial Y_+}$ (in this case $\partial Y\cap \partial Y_+ =\emptyset$); the
  attaching circles are shown in
  Figure~\ref{figure:surface-level-reg-nbhd}.  The proof is completed
  by applying the following fact, which we state as a lemma for later
  use as well.
\end{proof}

\begin{lemma}
  \label{lemma:general-handles-at-critical-level}
  Suppose $Y_+$, $Y_-$, $Y$ are compact codimension zero submanifolds
  in a 3-manifold $M$ such that $Y_+, Y_- \subset Y$.  Suppose
  $\overline{Y\setminus Y_+}$ is obtained from a regular neighborhood of
  $\overline{\partial Y\setminus \partial Y_+}$ in $Y$ by attachments of
  handles $h_1^{p_1},\ldots,h_k^{p_k}$ of index~$p_i$.  That is, there
  is a decomposition
  \[
  Y = \nu(\overline{\partial Y\setminus \partial Y_+})\cup h_1^{p_1}\cup \cdots
  \cup h_k^{p_k} \cup Y_+.
  \]
  Then the 4-manifold 
  \begin{equation}
    \label{equation:exterior-near-critical-level}
    W := \overline{M\times I \setminus \big(Y_-\times[0,\tfrac13] \cup
      Y\times[\tfrac13,\tfrac23] \cup Y_+\times[\tfrac23,1]\big)}    
  \end{equation}
  has a handle decomposition
  \[
  W\cong \overline{M\setminus Y_-}\times I \cup H_1^{p_1} \cup\cdots
  \cup H_k^{p_k}.
  \]
  where $H_i^{p_i}$ is a handle of index~$p_i$.
\end{lemma}

\begin{figure}[H]
  \begin{tikzpicture}[x=1pt,y=1pt]
    \small
    \node [anchor=south west, inner sep=0mm] {\includegraphics{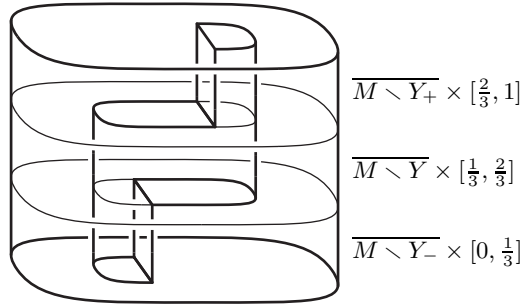}};
    \node [right] at (125,20) {$\overline{M\setminus Y_-}\times[0,\frac13]$};
    \node [right] at (125,50) {$\overline{M\setminus Y}\times[\frac13,\frac23]$};
    \node [right] at (125,80) {$\overline{M\setminus Y_+}\times[\frac23,1]$};
  \end{tikzpicture}
  \caption{The 4-manifold $W$ in Lemma~\ref{lemma:general-handles-at-critical-level}}
  \label{figure:critical-level-exterior}
\end{figure}

\begin{proof}
  Let $W_t = W\cap M\times [0,t]$ for $t\in I$.  We will investigate
  how $W_t$ changes as $t$ increases.  See the schematic picture in
  Figure~\ref{figure:critical-level-exterior}.  For $t\le \frac13$,
  $W_{t} = \overline{M\setminus Y_-}\times[0,t] \cong
  \overline{M\setminus Y_-}\times I$ obviously.  For $\frac13 < t <
  \frac23$, we still have $W_{t} \cong \overline{M\setminus Y_-}\times
  I$, since $W_{t}$ is obtained from $\overline{M\setminus Y_-}\times
  [0,t]$ by pushing the codimension zero submanifold
  $\overline{Y\setminus Y_-} \times[\frac13, t]$ into its complement.
  (This can be viewed as, for instance, repeatedly applying elementary
  cellular collapsing across the 4-balls $\Delta^3\times [\frac13,t]$
  where $\Delta^3$ is a 3-simplex of $\overline{Y\setminus Y_-}$.)
  For $t>\frac23$, fixing $t_0 \in (\frac12,\frac23)$, we have
  \[
    W_t \cong W_{t_0} \cup \overline{Y\setminus Y_+} \times
    [\tfrac12,t_0].
  \]
  The attachment of $\overline{Y\setminus Y_+} \times [\frac12,t_0]$
  to $W_t$ along
  $\overline{\partial Y\setminus \partial Y_+}\times [\frac12,t_0]$ is
  equivalent to attachments of handles
  $H_i^{p_i} := h_i^{p_i}\times[\frac12,t_0]$ of index $p_i$, since
  $\overline{Y\setminus Y_+}$ is obtained by attaching the handles $h_i^{p_i}$
  to a collar of $\overline{\partial Y\setminus \partial Y_+}$.  This shows
  that $W$ has the promised handle decomposition.
\end{proof}

\subsubsection*{Near a plumbing point (critical level of type $B$)}
Suppose $D$ is a 2-disk embedded in a 3-manifold $M$, and $L$ is a
$1$-submanifold in $M$ which is disjoint from $\partial D$ and meets
$D$ at a single transverse intersection point.  In the 4-manifold
$M\times I$, let
\[
  B := (L\times I) \cup (\partial D\times[0,\tfrac12]) \cup
  (D\times\tfrac12)
\]
Let $E_B=M\times I \setminus \nu(B)$ and
$E_{L\cup \partial D} = M\setminus \nu(L\cup \partial D)$ be the exteriors.

\begin{proposition}
  \label{proposition:plumbed-disk-regnbd-complement}
  The 4-manifold $E_B$ is diffeomorphic to $E_{L\cup \partial D}\times I$.
\end{proposition}

\begin{proof}
  We proceed similarly to the proof of
  Proposition~\ref{proposition:surface-stage-regnbd-complement}.  Let
  $Y=\nu(L\cup D)$ be a regular neighborhood in $M$ and choose regular
  neighborhoods $Y_+ = \nu(L)$ and $Y_-:=\nu(L\cup \partial D)$ which
  are contained in the interior of~$Y$.  Then $E_B$ is diffeomorphic
  to the 4-manifold described in
  \eqref{equation:exterior-near-critical-level} of
  Lemma~\ref{lemma:general-handles-at-critical-level}.  Since $D$
  intersects $L$ at one point, $Y\cong \nu(L)$ and
  $\overline{Y\setminus Y_+}$ is just a collar of $\partial Y=
  \overline{\partial Y \setminus \partial Y_+}$.  By
  Lemma~\ref{lemma:general-handles-at-critical-level}, it follows that
  $E_B \cong E_{L\cup \partial D}\times I$, without adding any handle.
  %
\end{proof}

\subsubsection*{Near punctured surfaces (critical levels of type $R$ and $S$)}

We also need the following two cases.  First, we consider a
``horizontal'' punctured disk described as follows.  Suppose $\Gamma$
is a planar surface embedded in a 3-manifold~$M$.  Let
$\partial_-\Gamma$ be a boundary component of $\Gamma$, and
$\partial_+ \Gamma:=\partial \Gamma\setminus \partial_-\Gamma$.  We
assume both $\partial_+\Gamma$ and $\partial_-\Gamma$ are nonempty.
Consider the surface
\[
  R:=(\partial_-\Gamma\times[0,\tfrac12]) \cup (\Gamma\times\tfrac12)
  \cup (\partial_+\Gamma\times[\tfrac12,1])
\]
embedded in $M\times I$.  Let $E_R=M\times I \setminus \nu(R)$ and
$E_{\partial_-\Gamma} = M\setminus \nu(\partial_-\Gamma)$ be the
exteriors.

\begin{proposition}
  \label{proposition:punctured-disk-regnbd-complement}
  The 4-manifold $E_R$ has a handle decomposition
  \[
  E_R \cong (E_{\partial_-\Gamma}\times I) \cup \big(\text{$(k-1)$ $2$-handles}\big)
  \]
  where $k$ is the number of components of $\partial_+\Gamma$.
\end{proposition}

The proof will be given together with that of
Proposition~\ref{proposition:punctured-vertical-sheet-regnbd-complement},
which treats a ``vertical'' punctured sheet described as follows.
Suppose $C_-$ is a 1-submanifold in a 3-manifold $M$ and $A$ is an
annulus embedded in $M$ with two boundary circles $\partial_0 A$,
$\partial_1 A$ such that $J_- := \partial_0 A \cap C_-$ is an arc and
$A\setminus J_-$ is disjoint from~$C_-$.  Let
$J_+=\overline{\partial_0 A \setminus J_-}$ and
$C_+= (C_-\setminus J_-)\cup J_+$.  Let
\[
  S = (C_-\times [0,\tfrac12]) \cup \big((C_-\cup A)\times\tfrac12\big)
  \cup \big((C_+ \cup \partial_1 A)\times[\tfrac12,1]) \subset M\times
  I.
\]
Note that $S$ is homeomorphic to $C_-\times I$ with a disk removed,
and embedded in $M\times I$ in such a way that the boundary of the
removed disk appears as a circle $\partial_1 A\times 1$.  See the
schematic picture in Figure~\ref{figure:punctured-sheet}.  Let
$E_S=M\times I \setminus \nu(S)$ and $E_{C_-} = M\setminus \nu(C_-)$
be the exteriors.

\begin{figure}[H]
  \begin{tikzpicture}[x=1pt, y=1pt]
    \small
    \node[anchor=south west, inner sep=0]
      {\includegraphics[scale=1]{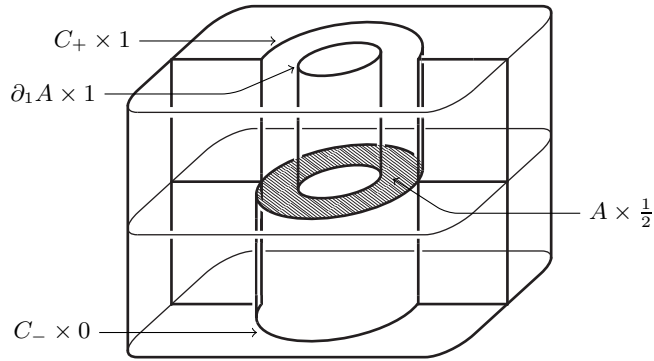}};
    \draw [->] (-10,10)--(49,10);
    \node [left] at (-10,10) {$C_-\times 0$};
    \draw [->] (5,120)--(56,120);
    \node [left] at (5,120) {$C_+\times 1$};
    \draw [->, rounded corners] (170,55)--(120,55)--(100,68);
    \node [right] at (170,55) {$A\times\tfrac12$};
    \draw [->, rounded corners] (-8,100)--(40,100)--(63,110);
    \node [left] at (-8,100) {$\partial_1 A\times 1$};
  \end{tikzpicture}
  \caption{An embedded punctured sheet.}
  \label{figure:punctured-sheet}
\end{figure}

\begin{proposition}
  \label{proposition:punctured-vertical-sheet-regnbd-complement}
  The 4-manifold $E_S$ has a handle decomposition
  \[
    E_S \cong (E_{C_-}\times I) \cup (\text{$2$-handle}).
  \]
\end{proposition}

\begin{proof}[Proofs of
  Propositions~\ref{proposition:punctured-disk-regnbd-complement}
  and~\ref{proposition:punctured-vertical-sheet-regnbd-complement}]

  Since $R$ and $S$ are \emph{embedded submanifolds}, we can apply the
  standard embedded Morse theory to obtain a handle decomposition of
  the exteriors (e.g.,
  see~\cite[Proposition~6.2.1]{Gompf-Stipsicz:1999-1}): briefly, in
  our case, a $p$-handle of the submanifold corresponds to a
  $(p+1)$-handle of the exterior.  Since
  $R\cong (\partial_-\Gamma\times I) \cup \big((k-1)$
  $1$-handles$\big)$, $E_R$ has a handle decomposition with $(k-1)$
  $2$-handles as claimed.  Since
  $S\cong (C_-\times I)\cup (1$-handle$)$, $E_S$ has a handle
  decomposition with a $2$-handle as claimed.

  Alternatively, one may use
  Lemma~\ref{lemma:general-handles-at-critical-level} similarly to the
  proofs of
  Propositions~\ref{proposition:surface-stage-regnbd-complement}
  and~\ref{proposition:plumbed-disk-regnbd-complement}.  The attaching
  circles for each case are shown in
  Figure~\ref{figure:punctured-sheet-level-reg-nbhd}.
\end{proof}

\begin{figure}[H]
  \begin{tikzpicture}[x=1bp,y=1bp]
    \small
    \node[anchor=south west, inner sep=0mm]
    {\includegraphics{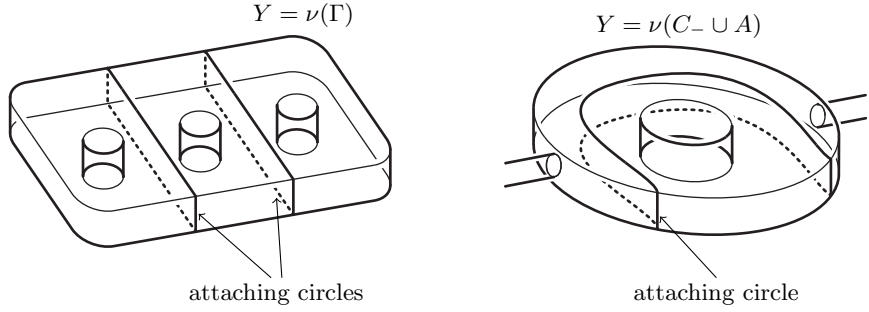}};
    \node [above] at (110,80) {$Y = \nu(\Gamma)$};
    \draw [->] (98,-11)--(71,13);
    \draw [->] (102,-11)--(100,19);
    \node [below] at (100,-10) {attaching circles};
    \node [above] at (250,76) {$Y=\nu(C_-\cup A)$};
    \draw [->] (264,-11)--(243,12);
    \node [below] at (264,-10) {attaching circle};
  \end{tikzpicture}
  \caption{Regular neighborhoods of $\Gamma$ and $C_-\cup A$.}
  \label{figure:punctured-sheet-level-reg-nbhd}
\end{figure}

We are now almost ready for the proof of
Theorem~\ref{theorem:composition-grope-handle-decomposition}.  For
clarity in the proof we will use the following definition.

\begin{definition}
  \label{definition:admissible-complex}
  Suppose $M$ is a 3-manifold.  We say that a 2-complex $G$ in
  $M\times I$ is \emph{$ABRS$-admissible} if for some
  $0<t_1<\cdots<t_r<1$ and $\epsilon >0$ the following hold.
  \begin{enumerate}
  \item $G\cap (M\times J)$ is an isotopy of a 1-submanifold in $M$
    for each subinterval $J=[0,t_1-\epsilon]$,
    $[t_1+\epsilon,t_{2}-\epsilon],\ldots,[t_{r-1}+\epsilon,t_{r}-\epsilon]$,
    and $[t_r+\epsilon,1]$.
  \item For each $i$, $G\cap (M\times [t_i-\epsilon,t_i+\epsilon])$ is
    the disjoint union of $L\times [t_i-\epsilon,t_i+\epsilon]$ and a
    2-complex $Z$, where $L$ is a link in $M$ and $Z$ is of the form
    of either $A$, $B$, $R$, or $S$ in
    Propositions~\ref{proposition:surface-stage-regnbd-complement},
    \ref{proposition:plumbed-disk-regnbd-complement},
    \ref{proposition:punctured-disk-regnbd-complement}
    and~\ref{proposition:punctured-vertical-sheet-regnbd-complement}.
  \end{enumerate}
  We call each $M\times t_i$ a \emph{critical level of type $A$, $B$,
    $R$, or~$S$}.
\end{definition}

The following are basic observations on the $ABRS$-admissibility of
capped gropes.

\begin{enumerate}
\item For a height $n$ satellite capped grope $G$ in $S^3\times I$
  obtained by pushing a 3D satellite capped grope $G_0$ for
  $(K,\alpha)$, we may assume that
  $G\cup (K\times I) \subset E_{\alpha}\times I$ is $ABRS$-admissible;
  for instance, choose sufficiently small $\epsilon>0$, and push a
  height $k$ surface $S\subset E_\alpha=E_\alpha\times 0$ in $G_0$ to
  $(\partial S\times[(k-1)\epsilon,k\epsilon]) \cup (S\times
  k\epsilon) \subset E_\alpha\times I$, $1\le k \le n$. For a cap $C$
  in $G_0$, which intersects $m$ times $K$ transversely, let $\Gamma$
  be the corresponding punctured cap. That is, $\Gamma$ is a planar
  surface contained in $C$ such that $\Gamma$ does not intersect $K$
  and $\partial\Gamma=\partial_-\Gamma\cup \partial_+\Gamma$ where
  $\partial_-\Gamma= \partial C$ and $\partial_+\Gamma$ are $m$
  meridional curves $\cup_{i=1}^m\mu_i$ of $K$. Then push $C$ to
  \[
  (\partial_-\Gamma\times[n\epsilon,(n+1)\epsilon]) \cup (\Gamma\times (n+1)\epsilon) \cup  (\partial_+\Gamma\times [(n+1)\epsilon,(n+2)\epsilon])\cup A  
  \]
 in $ E_\alpha\times I$ where 
 \[
 A=\bigcup_{i=1}^m ((\mu_i\times[(n+2)\epsilon, (n+2+i)\epsilon])\cup (D_i\times (n+2+i)\epsilon)))
 \]
 for a (planar) meridional disk $D_i$ with boundary $\mu_i$. Push the remaining caps further in a similar way. In this case, we use types $A$, $B$, and $R$.
\item We may assume that a capped grope concordance in $S^3\times I$
  obtained by pushing a 3D capped grope concordance is
  $ABRS$-admissible, by a similar isotopy.  Critical levels have types
  $A$, $B$, $R$, and~$S$.
\item If $G$ and $H$ are satellite capped gropes for $(K,\alpha)$ and
  $(J,\beta)$ such that
  \[
    G\cup (K\times I) \subset E_{\alpha}\times I \text{\quad and\quad}
    H\cup (J\times I) \subset E_{\beta}\times I
  \]
  are $ABRS$-admissible,
  then
  \[
    (G\cdot H) \cup (K(\alpha,J)\times I) \subset E_{\beta}\times I
  \]
  is $ABRS$-admissible.  It is verified straghtforwardly by inspecting
  Definition~\ref{definition:composition-satellite-capped-gropes}.  In
  this case, we use type $R$ when we attach gropes to a punctured cap
  of another grope.
\item If $G$ is a satellite capped grope for $(K,\alpha)$ such that $G\cup
  (K\times I) \subset E_{\alpha}\times I$ is $ABRS$-admissible and $H$
  is an $ABRS$-admissible capped grope concordance in $S^3\times I$,
  then $G\cdot H$ is an $ABRS$-admissible capped grope concordance.  It
  is verified straghtforwardly by inspecting
  Definition~\ref{definition:composition-satellite-grope-and-grope-concordance}.
  In general, we need all the types $A$, $B$, $R$, and~$S$.
\end{enumerate}

\begin{proof}[Proof of
  Theorem~\ref{theorem:composition-grope-handle-decomposition}]
  
  Repeatedly applying the above observations, it follows that the
  given product $G=G_1\cdots G_n\cdot H$ is an $ABRS$-admissible capped
  grope concordance in $S^3\times I$.  This is true for any choice of
  parenthesization of the product.  Near each critical level
  $S^3\times t_i$, apply one of
  Propositions~\ref{proposition:surface-stage-regnbd-complement},
  \ref{proposition:plumbed-disk-regnbd-complement},
  \ref{proposition:punctured-disk-regnbd-complement},
  and~\ref{proposition:punctured-vertical-sheet-regnbd-complement} to
  obtain 2-handles.  (Here, writing $G\cap (S^3
  \times[t_i-\epsilon,t_i+\epsilon]) =
  (L\times[t_i-\epsilon,t_i+\epsilon]) \cup Z$ as in
  Definition~\ref{definition:admissible-complex}, we apply the
  proposition to $Z$ in the product
  $E_L\times[t_i-\epsilon,t_i+\epsilon]$.)  Stacking them, we obtain a
  desired handle decomposition of $E_G$ with only 2-handles added to $E_K\times I$.
\end{proof}

\section{Knots in $S^3$ sliced by unknotted Whitney towers and gropes}
\label{section:doubly-slicing-by-knots}

\subsection{Whitney tower and grope bi-filtrations}

We begin by recalling the classical notion of doubly slice knots.  In
what follows, regard $S^3\subset S^4$ in the standard way.

\begin{definition}
  A knot $K$ in $S^3$ is \emph{doubly slice} if there is an unknotted
  flat 2-sphere embedded in $S^4$ which intersects $S^3$
  transversely at~$K$.
\end{definition}

Considering our Whitney tower and grope generalizations of unknotted
2-spheres, we are naturally led to the following generalization.
Denote by $D^4_+$ and $D^4_-$ the upper and lower hemispheres of $S^4$
bounded by $S^3$.

\begin{definition}\label{definition:Whitney tower bi-filtration}

  \begin{enumerate}
  \item We say that a knot $K$ in $S^3$ is a \emph{slice of a Whitney
      tower $T$ in $S^4$} if the base sphere of $T$ intersects $S^3$
    transversely and $K=T\cap S^3$.  In this case, each
    $T_\pm := T\cap D^4_\pm$ is a disk-like Whitney tower bounded
    by~$K$.
  \item For half-integers $m,n\ge 1$, a knot $K$ in $S^3$ is a
    \emph{height $(m,n)$ Whitney slice} if it is a slice of a
    $\pi_1$-unknotted sphere-like Whitney tower $T$ in $S^4$ such that
    $T_+$ and $T_-$ have height $m$ and $n$ respectively.  We denote
    by $\cW_{m,n}$ the set of height $(m,n)$ Whitney slice knots.  A
    knot $K$ is \emph{height $m$ Whitney doubly slice} if
    $K\in \cW_{m,m}$.
  \end{enumerate}
  Define a \emph{height $(m,n)$ grope slice knot} and a \emph{height
    $m$ grope doubly slice knot} by replacing Whitney towers with
  capped gropes.  Denote by $\cG_{m,n}$ the set of height $(m,n)$
  grope slice knots.
\end{definition}

Using a Seifert-van Kampen argument one can show that $\cW_{m,n}$ and
$\cG_{m,n}$ are closed under connected sum, that is, $\cW_{m,n}$ and
$\cG_{m,n}$ are submonoids of the monoid of knots.  By
Propositions~\ref{proposition:unknottedness-of-subtower} and
\ref{proposition:unknottedness-of-contraction}, we have
$\cW_{k,\ell}\subset \cW_{m,n}$ and $\cG_{k,\ell}\subset \cG_{m,n}$
for $k\ge m$ and $\ell\ge n$. Therefore we obtain bi-filtrations
$\{\cW_{m,n}\}$ and $\{\cG_{m,n}\}$ of the monoid of knots, which we
call the \emph{Whitney tower bi-filtration} and the \emph{grope
  bi-filtration}, respectively.

The Whitney tower and grope cases are related as follows in this
context.

\begin{theorem}
  \label{theorem:grope-bifiltration-contained-in-whitney-bifiltration}
  A height $(m,n)$ grope slice knot is height $(m,n)$ Whitney slice.
  That is, $\cG_{m,n} \subset \cW_{m,n}$.
\end{theorem}

\begin{proof}
  Suppose $K$ is a slice of a $\pi_1$-unknotted capped grope $G$ in
  $S^4$ and $G_+ := G\cap D^4_+$ and $G_-:=G\cap D^4_-$ have height $m$
  and $n$ respectively.  Apply
  Theorem~\ref{theorem:grope-whitney-tower-transformation} to
  transform $G$ to a $\pi_1$-unknotted Whitney tower~$T$.  Since the
  disk-like Whitney towers $T_\pm$ are obtained from the capped gropes
  $G_\pm$ (see Remark~\ref{remark:base-surface-transformation}), it
  follows that $T_+$ and $T_-$ have height $m$ and $n$ respectively,
  as in Corollary~\ref{corollary:unknotted-tower-from-grope-height-h}.
\end{proof}

\subsection{Construction of examples}\label{subsection:construction-of-examples}

In this subsection, using the results of
Section~\ref{section:handle-decomposition-grope-exterior} we construct certain examples of knots which are
height $(m,n)$ grope slice, and consequently height $(m,n)$ Whitney
slice by
Theorem~\ref{theorem:grope-bifiltration-contained-in-whitney-bifiltration}.  In the
next section, a particular subfamily of those examples will be shown
to be not doubly slice.  Furthermore the examples will be used to
exhibit the rich structure of the Whitney tower and grope
bi-filtrations.

We start with the following input data.  Fix nonnegative integers
$m\ge n$.  Suppose $J_0$ is a knot, and for each $k=0,\ldots,m-2$,
suppose $(K_k,\eta_k)$ is a pair of a knot $K_k$ and a simple closed
curve $\eta_k$ in the exterior $E_{K_k}$ which is unknotted in~$S^3$.
Suppose the following:

\begin{enumerate}[label=(G\arabic{*})]
\item\label{item:height-2-grope-for-J_0} There is a 3D capped grope
  concordance of height $2$ between $J_0$ and the unknot.
\item\label{item:height-1-grope-for-K_k-eta_k} The knot $K_k$ is
  ribbon and there is a 3D satellite capped grope of height $1$ for
  $(K_k,\eta_k)$ for each~$k$.
\end{enumerate}

We remark that there are numerous examples satisfying the above.  We
will specify explicit choices in later sections.

Let $\cR$ be the knot $9_{46}$, and let $\alpha'$ and $\beta'$ be the
curves depicted in Figure~\ref{figure:knot 9_46_2} (see \cite[Figure~7]{Horn:2010-1}).
Define a  knot $J_{k+1}$ inductively for $k\ge 0$ by
$J_{k+1} := K_k(\eta_k;J_k)$.  Finally define
\[
  J_{m,n} = \cR(\alpha',\beta';J_{m-1},J_{n-1}) :=
  \cR(\alpha',J_{m-1})(\beta',J_{n-1})
\]
to be the knot obtained by applying the satellite construction twice,
once along $\alpha'$ using $J_{m-1}$ as the companion and then along
$\beta'$ using $J_{n-1}$ as the companion.

\begin{figure}[t]
  \begin{tikzpicture}[x=1pt, y=1pt, scale=.95]
    \small
    \node[anchor=south west, inner sep=0, transform shape] {\includegraphics{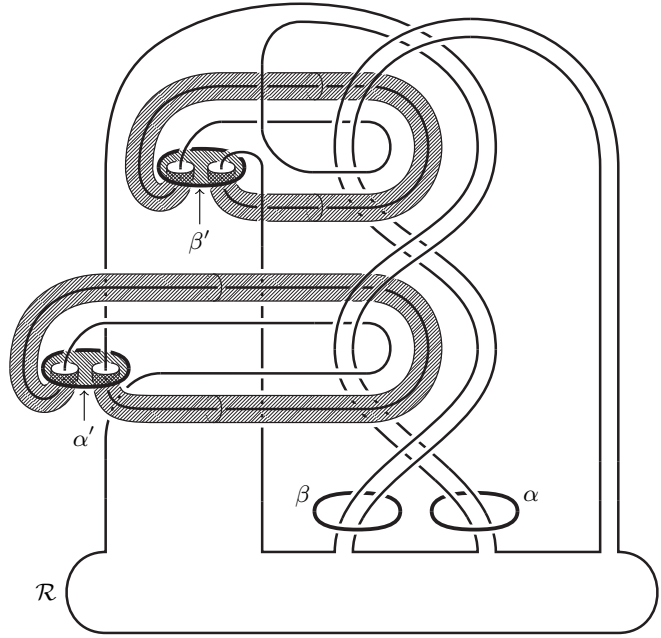}};
    \path (24,17) node[left]{$\cR$} (124,55) node[left]{$\beta$} (201,55) node[right]{$\alpha$};
    \draw[->, thin] (31.5,85) node[below,yshift=1]{$\alpha'$} -- (31.5,97);
    \draw[->, thin] (77,163) node[below,yshift=1]{$\beta'$} -- (77,177); 
  \end{tikzpicture}
  \caption{The knot $\cR=9_{46}$ and 3D satellite capped gropes
    bounded by $\alpha'$ and~$\beta'$.
    The curves $\alpha$ and $\beta$
    are dual to the left and right handles of the genus one Seifert
    surface.
  }
  \label{figure:knot 9_46_2}
\end{figure}

\begin{theorem}\label{theorem:grope-doubly-slice-construction} 
  If \ref{item:height-2-grope-for-J_0} and
  \ref{item:height-1-grope-for-K_k-eta_k} hold, then the knot
  $J_{m,n}$ is height $(m+2,n+2)$ grope slice, and consequently
  height $(m+2,n+2)$ Whitney slice by
  Theorem~\ref{theorem:grope-bifiltration-contained-in-whitney-bifiltration}.
\end{theorem}

\begin{proof}
  Let $J:=J_{m,n}$ for brevity.  First, we construct a disk-like
  capped grope of height $m+2$ with boundary $J$ in~$D^4_+$.

  In what follows $U$ designates a trivial knot.  Let
  \[
    \cR_{\beta'} := \cR(\alpha',\beta';U,J_{n-1}) = \cR(\beta';
    J_{n-1}).
  \]
  Then we can view $J$ as $J = \cR_{\beta'}(\alpha';J_{m-1})$.  Let
  $L_0=U$ and $L_{k+1}=K_k(\eta_k; L_k)$ for $k\ge 0$.  Let
  \[
    \cR_+:= \cR(\alpha',\beta';L_{m-1},J_{n-1}) =
    \cR_{\beta'}(\alpha';L_{m-1}).
  \]

  Figure~\ref{figure:knot 9_46_2} depicts a 3D satellite capped grope
  for~$(\cR,\alpha')$.  After the satellite operation along $\beta'$,
  it becomes a 3D satellite capped grope of height $1$ for
  $(\cR_{\beta'},\alpha')$.  Push this into $S^3\times I$ to obtain a
  satellite capped grope, which we denote by~$H_{m-1}$.  Push those
  given in~\ref{item:height-2-grope-for-J_0}
  and~\ref{item:height-1-grope-for-K_k-eta_k} to obtain a height $2$
  capped grope concordance $H$ between $J_0$ and $U$, and a height $1$
  satellite capped grope $H_k$ of $(K_k,\eta_k)$ for $k=0,\ldots,m-2$.
  Let
  \[
    P_+:=H_{m-1}\cdot H_{m-2}\cdots H_0\cdot H.
  \]
  This is a capped grope concordance of height $m+2$ between $J$
  and~$\cR_+$.  By
  Theorem~\ref{theorem:composition-grope-handle-decomposition}, the
  exterior $E_{P_+}$ has a handle decomposition
  $E_{P_+}\cong (E_J\times I) \cup \text{($2$-handles)}$.  Also, the
  construction of the product (see
  Section~\ref{subsection:product-of-gropes}) tells us that
  $\beta \subset E_J$ is isotopic, in $E_{P_+}$, to
  $\beta \subset E_{\cR_+}$. The same conclusion holds for (a parallel
  of) $\beta'$ as well.

  Since each $K_k$ is ribbon and $L_0$ is the unknot, it follows that
  each $L_k$ is ribbon by an induction.  Therefore, there is a ribbon
  concordance, say $Q_+$, from $\cR_+=\cR_{\beta'}(\alpha';L_{m-1})$
  to $\cR_{\beta'}=\cR_{\beta'}(\alpha';U)$.

  The obvious genus one Seifert surface for $\cR$ (see
  Figure~\ref{figure:knot 9_46_2}) becomes a genus one Seifert surface
  for $\cR_{\beta'}$ after the satellite operation.  If we cut the
  right band (whose linking circle is the curve $\beta$) of the Seifert
  surface by attaching a band, $\cR_{\beta'}$ becomes a 2-component
  trivial link.  This gives us a genus zero cobordism between
  $\cR_{\beta'}$ and a 2-component trivial link.  Cap off the latter
  to obtain a ribbon disk $\Delta_+$ in $D^4$ for $\cR_{\beta'}$.  (In
  this paper a ribbon disk in $D^4$ designates a slice disk obtained
  by pushing an immersed ribbon disk in~$S^3$.)  Note that (parallels
  of) $\beta$ and $\beta'$ in $E_{\cR_{\beta'}}$ are null-homotopic
  in~$E_{\Delta_+}$.

  Identify a collar of $S^3$ in the upper hemisphere $D^4_+$ with
  $S^3\times[0,2]$, so that $D^4_+= S^3\times[0,2] \cup D^4$ where
  $S^3\times 0$ is the boundary of~$D^4_+$.  View $P_+$, $Q_+$,
  and $\Delta_+$ as subsets of $S^3\times[0,1]$, $S^3\times[1,2]$ and
  $D^4$ respectively.  Let
  \[
    G_+:= P_+ \cupover{\cR_+} Q_+ \cupover{\cR_{\beta'}} \Delta_+ \subset 
    D^4_+.
  \]
  Then $G_+$ is a disk-like capped grope of height $m+2$ in $D^4_+$ with
  boundary~$J$.  The following hold:
  \begin{enumerate}[label=(\roman{*})$_+$]
  \item The exterior $E_{G_+}$ is obtained from $E_J\times I$ by
    attaching handles of index $\ge 2$.
  \item The curves $\beta$ and $\beta' \subset E_J$ are null-homotopic
    in~$D^4_+\setminus G_+$.
  \end{enumerate}
  The condition (i)$_+$ follows from that $E_{G_+}$ is obtained by
  stacking $E_{P_+}$, $E_{Q_+}$, and $E_{\Delta_+}$, each of which has
  handles of index $\ge 2$ only.  The condition (ii)$_+$ for $\beta$
  holds since $\beta \subset E_J$ is isotopic to
  $\beta \subset E_{\cR_{\beta'}}$ in $E_{P_+}\cup E_{Q_+}$ and
  $\beta \subset E_{\cR_{\beta'}}$ is null-homotopic
  in~$E_{\Delta_+}$.  The same argument works for $\beta'$ as well.
 
  Similarly to the above, construct a disk-like capped grope $G_-$ of height
  $n+2$ in $D^4_-$ with boundary $J$ which satisfies the following:
  \begin{enumerate}[label=(\roman{*})$_-$]
  \item The exterior $E_{G_-}$ is obtained from $E_J\times I$ by
    attaching handles of index $\ge 2$.
  \item The curves $\alpha$ and $\alpha' \subset E_J$ are
    null-homotopic in~$D^4_-\setminus G_-$.
  \end{enumerate}

  Finally, let $G:=G_-\cup_J G_+$ in $S^4=D^4_-\cup_{S^3} D^4_+$.
  Then $G$ is a sphere-like capped grope in $S^4$ which intersects
  $S^3$ at the knot~$J$.  The only remaining thing to show is that $G$
  is $\pi_1$-unknotted, that is, $\pi_1(S^4\setminus G)\cong \Z$.

  By Seifert-van Kampen,
  \[
    \pi_1(S^4\setminus G)\cong \pi_1(E_{G_-})
    \mathbin{\mathop{*}_{\pi_1(E_J)}} \pi_1(E_{G_+}).
  \]
  From~(i)$_\pm$, it follows that each inclusion-induced map
  $\pi_1 E_J \to \pi_1E_{G_\pm}$ is surjective.  Therefore
  $\pi_1(S^4\setminus G)$ is a quotient of $\pi_1E_J$.
  From~(ii)$_\pm$, it follows that $\alpha, \alpha', \beta$, and
  $\beta'$ are null-homotopic in $\pi_1(S^4\setminus G)$.  Therefore
  $\pi_1(S^4\setminus G)$ is a quotient of
  $\pi_1E_J/\langle \alpha, \alpha', \beta, \beta'\rangle$ where
  $\langle{}\cdots{}\rangle$ denotes the normal subgroup generated
  by~$\cdots$.

  Recall that $\alpha'$ and $\beta'$ are identified with the meridians
  of $J_{m-1}$ and $J_{n-1}$ respectively.  By Seifert-van Kampen, it
  follows that
  \[
    \pi_1E_J/\langle \alpha',\beta'\rangle \cong
    \pi_1E_J/\langle\Im\{\pi_1E_{J_{m-1}}, \pi_1E_{J_{n-1}}\}\rangle
    \cong \pi_1E_\cR/\langle \alpha',\beta'\rangle.
  \]
  Consequently
  $\pi_1E_J/\langle \alpha, \alpha', \beta, \beta'\rangle$ is
  isomorphic to
  $\pi_1E_\cR/\langle \alpha, \alpha', \beta, \beta'\rangle$, which is
  a quotient of $\Z$ since
  $\pi_1E_\cR/\langle \alpha, \beta\rangle\cong \Z$.

  Therefore, $\pi_1(S^4\setminus G)$ is a quotient of $\Z$.  Since
  $H_1(S^4\setminus G)\cong \Z$, it follows that
  $\pi_1(S^4\setminus G)\cong \Z$.
\end{proof}

\section{Obstructions to grope and Whitney doubly slicing}
\label{section:obstructions}

In this section, we give examples illustrating the rich structure of the grope and Whitney tower bi-filtrations. The main result of this section is the following:

\begin{theorem}
  \label{theorem:nontriviality-of-bi-filtrations}
  \leavevmode\Nopagebreak
  \begin{enumerate}
  \item For any integer $m\ge 3$, there exists a family of slice knots
    $\{J^i\}_{i=1,2\ldots}$ such that each $J^i$ is height $m$
    grope doubly slice but any nontrivial linear combination of the
    $J^i$ is not height $m.5$ Whitney doubly slice.
  \item For any integers $m, n\ge 3$, there exists a family of knots
    $\{J_{m,n}^i\}_{i=1,2,\ldots}$ such that each $J_{m,n}^i$ is height
    $(m,n)$ grope slice but any nontrivial linear combination of the
    $J_{m,n}^i$ is not height $(m.5,n)$ Whitney slice and not height
    $(m,n.5)$ Whitney slice.
  \end{enumerate}
\end{theorem}

The above theorem is proved in a stronger form in Theorem~\ref{theorem:infinite-rank}. Since height $(m,n)$ grope slice knots are height $(m,n)$ Whitney
slice
(Theorem~\ref{theorem:grope-bifiltration-contained-in-whitney-bifiltration}),
it follows that both grope and Whitney tower bi-filtrations $\{\cG_{m,n}\}$
and $\{\cW_{m,n}\}$ are highly nontrivial.

The examples $J^i$ and $J_{m,n}^i$ in
Theorem~\ref{theorem:nontriviality-of-bi-filtrations} will be obtained
by the construction in
Section~\ref{subsection:construction-of-examples}.
Theorem~\ref{theorem:grope-doubly-slice-construction} will tell us the
promised grope/Whitney sliceness.  

To show that the examples and their linear combinations are \emph{not} grope/Whitney slice, we will first relate the grope/Whitney sliceness to the
\emph{$(m,n)$-solvability} stuided in \cite{Kim:2006-1} in Section~\ref{subsection: (m,n)-solvable knot}, and then use
\emph{amenable $\rhot$-invariant} techniques developed in
\cite{Cha-Orr:2009-1,Cha:2010-1,Cha:2012-1} in Sections \ref{subsection:amenable-signature}, \ref{subsection:nontriviality-bifiltration}, and \ref{subsection:proofs-of-lemma-and-proposition}.

\subsection{$(m,n)$-solvable knots}\label{subsection: (m,n)-solvable knot}

In \cite{Kim:2006-1}, the second author introduced the notion of
\emph{$(m,n)$-solvable} knots which is a generalization of a doubly
slice knot.  Since its definition is lengthy, we will first state the
main result of this subsection and then describe necessary terms.

\begin{theorem}
  \label{theorem:comparison of bi-filtrations}
  A height $(m+2,n+2)$ Whitney slice knot is $(m,n)$-solvable.
\end{theorem}

To define an $(m,n)$-solvable knot, we need to recall the notion of an
\emph{$n$-solvable knot} defined in \cite{Cochran-Orr-Teichner:1999-1}, from
which the former was inspired.  For a group $G$ and $n\ge 0$, the
\emph{$n$th derived subgroup $G^{(n)}$} is defined by
$G^{(0)}:=G$ and $G^{(n)}:=[G^{(n-1)},G^{(n-1)}]$ for $n\ge 1$.

\begin{definition}[{\cite{Cochran-Orr-Teichner:1999-1}}]
\label{definition:n-solution}
  Suppose $n$ is a nonnegative integer, $M$ is a closed 3-manifold
  with $H_1(M) \cong \Z$, and $W$ is a compact spin 4-manifold with
  boundary~$M$.  Let $\pi:=\pi_1(W)$.
  \begin{enumerate}
  \item $W$ is an \emph{$n$-solution} for $M$ if the inclusion
    induces an isomorphism $H_1(M)\cong H_1(W)$ and there exist
    $[x_1],\ldots, [x_r]$ and $[y_1],\ldots, [y_r]$ in
    $H_2(W;\Z[\pi/\pi^{(n)}])$ with $r=\frac12\rank_\Z H_2(W)$ such
    that the intersection form and self-intersection form
    \begin{gather*}
      \lambda_n\colon H_2(W;\Z[\pi/\pi^{(n)}])\times
      H_2(W;\Z[\pi/\pi^{(n)}]) \to \Z[\pi/\pi^{(n)}]
      \\
      \mu_n\colon H_2(W;\Z[\pi/\pi^{(n)}]) \to \Z[\pi/\pi^{(n)}]/\{a-\bar
      a\}
    \end{gather*}	
    have values $\lambda_n([x_i],[x_j])=\mu_n([x_i])=0$ and
    $\lambda_n([x_i],[y_j])=\delta_{ij}$ for all $i$ and~$j$. 
  \item $W$ is an \emph{$n.5$-solution} for $M$ if $W$ satisfies (1)
    and in addition the $[x_i]$ have lifts
    $[\tilde{x_i}]\in H_2(W;\Z[\pi/\pi^{(n+1)}])$ such that
    $\lambda_{n+1}([\tilde{x_i}],[\tilde{x_j}])=\mu_{n+1}([\tilde{x_i}])=0$
    for all $i$ and~$j$.  
  \item A knot $K$ is \emph{$h$-solvable}
    ($h\in \frac12 \Z_{\ge 0}$) if the zero surgery $M(K)$ has an
    $h$-solution.
  \end{enumerate}   
\end{definition}

We denote by $\F_n$ the subgroup of the concordance classes of
$n$-solvable knots.  Since an $m$-solvable knot is $n$-solvable
for $m\ge n$, the subgroups $\F_n$ with $n\in \frac12 \Z_{\ge 0}$ form a descending filtration of the knot concordance group.

\begin{definition}[{\cite[Definitions 2.2, 2.3, and 6.2]{Kim:2006-1}}]
\label{definition:(m,n)-solution}
  Suppose $m,n\in \frac12\Z_{\ge 0}$ and $M$ is a closed 3-manifold with
  $H_1(M)\cong \Z$.
  \begin{enumerate}
  \item A pair $(W_1,W_2)$ of compact 4-manifolds $W_i$ bounded by~$M$
    is an \emph{$(m,n)$-solution} for $M$ if $W_1$ is an
    $m$-solution for $M$, $W_2$ is an $n$-solution for $M$, and
    $\pi_1(W_1\cup_M~W_2)$ is an infinite cyclic group.
  \item A knot $K$ is \emph{$(m,n)$-solvable} if there is an
    $(m,n)$-solution $(W_1,W_2)$ for the zero surgery~$M(K)$.
  \item We call an $(m,m)$-solution a \emph{double $m$-solution}.
    We also call an $(m,m)$-solvable knot or 3-manifold \emph{doubly
      $m$-solvable}.
  \end{enumerate}
\end{definition}

We denote by $\F_{m,n}$ the set of (isotopy classes of)
$(m,n)$-solvable knots. By \cite[Proposition~2.6]{Kim:2006-1} each
$\F_{m,n}$ is a submonoid under connected sum, and since an
$m$-solution for a knot $K$ is an $n$-solution for $m\ge n$, it is
obvious that $\F_{k,\ell}\subset \F_{m,n}$ for $k\ge m$ and
$\ell\ge n$. Therefore $\{\F_{m,n}\}$ is a bi-filtration of the monoid
of knots. We call it \emph{the solvable bi-filtration of
  knots}. Theorems~\ref{theorem:grope-bifiltration-contained-in-whitney-bifiltration}
and~\ref{theorem:comparison of bi-filtrations} tell us that
$\cG_{m+2,n+2} \subset \cW_{m+2,n+2} \subset \F_{m,n}$.

Due to~\cite[Theorem~8.11]{Cochran-Orr-Teichner:1999-1}, a height
$h+2$ Whitney tower in $D^4$ bounded by a knot $K$ in $S^3$ can be
transformed to an $h$-solution for~$K$.  Here we show that this can be
done in such a way that the fundamental group does not grow.

\begin{theorem}
  \label{theorem:whitney-tower-to-solution-fundamental-group}
  Suppose $T$ is a disk-like Whitney tower of height $h+2$ in $D^4$
  bounded by a knot $K$ in~$S^3$.  Then there is an $h$-solution $W$
  for $K$ with an epimorphism
  $\phi\colon \pi_1(D^4\setminus T)\to \pi_1(W)$ making the following
  diagram commute:
  \[ \xymatrix{ 
    \pi_1(S^3\setminus K)	\ar[rd]^{j_*} \ar[d]_{i_*}\\
    \pi_1(D^4\setminus T) \ar[r]_{\phi}&
    \pi_1(W),    }
  \]
where $i_*$ and $j_*$ are homomorphisms induced from inclusions. 
\end{theorem}

Our construction of the $h$-solution in
Theorem~\ref{theorem:whitney-tower-to-solution-fundamental-group} is
slightly different from that in ~\cite[Theorem 8.12]{Cochran-Orr-Teichner:1999-1}.

\begin{proof}
  Write $h=n$ or $n.5$ with $n$ an integer.  The top stage disks of
  $T$ have height $n+2$ if $h=n$, height $n+3$ if $h=n.5$.

  For each Whitney disk $D_i$ of height $n+2$, choose a parallel copy
  of $\partial D_i$ in the interior of $D_i$, and perform surgery on
  $D^4$ along the parallel copy (using the framing induced by~$D_i$).
  Let $V$ be the 4-manifold obtained by surgery.  Using the surgery
  core disk as an embedded Whitney disk for the Whitney
  circle~$\partial D_i$, do Whitney moves to make height $n+1$ Whitney
  disks embedded in~$V$.  Repeating this inductively for height
  $n,\ldots,2,1$, the base disk of $T$ is changed to an embedded disk
  $\Delta$ in~$V$.  Let $W=V\setminus \nu(\Delta)$, the exterior
  of~$\Delta$ in $V$.

  Let $T_0$ be the subtower of $T$ which consists of height $\le n+1$
  Whitney disks, and let $T'$ be the Whitney tower of height $n+2$ in
  $V$ obtained from $T_0$ by adding the surgery core disks as height
  $n+2$ disks.  By
  Proposition~\ref{proposition:unknottedness-of-subtower},
  $\pi_1(D^4\setminus T_0)$ is a quotient of $\pi_1(D^4\setminus T)$.
  A Seifert-van Kampen argument shows that
  $\pi_1(D^4\setminus T_0) \cong \pi_1(V\setminus T')$.  By applying
  Lemma~\ref{lemma:finger-move-fundamental-group} repeatedly,
  $\pi_1(V\setminus T') \cong \pi_1(V\setminus \Delta)$.  It follows
  that $\pi_1(V\setminus \Delta)$ is a quotient of
  $\pi_1(D^4\setminus T)$.

  We claim that $W$ is an $h$-solution for~$K$.  The spin structure of
  $D^4$ gives rise to a spin structure on~$W$.  Since the base disk of
  $T$ induces the zero framing on its boundary $K$ and Whitney moves
  do not alter this framing, we have $\partial W=M(K)$, the zero surgery on
  $K$ in $S^3$.  Since the
  surgery on $D^4$ is performed along null-homotopic curves,
  $H_2(V)\cong \Z^{2k}$ where $k$ is the number of height $n+2$
  Whitney disks~$D_i$.  From the long exact sequence for $(V,W)$, it
  follows that $H_2(W)\cong H_2(V)\cong \Z^{2k}$.  The union of
  $D_i\setminus ($collar of $\partial D_i)$ and the surgery core disk
  bounded by $\partial D_i$ is an immersed sphere in $W$, which we
  denote by~$S_i$.  Choose a Clifford torus $T_i$ around one of the
  two intersections paired by the Whitney disk~$D_i$.  Then
  $\{S_i,T_i\}_{i=1,\ldots,k}$ is a basis for $H_2(W)$.  A basis curve
  on $T_i$ is a meridian $\mu$ of a height $n+1$ disk~$D$.  A Clifford
  torus $T$ around an intersection paired (with another intersection)
  by $D$ meets $D$ at one point, and the basis curves of $T$ are
  meridians of height $n$ disks.  It follows that $\mu$ is a
  commutator of meridians of height $n$ disks.  Repeating the same
  argument inductively for height $n,\ldots,2,1$, it follows that
  $\mu$ is an element in the $n$th derived subgroup~$\pi^{(n)}$, where
  $\pi=\pi_1(W)$.  Thus the Clifford torus $T_i$ lifts to the
  $\pi/\pi^{(n)}$-cover of $W$ and represents a homology class
  $[T_i]\in H_2(W;\Z[\pi/\pi^{(n)}])$.

  To complete the proof, we now consider the two cases of $h=n$ and
  $h=n.5$ separately.  For $h=n$, since the tori $T_i$ are framed and
  pairwise disjoint, $\mu_n([T_i]) = 0$ and $\lambda_n([T_i],[T_j])=0$
  for any $i, j$.  Since $S_i$ meets $T_i$ at one point and disjoint
  from $T_j$ for $j\ne i$, $\lambda_n([T_i],[S_j])=\delta_{ij}$. 
  
  For $h=n.5$, consider the homology classes
  $[S_i]\in H_2(W;\Z[\pi/\pi^{(n+1)}])$. Since the sphere $S_i$ are framed, $\mu_n([S_i])=0$. Recall that the
  intersections of the height $n+2$ disks are paired up by height
  $n+3$ disks in~$D^4$.  Since the height $n+3$ disks are disjoint
  from height $\le n+1$ disks, the height $n+3$ disks lie in $V$ and
  pair up intersections of the immersed spheres~$S_i$.  It follows
  that $\lambda_{n+1}([S_i], [S_j])=0$ for any $i,j$.  As above,
  $\lambda_n([S_i],[T_j])=\delta_{ij}$.  
\end{proof}

Now we are ready for the proof of Theorem~\ref{theorem:comparison of
  bi-filtrations}.

\begin{proof}[Proof of Theorem~\ref{theorem:comparison of bi-filtrations}]
  Suppose a knot $K$ in $S^3$ is sliced by a $\pi_1$-unknotted Whitney
  tower $T$ in~$S^4$ whose upper half $T_+=T\cap D^4_+$ and lower half
  $T_-=T\cap D^4_-$ have height $m+2$ and $n+2$ respectively.  Apply
  Theorem~\ref{theorem:whitney-tower-to-solution-fundamental-group} to
  $T_+$ and $T_-$ to obtain an $m$-solution $W_+$ and an $n$-solution
  $W_-$ such that $\pi_1(W_+)$ and $\pi_1(W_-)$ are homomorphic images
  of $\pi_1(D^4_+\setminus T_+)$ and $\pi_1(D^4_-\setminus T_-)$
  respectively.  Since
  \[
    \Z\cong \pi_1(S^4\setminus T) \cong \pi_1(D^4_+\setminus T_+)
    \mathbin{\mathop{*}_{\pi_1(S^3\setminus K)}}\pi_1(D^4_-\setminus
    T_-),
  \]
  it follows that
  \[
    \pi_1\Big(W_+\mathbin{\mathop{\cup}_{M(K)}} W_-\Big) \cong
      \pi_1(W_+) \mathbin{\mathop{*}_{\pi_1 M(K)}}\pi_1(W_-)
  \]
  is a quotient of~$\Z$.  Since
  $H_1(W_+\mathbin{\mathop{\cup}_{M(K)}} W_-)=\Z$, it follows that
  $\pi_1(W_+\mathbin{\mathop{\cup}_{M(K)}} W_-)=\Z$.  Therefore
  $(W_+,W_-)$ is an $(m,n)$-solution.
\end{proof}

\subsection{Amenable signature theorem and mixed commutator series}
\label{subsection:amenable-signature}

We recall some results on the von Neumann-Cheeger-Gromov
$\rhot$-invariant for a closed 3-manifold and mixed commutator series
in \cite{Cha:2010-1}. 

Let $M$ be a closed 3-manifold and $\phi\colon \pi_1(M)\to G$ a group
homomorphism. Suppose there is a 4-manifold $W$ bounded by $M$ such
that $\phi$ extends to $\psi\colon \pi_1(W)\to G$. Then the von
Neumann-Cheeger-Gromov $\rhot$-invariant associated to $(M,\phi)$ is
defined to be
\[
\rhot (M,\phi) = \sign^{(2)}_G(W) - \sign(W)\in \R
\]
where $\sign^{(2)}_G(W)$ and $\sign(W)$ denote the $\Lt$-signature of
$W$ associated to $\psi$ and the ordinary signature of $W$,
respectively. We refer the reader to \cite{Cha:2014-1} for more about
$\Lt$-signatures and $\rhot$-invariants.

The following lemma gives a universal bound on the $\rhot$-invariants. 

\begin{lemma}[{\cite[Theorem~1.9]{Cha:2014-1}}] \label{lemma:universal-bound}
  If $M$ is the zero surgery on a knot with crossing number $n$, then
  $|\rhot (M,\phi)|\le 69713280\cdot n$ for all homomorphisms $\phi$.
\end{lemma}

The following is the $\rhot$-invariant obstruction
for a knot to being $n.5$-solvable which we will use in
Section~\ref{subsection:proofs-of-lemma-and-proposition}.  Although we
use the most general known term \emph{amenable group lying in $D(R)$} in the
statement, we do not define it here (precise definitions can be found
from \cite{Cha:2010-1} as well as
\cite{Paterson:1988-1,Strebel:1974-1}), since we will use
Theorem~\ref{theorem:obstruction} only for the class of groups
described in Lemma~\ref{lemma:target-group} below.

\begin{theorem}[{\cite[Theorem 3.2]{Cha:2010-1}}]
  \label{theorem:obstruction}
  Let $K$ be an $n.5$-solvable knot. Let $G$ be an amenable group
  lying in $D(R)$ for $R=\Z_p$ or $\Q$ with $G^{(n+1)} =
  \{e\}$.
  Suppose $\phi\colon \pi_1 M(K)\to G$ is a homomorphism which extends
  to an $n.5$-solution for $M(K)$ and sends a meridian to an
  infinite order element in $G$. Then $\rhot (M(K),\phi)=0$.
\end{theorem}

\begin{definition}[{\cite[Definition 4.1]{Cha:2010-1}}]
  \label{definition:commutator-series}
  Let $G$ be a group and $\cP=(R_0, R_1,\ldots)$ be a sequence of
  rings with unity. The \emph{$\cP$-mixed-coefficient commutator
    series} $\{\cP^kG\}$ is defined inductively by $\cP^0G:=G$
  and
  \[
    \cP^{k+1}G:=\Ker\left\{\cP^kG\to \frac{\cP^kG}{[\cP^kG,\cP^kG]}\to
      \frac{\cP^kG}{[\cP^kG,\cP^kG]}\otimesover\Z
      R_k=H_1(G;R_k[G/\cP^k G])\right\}.
  \]
\end{definition}

For instance, if $R_k=\Z$ for all $k$, then $\{\cP^kG\}$ is the
standard derived series of~$G$, that is, $\cP^kG=G^{(k)}$.  If
$R_k=\Q$ for all $k$, then $\{\cP^kG\}$ is the rational derived series
of~$G$.

The lemma below immediately follows from \cite[Lemma 4.3]{Cha:2010-1}.
\begin{lemma}
  \label{lemma:target-group}
  Let $n$ be a positive integer and $G$ a group. Let
  $\cP=(R_0, R_1,\ldots, R_n)$ with $R_k=\Q$ for $k<n$ and $R_n=\Z_p$
  for a prime $p$. Then, the group $G/\cP^{n+1}G$ is amenable and lies
  in $D(\Z_p)$.
\end{lemma}

\subsection{Non-triviality of the bi-filtrations}
\label{subsection:nontriviality-bifiltration}

In this subsection, we prove the following:

\begin{theorem}
  \label{theorem:infinite-rank} Let $m$ and $n$ be positive
  integers.
  \begin{enumerate}
  \item There exists a family of slice knots $\{J^i\}_{i=1,2\ldots}$ in
    $\cG_{m+2,m+2}$ whose arbitrary nontrivial linear combinations are
    not in $\F_{m.5,m.5}$.
  \item There exists a family of knots $\{J_{m,n}^i\}_{i=1,2,\ldots}$
    in $\cG_{m+2,n+2}$ whose arbitrary nontrivial linear combinations
    are not in $\F_{m.5,n}\cup \F_{m,n.5}$.
  \end{enumerate}
\end{theorem}

From this it follows that there are infinitely many linearly
independent knots at each level of the grope, Whitney tower, and
solvable bi-filtrations of knots, since
$\cG_{m+2,n+2}\subset \cW_{m+2,n+2}\subset \F_{m,n}$ by
Theorems~\ref{theorem:grope-bifiltration-contained-in-whitney-bifiltration}
and~\ref{theorem:comparison of bi-filtrations}.  In particular,
Theorem~\ref{theorem:nontriviality-of-bi-filtrations} follows
immediately from Theorem~\ref{theorem:infinite-rank}.

We remark that Theorem~\ref{theorem:infinite-rank} generalizes
Theorem~1.1 in \cite{Kim:2006-1}.

\begin{proof}[Proof of Theorem~\ref{theorem:infinite-rank}]

  We will prove (2) first.  Since $\F_{k,\ell}=\F_{\ell,k}$, we may
  assume $m\ge n$.  

  For each $i\ge 1$, construct a knot $J_{m,n}^i$ as described in
  Section~\ref{subsection:construction-of-examples}, using the knot
  $\cR=9_{46}$ and the following choices for $K_k$, $\eta_k$, and
  $J_0=J_0^i$: in addition to \ref{item:height-2-grope-for-J_0} and
  \ref{item:height-1-grope-for-K_k-eta_k} in
  Section~\ref{subsection:construction-of-examples}, suppose the
  following.
  
  \begin{enumerate}[label=(N\arabic{*})]
  \item\label{item:cyclic-K} Each $K_k$ has a cyclic Alexander module
    generated by $\eta_k$. For convenience, we take the same
    $(K, \eta)$ as $(K_k,\eta_k)$ for all~$k$.
  \item\label{item:independent-J_0^i} For the knot $K$ in
    \ref{item:cyclic-K} above, let $a_K$ be the top coefficient of the
    Alexander polynomial of $K$ and let $C_K$ be a constant such that
    $|\rhot(M(K), \phi)|<C_K$ for all homomorphisms
    $\phi\colon \pi_1M(K)\to G$ given by
    Lemma~\ref{lemma:universal-bound}.  Then, $J_0^i$ are knots
    obtained by Lemma~\ref{lemma:J_0^i} below with $C_0=(m-1)C_K$ and
    $A=a_K$.
  \end{enumerate}
  \begin{lemma}[{\cite[Proposition 3.4]{Jang:2017-1}}]
    \label{lemma:J_0^i}
    For any constants $C_0$ and $A$, there exist knots
    $J_0^1, J_0^2, \ldots,$ and an increasing sequence of primes
    $A<p_1<p_2\ldots, $ which satisfy the following: letting
    $\omega_i = e^{2\pi\sqrt{-1}/p_i}$, a primitive $p_i$-th root of
    unity,
    \begin{enumerate}
    \item there is a 3D capped grope concordance of height 2 between
      $J_0^i$ and the unknot,
    \item $\sum_{r=0}^{p_i-1}\sigma_{J_0^i}(\omega_i^r) > C_0$, and
    \item $\sum_{r=0}^{p_j-1}\sigma_{J_0^i}(\omega_j^r) =0$ for $j<i$.
    \end{enumerate}  
  \end{lemma}

  In Lemma~\ref{lemma:J_0^i}, Condition~(1) is needed to satisfy
  \ref{item:height-2-grope-for-J_0}, Condition~(2) is needed to make
  each $J_{m,n}^i$ nontrivial modulo $\F_{m.5,n}\cup \F_{m,n.5}$, and
  (3) is needed to have any nontrivial linear combination of
  $J_{m,n}^i$ not in $\F_{m.5,n}\cup \F_{m,n.5}$.

  For instance, take the knot $6_1$ (stevedore's knot) as $K=K_k$, and
  take the curve $\eta$ shown in Figure~\ref{figure:knot 6_1}
  as~$\eta=\eta_k$ for all $k$.  In this case $C_K=418279680$. Note
  that $K$ is a ribbon knot and has a cyclic Alexander module
  $\Z[t^{\pm 1}]/(2t^2-5t+2)$ generated by $\eta$, and there is a 3D
  satellite capped grope of height 1 for $(K,\eta)$ as in
  Figure~\ref{figure:knot 6_1}.

  \begin{figure}[H]
    \labellist
    \small\hair 0mm
    \pinlabel {$K$} at 120 5
    \pinlabel {$\eta$} at 90 110	
    \endlabellist
    \includegraphics{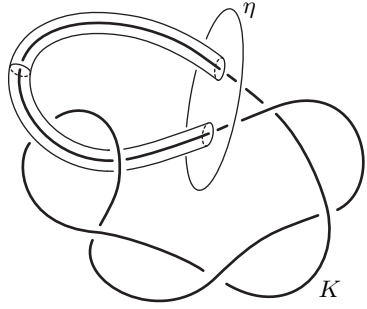}
    \caption{The knot $K=6_1$ and a satellite capped grope}
    \label{figure:knot 6_1}
  \end{figure}

  \begin{proof}[Proof of Lemma~\ref{lemma:J_0^i}]
    For each $m\ge 1$, let $P_m$ be the knot in
    \cite[p.756]{Horn:2010-1}, which generalizes the knot in
    \cite[Figure~3.6]{Cochran-Teichner:2003-1}.

    The Levine-Tristram signature function $\sigma_{P_m}$ of the knot
    $P_m$ is
    \[
      \sigma_{P_m}(e^{2\pi\sqrt{-1}\theta})=
      \begin{cases}
        2	&	\mbox{if }\theta_m < |\theta| \le \pi,\\[1ex]
        0	&	\mbox{if }0\le |\theta|<\theta_m,
      \end{cases}
    \]
    where $\theta_m$ is uniquely determined by the conditions
    $0< \theta_m<\pi$ and
    $\cos (\theta_m)=\frac{2\sqrt[3]{m}-1}{2\sqrt[3]{m}}$ (see
    \cite[Section 4]{Horn:2010-1}). Since $\{\theta_m\}$ is a
    decreasing sequence converging to 0, there are integers $m_i$ and
    primes $p_i$ such that $A< p_1$ and
    $\theta_{m_{i+1}} < 2\pi/p_i < \theta_{m_i}$ for all $i$. Choose
    an integer $N$ such that $N> C/4$ and let $J_0^i$ be a connected
    sum of $N$ copies of $P_{m_{i+1}}\# (-P_{m_i})$ for each
    $i$. Then, $J_0^i$ satisfy the conditions (2) and (3). In
    \cite[Proposition~3.1]{Horn:2010-1}, it is shown that in
    $S^3\times [0,1]$ there is a capped grope concordance of height 2
    between $P_m$ and the unknot. In fact, this capped grope
    concordance is a 3D capped grope concordance (see Figure~5 in
    \cite{Horn:2010-1}). From this observation, one can easily see
    that there is a 3D capped grope concordance of height 2 between
    $J_0^i$ and the unknot.
  \end{proof}

  Now we have knots $J_{m,n}^i$ which satisfy
  \ref{item:height-2-grope-for-J_0},
  \ref{item:height-1-grope-for-K_k-eta_k}, \ref{item:cyclic-K}, and
  \ref{item:independent-J_0^i}. By
  Theorem~\ref{theorem:grope-doubly-slice-construction}, we have
  $J_{m,n}^i\in \cG_{m+2,n+2}$ for each $i$.  Let
  $J=\#_i a_iJ_{m,n}^i$ ($a_i\in \Z$) be a nontrivial linear
  combination of~$J_{m,n}^i$.  We will show that $J$ is neither
  $(m.5,n)$-solvable nor $(m,n.5)$-solvable.  We may assume that there
  are finitely many summands $J_{m,n}^i$ and all coefficients $a_i$
  are nonzero.  By taking $-J$ if necessary, we may assume $a_1>0$.

  Let $C$ be the standard cobordism between $\cup_{a_i} M(J_{m,n}^i)$
  and $M(J)$ obtained by following the method in
  \cite[p.113]{Cochran-Orr-Teichner:2002-1}.  Let $(U,V)$ be an
  $(m,n)$-solution for $J$. We define
  \begin{align*}
    U_C &= C \amalgover{M(J)} U, \\
    V_C  &= C \amalgover{M(J)} V. 
  \end{align*}
  In particular,
  $\partial U_C = \partial V_C = \bigcup_{a_i} M(J_{m,n}^i)$, and we see
  $M(J_{m,n}^1)$ as a boundary component of $U_C$ and $V_C$.

  Recall that $\alpha'$ and $\beta'$ are the curves in the exterior of
  $\cR$ which are shown in Figure~\ref{figure:knot 9_46_2}.  Let
  $\alpha_1$ and $\beta_1$ be the curves in $M(J_{m,n}^1)$ which are
  the images of the parallel copies of $\alpha'$ and $\beta'$ under
  the satellite construction with the pattern~$\cR$.  For $R=\Q$ or
  $\Z_p$ with $p$ an odd prime, we have the following:
  \[
    H_1(M(J_{m,n}^1);R[t^{\pm 1}])\cong H_1(M(\cR);R[t^{\pm 1}])\cong
    P\oplus Q
  \]
  where $P = \langle \alpha_1\rangle\cong R[t^{\pm 1}]/(2t-1) $ and
  $Q = \langle\beta_1\rangle \cong R[t^{\pm 1}]/(t-2)$.  This can be
  verified by straightforward computation.  (Alternatively, one may
  invoke Proposition~\ref{proposition:homology-splitting} which is
  stated in Section~\ref{subsection:proofs-of-lemma-and-proposition}.)

  \begin{lemma}\label{lemma:nonvanishing-of-infection-curves}
    Let $R=\Q$ or $\Z_p$ with $p$ an odd prime. Let $(U,V)$ be an
    $(m,n)$-solution for~$J$.  Let $U_C$ and $V_C$ be defined as
    above. Let
    $i_*\colon H_1(M(J_{m,n}^1);R[t^{\pm 1}])\to H_1(U_C;R[t^{\pm
      1}])$ and
    $j_*\colon H_1(M(J_{m,n}^1);R[t^{\pm 1}])\to H_1(V_C;R[t^{\pm
      1}])$ be inclusion-induced homomorphisms.  Then, either $\left\{
      \begin{tabular}{@{}l@{}}
        $i_*(\alpha_1)\ne 0$ \\
        $j_*(\beta_1)\ne 0$
      \end{tabular}
    \right\}$
    or
    $\left\{
      \begin{tabular}{@{}l@{}}
        $j_*(\alpha_1)\ne 0$ \\
        $i_*(\beta_1)\ne 0$
      \end{tabular}
    \right\}$.
  \end{lemma}

  The proof of Lemma~\ref{lemma:nonvanishing-of-infection-curves} is
  postponed to
  Section~\ref{subsection:proofs-of-lemma-and-proposition}.
  
  The next proposition is a crucial technical result. Below we use the
  convention that
  $J_{m,0}^i:=\cR(\alpha',\beta';J_{m-1}^i,\mbox{unknot})$ and
  $J_{0,n}^i:=\cR(\alpha',\beta';\mbox{unknot},J_{n-1}^i)$.
  
  \begin{proposition}
    \label{proposition:non-solvability}
    Let $m,n\ge0$. Let $R=\Q$ or $\Z_{p_1}$. Let $W$ be a compact
    4-manifold with boundary $M(J)$. Let
    $W_C = C\mathbin{\amalg_{M(J)}} W$. Let
    \[
      inc_*\colon H_1(M(J_{m,n}^1);R[t^{\pm 1}])\to H_1(W_C;R[t^{\pm
        1}])
    \]
    be the inclusion-induced homomorphism. Then, the following
    hold:
    \begin{enumerate}
    \item If $m\ge 1$ and $inc_*(\alpha_1)\ne 0$, then $W$ is not an
      $m.5$-solution for $J$.
    \item If $n\ge 1$ and $inc_*(\beta_1)\ne 0$, then $W$ is not an
      $n.5$-solution for $J$.
    \end{enumerate}
  \end{proposition}
  
  We postpone the proof of
  Proposition~\ref{proposition:non-solvability} to
  Section~\ref{subsection:proofs-of-lemma-and-proposition}.

  Suppose $J$ is $(m.5,n)$-solvable via~$(U,V)$.  We apply
  Proposition~\ref{proposition:non-solvability} for $W=U$ and
  $inc_* = i_*\colon H_1(M(J_{m,n}^1);R[t^{\pm 1}])\to
  H_1(U_C;R[t^{\pm 1}])$; since $U$ is an $m.5$-solution for~$J$, by
  Proposition~\ref{proposition:non-solvability}~(1) we have
  $i_*(\alpha_1)=0$. Then, by
  Lemma~\ref{lemma:nonvanishing-of-infection-curves}, we have
  $j_*(\alpha_1)\ne 0$ and $i_*(\beta_1)\ne 0$.  Since
  $i_*(\beta_1)\ne 0$, by
  Proposition~\ref{proposition:non-solvability}(2), $U$ is not an
  $n.5$-solution for~$J$. Since $m\ge n$, this implies that $U$ is not
  an $m.5$-solution for $J$, which is a contradiction.

  Next, we prove that $J$ is not $(m, n.5)$-solvable. Since
  $\F_{k,\ell} = \F_{\ell,k}$ and it has been shown that $J$ is not
  $(m.5, n)$-solvable, we may assume $m>n$. Suppose $J$ is
  $(m, n.5)$-solvable via~$(U,V)$.  By applying
  Proposition~\ref{proposition:non-solvability}~(2) for $W=V$ and
  \[
    inc_* = j_*\colon H_1(M(J_{m,n}^1);R[t^{\pm 1}])\to
    H_1(V_C;R[t^{\pm 1}]),
  \]
  we have $j_*(\beta_1) = 0$. Then, by
  Lemma~\ref{lemma:nonvanishing-of-infection-curves}, we have
  $i_*(\beta_1) \ne 0$. Then, by applying
  Proposition~\ref{proposition:non-solvability}(2) for $W=U$ and
  \[
    inc_* = i_*\colon H_1(M(J_{m,n}^1);R[t^{\pm 1}])\to
    H_1(U_C;R[t^{\pm 1}]),
  \]
  it follows that $U$ is not an $n.5$-solution for~$J$.  But since
  $m>n$, this contradicts that $U$ is an $m$-solution for $J$. This
  completes the proof of Theorem~\ref{theorem:infinite-rank}(2).
  
  To prove Theorem~\ref{theorem:infinite-rank}(1), we use an argument
  similar to the proof of Theorem~\ref{theorem:infinite-rank}(2).
  Choose $K_k$, $\eta_k$, and $J_0^i$ as in the proof of
  Theorem~\ref{theorem:infinite-rank}(2).  Let
  $J^i=\cR(\alpha'; J_{m-1}^i)$ where $\cR$, $\alpha'$, and
  $J_{m-1}^i$ are those given and defined in
  Section~\ref{subsection:construction-of-examples}. Then each $J^i$
  is ribbon, in particular slice, since it bounds a ribbon disk, say
  $\Delta$, obtained by cutting the band dual to $\alpha$. Let $G'$ be
  a capped grope of height $n+2$ in $D^4$ bounded by $J^i$ obtained by
  replacing a sufficiently small disk in the interior of $\Delta$ by a
  model disk-like capped grope of height $m+2$. Then there is an
  epimorphism $\pi_1(S^3\setminus J^i) \to \pi_1(D^4\setminus G')$
  induced from the inclusion since $\Delta$ and $G'$ have homeomorphic
  exteriors in $D^4$. Now by
  Theorem~\ref{theorem:grope-doubly-slice-construction} and its proof,
  where we use $G'\subset D^4_-$ instead of $G_-\subset D^4_-$, we can
  see $J^i\in \cG_{m+2, m+2}$.

  Suppose $J=\#_i a_iJ^i$ is a nontrivial linear combination.  We will
  show that $J$ is not doubly $m.5$-solvable. As before, we may assume
  that there are finitely many summands $J^i$ and all the coefficients
  $a_i$ are nonzero.  By taking $-J$ if necessary, we may assume
  $a_1>0$.

  Suppose $J$ is doubly $m.5$-solvable via $(U, V)$. Let $\alpha_1$ be
  the curve in $M(J^1)$ which is the image of a parallel copy of the
  curve $\alpha'$ under the satellite construction with the pattern
  $\cR$ . For $R=\Q$ or $\Z_{p_1}$, let $i_*$ and $j_*$ be the
  inclusion-induced homomorphisms as in the proof of the
  statement~(2). By
  Lemma~\ref{lemma:nonvanishing-of-infection-curves}, we have
  $i_*(\alpha_1)\ne 0$ or $j_*(\alpha_1)\ne 0$. By exchanging $U$ and
  $V$ if necessary, we may assume that $i_*(\alpha_1)\ne 0$.  Note
  that $J^i=\cR(\alpha',\beta'; J_{m-1}^i,\mbox{unknot})=J_{m,0}^i$
  with $m\ge 1$. Therefore, by
  Proposition~\ref{proposition:non-solvability}~(1), $U$ is not an
  $m.5$-solution, which is a contradiction.

  This completes the proof of Theorem~\ref{theorem:infinite-rank},
  modulo the proofs of
  Lemma~\ref{lemma:nonvanishing-of-infection-curves} and
  Proposition~\ref{proposition:non-solvability}.
\end{proof}

\subsection{Proofs of Lemma~\ref{lemma:nonvanishing-of-infection-curves} and Proposition~\ref{proposition:non-solvability}}
\label{subsection:proofs-of-lemma-and-proposition}

In this subsection, we prove
Lemma~\ref{lemma:nonvanishing-of-infection-curves} and
Proposition~\ref{proposition:non-solvability}.  For this purpose we
need the following splitting result for doubly $1$-solvable knots.
Let $R=\Q$ or $\Z_p$.  For a knot $K$, the \emph{Blanchfield form}
\[
  \Bl \colon H_1(M(K);R[t^{\pm 1}]) \times H_1(M(K);R[t^{\pm 1}]) \to
  R(t)/R[t^{\pm 1}]
\]
is defined~\cite{Blanchfield:1957-1}.  For a submodule $P$ of
$H_1(M(K); R[t^{\pm 1}])$, define
\[
P^\perp:=\{x\in H_1(M(K); R[t^{\pm 1}])\,\mid\, \Bl(x,y)=0\mbox{ for all }y\in P\}.
\]
We say that $P$ is \emph{self-annihilating} with respect to $\Bl$ if
$P=P^\perp$. We have the following proposition.

\begin{proposition}[{\cite[Proposition 2.10]{Kim:2006-1} for $R=\Q$ and $\Z$}]
  \label{proposition:homology-splitting}
  Let $R=\Q$ or $\Z_p$. Let $K$ be
  a knot which is doubly $1$-solvable via $(W_1,W_2)$.  Let $P_i$ be
  the kernel of the inclusion-induced homomorphism
  $H_1(M(K);R[t^{\pm 1}])\to H_1(W_i;R[t^{\pm 1}])$ for $i=1,2$.  Then
  \[
    H_1(M(K);R[t^{\pm 1}])\cong P_1\oplus P_2
  \]
  as $R[t^{\pm 1}]$-modules, and each $P_i$ is self-annihilating with
  respect to the Blanchfield form.
\end{proposition}

Now we are ready to prove Lemma~\ref{lemma:nonvanishing-of-infection-curves}.

\begin{proof}[Proof of Lemma~\ref{lemma:nonvanishing-of-infection-curves}]
  Recall that $R=\Q$ or $\Z_p$.  Using a Mayer-Vietoris sequence, we
  obtain
  \[
    H_1(C;R[t^{\pm 1}])\cong H_1(\partial_+ C;R[t^{\pm 1}])\cong
    H_1(\partial C_-;R[t^{\pm 1}])
  \]
  where $\partial_+ C = M(J) $,
  $\partial_- C = \cup_{a_i} M(J_{m,n}^i)$, and the isomorphisms are
  induced by the inclusions. Therefore, via the isomorphisms we can
  consider $H_1(M(J_{m,n}^1);R[t^{\pm 1}])$ as a submodule of
  $H_1(M(J);R[t^{\pm 1}])$. Using Mayer-Vietoris sequences, it follows
  that $H_1(U;R[t^{\pm 1}])\cong H_1(U_C;R[t^{\pm 1}])$ and
  $H_1(V;R[t^{\pm 1}])\cong H_1(V_C;R[t^{\pm 1}])$.

  Since $m,n\ge 1$, $(U,V)$ is a double $1$-solution for
  $M(J)$. Therefore, by
  Proposition~\ref{proposition:homology-splitting},
  $H_1(M(J);R[t^{\pm 1}])\cong \Ker i_*^U\oplus \Ker j_*^V$ where
  $i_*^U\colon H_1(M(J);R[t^{\pm 1}])\to H_1(U;R[t^{\pm 1}])$ and
  $j_*^V\colon H_1(M(J);R[t^{\pm 1}])\to H_1(V;R[t^{\pm 1}])$ are
  inclusion-induced homomorphisms.

  Suppose $i_*(\alpha_1)=j_*(\alpha_1)=0$. Considering
  $H_1(M(J_{m,n}^1);R[t^{\pm 1}])$ as a submodule of
  $H_1(M(J);R[t^{\pm 1}])$, this implies that
  $i_*^U(\alpha_1)=j_*^V(\alpha_1)=0$. Since
  $H_1(M(J);R[t^{\pm 1}])\cong \Ker i_*^U\oplus \Ker j_*^V$ and
  $\alpha_1$ is a nontrivial element of $H_1(M(J);R[t^{\pm 1}])$, this
  is a contradiction. Therefore we cannot have
  $i_*(\alpha_1)=j_*(\alpha_1)=0$. With a similar reason, we cannot
  have $i_*(\beta_1)=j_*(\beta_1)=0$.

  Since $m,n\ge 1$, the knots $J_{m,n}^i$ are $1$-solvable. Let $W^i$
  be a $1$-solution for $M(J_{m,n}^i)$ for each $i$. Let $W^i_r$ be a
  copy of $-W^i$.

  \begin{figure}[H]
    \begin{tikzpicture}[x=1pt, y=1pt, every node/.style=transform shape, scale=.9]
      \small
      \node[anchor=south west, inner sep=0]
      {\includegraphics[scale=1]{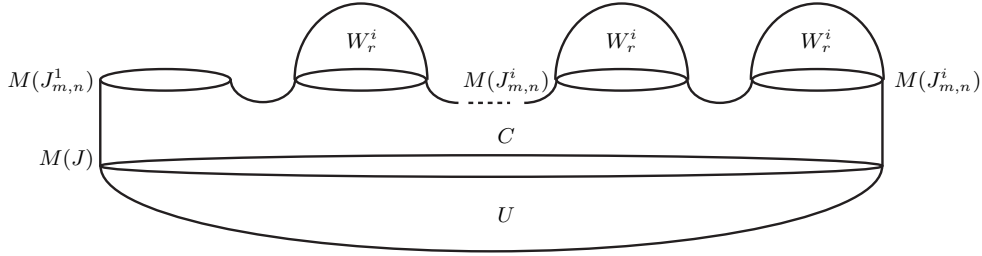}};
      \node [left] at (2,73) {$M(J_{m,n}^1)$};
      \node  at (110,90) {$W_r^i$};
      \node  at (218,90) {$W_r^i$};
      \node  at (299,90) {$W_r^i$};
      \node at (170, 73) {$M(J_{m,n}^i)$};
      \node [right] at (328, 73) {$M(J_{m,n}^i)$};
      \node at (170, 50) {$C$};
      \node [left] at (2,40) {$M(J)$};
      \node at (170, 17) {$U$};   
    \end{tikzpicture}
    \caption{The cobordism $U_1$}
    \label{figure:cobordism U_1}
  \end{figure}
  
  As in Figure~\ref{figure:cobordism U_1}, we define
  \[
    U_1= U\amalgover{M(J)} C \amalgover{(\partial_-C) \setminus
      M(J_{m,n}^1)} \Bigg((a_i-1)W_r^1\cup\bigg(\bigcup_{i\ge 2}a_i
        W^i_r\bigg)\Bigg).
  \]

  Then $\partial U_1=M(J_{m,n}^1)$. Let
  $i^1_*\colon H_1(M(J_{m,n}^1);R[t^{\pm 1}])\to H_1(U_1;R[t^{\pm
    1}])$ be the inclusion-induced homomorphism. Since $U$ is an
  $m$-solution, it is also a $1$-solution. Using Mayer-Vietoris
  sequences, one can show that $U_1$ is a $1$-solution for
  $M(J_{m,n}^1)$.

  It is known that $\Ker i^1_*$ is a self-annihilating submodule of
  $H_1(M(J_{m,n}^1);R[t^{\pm
    1}])$~\cite[Theorem~4.4]{Cochran-Orr-Teichner:1999-1}.  Therefore
  $\Ker i^1_*$ is a proper submodule.  Since
  $H_1(M(J_{m,n}^1);R[t^{\pm 1}])$ is generated by $\alpha_1$ and
  $\beta_1$, it follows that we cannot have
  $i_*(\alpha_1)=i_*(\beta_1)= 0$. Similarly, we cannot have
  $j_*(\alpha_1)=j_*(\beta_1)= 0$.

  Now suppose $i_*(\alpha_1)=0$. Since we can have neither
  $i_*(\alpha_1)=j_*(\alpha_1)=0$ nor $i_*(\alpha_1)=i_*(\beta_1)= 0$,
  it follows $j_*(\alpha_1)\ne 0$ and $i_*(\beta_1)\ne 0$. Next,
  suppose $j_*(\alpha_1)=0$. Using similar arguments we can easily see
  that $i_*(\alpha_1)\ne 0$ and $j_*(\beta_1)\ne 0$.

  Suppose the remaining case $i_*(\alpha_1)\ne 0$ and
  $j_*(\alpha_1)\ne 0$. Since we cannot have
  $i_*(\beta_1)=j_*(\beta_1)=0$, either $i_*(\beta_1)\ne 0$ or
  $j_*(\beta_1)\ne 0$. In either case, the conclusion of the theorem
  holds.
\end{proof}

In the proof of Proposition~\ref{proposition:non-solvability}, we will
use the following fact.

\begin{proposition}[{\cite[Proposition 4.10]{Cha:2010-1}}]
  \label{proposition:m-or-n-solution}
  Suppose the Arf invariant of $J_0$ vanishes, and let $J_{m,n}$ be
  the knot constructed in
  Section~\ref{subsection:construction-of-examples}.  Then there exist
  an $m$-solution $U$ and an $n$-solution $V$ for $J_{m,n}$
  satisfying the following:
  \begin{enumerate}
  \item If $\phi\colon \pi_1M(J_{m,n})\to G$ is a homomorphism which
    extends to $U$ and $G$ is an amenable group lying in $D(R)$ for
    some ring $R$ with unity, then
    $\rhot(M(J_{m,n},\phi)) = \rhot(M(J_0,\psi))$ where
    $\psi\colon \pi_1M(J_0)\to \Z_d$ is a surjection and $d$ is the
    order of the image of the meridian of $J_0$ under $\phi$ in~$G^{(m)}$.  The
    same statement holds when $U$ and $G^{(m)}$ are replaced by $V$
    and~$G^{(n)}$.
  \item For $R=\Q$ and $\Z_p$ with $p$ an odd prime,
    $H_1(U;R[t^{\pm 1}])\cong R[t^{\pm 1}]/(2t-1)$ generated by the
    curve $\alpha_1$ and $H_1(V;R[t^{\pm 1}])\cong R[t^{\pm 1}]/(t-2)$
    generated by the curve $\beta_1$.
  \end{enumerate}
\end{proposition}

The statement~(1) follows immediately from \cite[Proposition
4.10]{Cha:2010-1}, and the statement~(2) is implicitly proven in the
proof of \cite[Proposition 4.10]{Cha:2010-1}. Now we give a proof of
Proposition~\ref{proposition:non-solvability}.

\begin{proof}[Proof of Proposition~\ref{proposition:non-solvability}]
  To prove (1), we will construct a certain 4-manifold with boundary
  $M(J_0^1)$.  The building blocks for the construction are as
  follows.

  \begin{enumerate}
  \item Let $U^i$ be the $m$-solution for $M(J_{m,n}^i)$ given in
    Proposition~\ref{proposition:m-or-n-solution}.
  \item Let $\cE_+$ be the exterior in $D^4$ of the slice disk for
    $\cR_{\beta'}:=\cR(\beta';J_{n-1}^1)$ in which $\beta'$ is
    null-homotopic, where the slice disk is obtained by cutting the
    band dual to $\beta$.
  \item Recalling $K_{k+1}=K_k(\eta_k;J_k)$, let $E_k$ be the standard
    cobordism between $M(J_k^1)\cup M(K_k)$ and $M(K_{k+1})$. That is,
    letting $E(J_k^1)$ denote the exterior of $J_k^1$ in $S^3$,
    $E_k = M(J_k^1)\times [0,1] \amalg M(K_k)\times [0,1]/\sim$ where
    the solid torus $M(J_k^1)\setminus E(J_k^1)$ in
    $M(J_k^1) = M(J_k^1)\times 0$ is identified with the tubular
    neighborhood of $\eta_k$ in $M(K_k)=M(K_k)\times 1$ in such a way
    that the meridian (resp. the 0-linking longitude) of $J_k^i$ is
    identified with the 0-linking longitude (resp. the meridian) of
    $\eta_k$ (see \cite[p.1429]{Cochran-Harvey-Leidy:2009-1}).
  \item Let $E_{m-1}^+$ be the standard cobordism between
    $M(J_{m-1}^1)\cup M(R_{\beta'})$ and $M(J_{m,n}^1)$ constructed
    similarly to (3), noting that
    $J_{m,n}^1 = R_{\beta'}(\alpha';J_{m-1}^1)$.
  \item Let $C$ be the standard cobordism from
    $\cup_{a_i} M(J_{m,n}^i)$ to $M(J)$ as defined in the proof of
    Theorem~\ref{theorem:infinite-rank}.
  \end{enumerate}

  Let $U^i_r$ be a copy of $-U^i$.  We define
  \[
    W_m= W\amalgover{M(J)} C \amalgover{(\partial_-C) \setminus
      M(J_{m,n}^1)} \Bigg((a_i-1)U_r^1\cup\bigg(\bigcup_{i\ge 2}a_i
        U^i_r\bigg)\Bigg).
  \]
  Then, $\partial W_m = M(J_{m,n}^1)$. We define
  \[
    W_{m-1} = W_m\amalgover{M(J_{m,n}^1)}
    E_{m-1}^+\amalgover{M(R_{\beta'})} \cE_+.
  \]
  Then, $\partial W_{m-1} = M(J_{m-1}^1)$.  Finally, for
  $k=m-2, m-3, \ldots, 0$, as in Figure~\ref{figure:cobordism W_k} we
  define
  \begin{align*}
    W_k & = E_k\amalgover{M(J_{k+1}^1)} W_{k+1}\\
	& = E_k\amalgover{M(J_{k+1}^1)} E_{k+1}\amalgover{M(J_{k+2}^1)}
          \cdots\amalgover{M(J_{m-2}^1)}E_{m-2}\amalgover{M(J_{m-1}^1)} W_{m-1}.
  \end{align*}
  Then
  $\partial W_k = M(J_k^1)\cup \big(\bigcup_{j=k}^{m-2} M(K_j)\big) =
  M(J_k^1)\cup (m-k-1) M(K)$.
  
  \begin{figure}[H]
    \begin{tikzpicture}[x=1bp,y=1bp,every node/.style=transform shape,scale=.78]
      \node [anchor=south west, inner sep=0mm] {\includegraphics[trim=0 10 0 0]{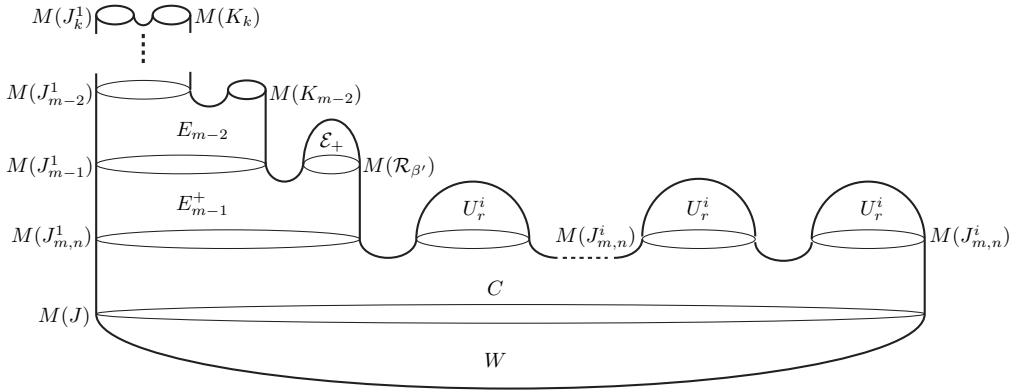}};
      \begin{scope}[yshift=-10]
        \node at (4,46) {$M(J)$};
        \node at (-2,84) {$M(J_{m,n}^1)$};
        \node at (-3,118) {$M(J_{m-1}^1)$};
        \node at (164,118) {$M(\cR_{\beta'})$};
        \node at (-3,153) {$M(J_{m-2}^1)$};
        \node at (3,190) {$M(J_k^1)$};
        \node at (82,190) {$M(K_k)$};
        \node at (124,153) {$M(K_{m-2})$};
        \node at (258,84) {$M(J_{m,n}^i)$};
        \node at (437,84) {$M(J_{m,n}^i)$};
        \node at (210,60) {$C$};
        \node at (210,25) {$W$};
        \node at (132,130) {$\cE_+$};
        \node at (70,100) {$E_{m-1}^+$};
        \node at (70,135) {$E_{m-2}$};
        \node at (307,98) {$U_r^i$};
        \node at (200,98) {$U_r^i$};
        \node at (390,98) {$U_r^i$};
      \end{scope}
    \end{tikzpicture}
    \caption{The cobordism $W_k$}
    \label{figure:cobordism W_k}
  \end{figure}

  Let $\cP = (R_0,R_1,\ldots, R_m)$ be a sequence of rings where
  $R_k = \Q$ for $0\le k\le m-1$ and $R_m = \Z_{p_1}$. Then for a
  group $G$, we get the $\cP$-mixed-coefficient commutator series
  $\cP^k G$ for $k = 0, 1, \ldots, m+1$ as in
  Definition~\ref{definition:commutator-series}. The following lemma
  is essentially due to \cite[Theorem~4.14]{Cha:2010-1}.

  \begin{lemma}[{\cite[Theorem 4.14]{Cha:2010-1}}]
    \label{lemma:nonvanishing-of-the-meridian}
    Rename $J_m^1 := J_{m,n}^1$ for brevity.  Suppose
    $inc_*(\alpha_1)\ne 0$ as in the hypothesis of
    Proposition~\ref{proposition:non-solvability}~(1). For
    $k=0,1,\ldots, m$, the homomorphism
    \[
      \phi_k\colon \pi_1 W_k\to \pi_1 W_k/\cP^{m-k+1}\pi_1 W_k
    \]
    sends the meridian of $J_k^1$ into the abelian subgroup
    $\cP^{m-k}\pi_1 W_k/\cP^{m-k+1}\pi_1W_k$. Furthermore, the image
    of the meridian of $J_k^1$ under $\phi_k$ has order $p_1$ if
    $k=0$, and is of infinite order if $k>0$.
  \end{lemma}

  We postpone the proof of
  Lemma~\ref{lemma:nonvanishing-of-the-meridian} to the end of this
  section and finish the proof of
  Proposition~\ref{proposition:non-solvability}. Let
  $G=\pi_1 W_0/\cP^{m+1}\pi_1 W_0$. In particular,
  $G^{(m+1)} = \{e\}$. By Lemma~\ref{lemma:target-group}, the group
  $G$ is amenable and lies in $D(\Z_{p_1})$. Now we have the canonical
  homomorphism $\phi\colon \pi_1 W_0\to G$, and by abuse of notation,
  we denote by $\phi$ restrictions of $\phi$ to subspaces of $W_0$.

  For a 4-manifold $X$ with a homomorphism $\pi_1 X\to G$, we let
  $S_G(X)=\sign_G^{(2)}(X) - \sign(X)$, the $\Lt$-signature
  defect. Since
  $\partial W_0 = M(J_0^1)\cup \big(\bigcup_{k=0}^{m-2} M(K_k)\big)$,
  using the definition of $\rhot$-invariants in
  Section~\ref{subsection:amenable-signature}, we have
  \[
    \rhot(\partial W_0,\phi)=\rhot(M(J_0^1),\phi) +\sum_{k=0}^{m-2}
    \rhot(M(K_k),\phi) = S_G(W_0).
  \]
  
  On the other hand, using Novikov additivity, we have
  \[
    S_G(W_0) = S_G(W) + S_G(C) + \sum_{i}\sum_{r=1}^{b_i} S_G(U^i_r) +
    S_G(E_{m-1}^+) + S_G(\cE_+) + \sum_{k=0}^{m-2}S_G(E_k).
  \]
  Here we use the convention that
  \[
    \sum_{k=0}^{m-2} \rhot(M(K_k),\phi)=\sum_{k=0}^{m-2}S_G(E_k)=0
  \]
  if $m=1$. We compute each term of these equalities below.

  \begin{enumerate}
  \item
    $\rhot (M(J_0^1),\phi) = \sum_{r=0}^{p_1-1}
    \sigma_{J_0^1}(\omega_1^r)>0$ where
    $\omega_1=e^{2\pi \sqrt{-1}/p_1}$ and $\sigma_{J_0^1}$ is the
    Levine-Tristram signature function for $J_0^1$: by
    Lemma~\ref{lemma:nonvanishing-of-the-meridian} the image of the
    meridian of $J_0^1$ under $\phi$ has order $p_1$, and in this case
    it is known that
    $\rhot (M(J_0^1),\phi) = \sum_{r=0}^{p_1-1}
    \sigma_{J_0^1}(\omega_1^r)>0$ (see \cite[Proposition
    5.1]{Cochran-Orr-Teichner:2002-1} and \cite[Corollary
    4.3]{Friedl:2003-5}). The last inequality follows from
    Lemma~\ref{lemma:J_0^i}(2).
  \item $S_G(C)=0$ by \cite[Lemma~4.2]{Cochran-Orr-Teichner:2002-1}
  \item $S_G(U_r^1)$ is 0 or
    $-\sum_{r=0}^{p_1-1} \sigma_{J_0^1}(\omega_1^r)$: since
    $G^{(m+1)}=\{e\}$, the subgroup $G^{(m)}$ is abelian. By our
    choice of the ring $R_m=\Z_{p_1}$, the group $G^{(m)}$ is a vector
    space over $\Z_{p_1}$, and hence every element in $G^{(m)}$ has
    order either 1 or $p_1$. By
    Proposition~\ref{proposition:m-or-n-solution},
    $S_G(U_r^1) = S_G(-U^1)=- \rhot(M(J_0^1), \psi)$ where
    $\psi\colon \pi_1M(J_0^1)\to \Z_{p_1}$ is a homomorphism. If
    $\psi$ is trivial, then $\rhot(M(J_0^1), \psi)=0$, and otherwise
    $\rhot(M(J_0^1), \psi)= \sum_{r=0}^{p_1-1}
    \sigma_{J_0^1}(\omega_1^r)$ as calculated in (1) above.
  \item $S_G(U^i_r)=0$ for $i>1$. With the same reason in (3), we have
    $S_G(U^i_r)$ is 0 or
    $-\sum_{r=0}^{p_1-1} \sigma_{J_0^i}(\omega_1^r)$. But by
    Lemma~\ref{lemma:J_0^i}(3), we have
    $\sum_{r=0}^{p_1-1} \sigma_{J_0^i}(\omega_1^r)=0$ for $i>1$.
  \item $S_G(\cE_+)=0$: since $\cE_+$ is a slice disk exterior, we
    have $H_2(\cE_+)=0$. Therefore $\sign(\cE_+)=0$.  By
    \cite[Lemma~2.7]{Cha:2006-1},
    $|\sign_G^{(2)}(\cE_+)|\le \rank_\Z H_2(\cE_+)=0$. Therefore
    $\sign_G^{(2)}(\cE_+)=0$.
  \item $S_G(E_{m-1}^+)=S_G(E_k) = 0$ for all $k$ due to
    \cite[Lemma~2.4]{Cochran-Harvey-Leidy:2009-1}.
  \end{enumerate}
  From these computations, we conclude that
  \[
    \epsilon\cdot \sum_{r=0}^{p_1-1} \sigma_{J_0^1}(\omega_1^r) +
    \sum_{k=0}^{m-2} \rhot(M(K_k),\phi) = S_G(W)
  \]
  for some constant $\epsilon\ge 1$.

  Suppose $W$ is an $m.5$-solution for $J$. Then, by
  Theorem~\ref{theorem:obstruction}, we have $S_G(W)=0$. But this
  leads us to a contradiction: since $K_k=K$ for all $k$, and by our
  choice of $J_0^1$ we have
  $\epsilon\cdot \sum_{r=0}^{p_1-1} \sigma_{J_0^1}(\omega_1^r)>
  (m-1)C_K$. On the other hand,
  $|\sum_{k=0}^{m-2} \rhot(M(K_k),\phi)|\le (m-1)C_K$. This completes
  the proof of the statement~(1) modulo the proof of
  Lemma~\ref{lemma:nonvanishing-of-the-meridian}.

  The statement~(2) can be proved using arguments similar to the one
  for the statement~(1).
\end{proof}

Now we finish the proof of
Proposition~\ref{proposition:non-solvability} by proving
Lemma~\ref{lemma:nonvanishing-of-the-meridian}.

\begin{proof}[Proof of Lemma~\ref{lemma:nonvanishing-of-the-meridian}]
  The proof is essentially identical to the one of
  \cite[Theorem~4.14]{Cha:2010-1}. The only difference occurs at the
  first step where we use $E_{m-1}^+\amalgover{M(R_{\beta'})} \cE_+$
  instead of $E_{m-1}$ to construct $W_{m-1}$ from $W_m$, and we only
  need to show that
  \[
    \cP^1\pi_1 W_{m-1}/\cP^2\pi_1 W_{m-1} \cong \cP^1\pi_1 W_m/\cP^2\pi_1 W_m.
  \]
	
  Let $W' = W_m\amalgover{M(J_{m,n}^1)} E_{m-1}^+$. Then, using
  exactly the same argument as in the proof of Assertion~1 in
  \cite[p.4799]{Cha:2010-1}, one can show that
  \[
    \pi_1W'/\pi_1W'^{(2)} \cong \pi_1W_m/\pi_1W_m^{(2)}.
  \]
  Therefore, we have 
  \[
    \cP^1\pi_1 W'/\cP^2\pi_1 W' \cong \cP^1\pi_1 W_m/\cP^2\pi_1 W_m.
  \]
  
  Since $W_{m-1} = W'\amalgover{M(R_{\beta'})} \cE_+$ and
  $\pi_1 \cE_+\cong \pi_1(M(R_{\beta'}))/\langle \beta'\rangle$ where
  $\langle \beta'\rangle$ denotes the subgroup normally generated by
  $\beta'$, by Seifert-van Kampen we obtain
  $\pi_1 W_{m-1} \cong \pi_1 W'/\langle \beta'\rangle$.

  For the moment, suppose $\beta'\in \cP^2\pi_1 W'$. Then, since
  $\cP^2\pi_1W'$ maps into $\cP^2\pi_1W_{m-1}$, we have
  \[
    \cP^1\pi_1 W'/\cP^2\pi_1 W' \cong \cP^1\pi_1 W_{m-1}/\cP^2\pi_1 W_{m-1},
  \]
  and therefore
  \[
    \cP^1\pi_1 W_{m-1}/\cP^2\pi_1 W_{m-1} \cong \cP^1\pi_1 W_m/\cP^2\pi_1 W_m.
  \]
  
  Therefore, it suffices to show $\beta'\in \cP^2\pi_1 W'$. Note that
  $\beta'$ is isotopic to $\beta_1 \subset M(J_{m,n}^1)$ in
  $E_{m-1}^+$, and again it suffices to show
  $\beta_1\in \cP^2\pi_1W_m$.

  Using Mayer-Vietoris sequences one can show that $W_m$ is a
  $1$-solution for $J_{m,n}^1$. Noting $R_1=\Z_{p_1}$ if $m=1$ and
  $\Q$ if $m>1$, by \cite[Theorem~4.4]{Cochran-Orr-Teichner:1999-1} it
  is known that $\Ker i_*$ is a self-annihilating submodule of
  $H_1(M(J_{m,n}^1);R_1[t^{\pm 1}])$ where
  $i_*\colon H_1(M(J_{m,n}^1);R_1[t^{\pm 1}])\to H_1(W_m;R_1[t^{\pm
    1}])$ is the inclusion-induced homomorphism. Therefore
  $\Ker i_*=P=\langle \alpha_1\rangle$ or
  $\Ker i_*=Q=\langle \beta_1\rangle$. Recall that
  \[
    W_C= C\amalgover{M(J)} W \mbox{ and }
    W_m=W_C\amalgover{(\partial_-C) \setminus M(J_{m,n}^1)}
    \Bigg((a_i-1)U_r^1\cup\bigg(\bigcup_{i\ge 2}a_i
        U^i_r\bigg)\Bigg).
  \] 
  Since $U_r^i=-U^i$ where $U^i$ has been obtained using
  Proposition~\ref{proposition:m-or-n-solution}, we have
  $H_1(U_r^i;R[t^{\pm 1}])\cong R[t^{\pm 1}]/(2t-1)$. Therefore, using
  Mayer-Vietoris sequences, one can see that the inclusion-induced
  homomorphism $H_1(W_C;R[t^{\pm 1}])\to H_1(W_m;R[t^{\pm 1}])$ is a
  surjection whose kernel is $(t-2)$-torsion. Since $\alpha_1$ is
  $(2t-1)$-torsion in $H_1(M(J_{m,n}^1);R[t^{\pm 1}])$ and
  $inc_*(\alpha_1)\ne 0$ in $H_1(W_C;R[t^{\pm 1}])$ by the hypothesis,
  $inc_*(\alpha_1)$ is not contained in the kernel of the
  surjection. It follows that $i_*(\alpha_1)\ne 0$. Therefore
  $\Ker i_* = Q=\langle \beta_1 \rangle$, and hence $i_*(\beta_1)=0$.

  Since $\pi_1 W_m/\cP^1\pi_1 W_m\cong \Z=\langle t\rangle$, the
  quotient group $\cP^1\pi_1 W_m/\cP^2\pi_1 W_m$ injects into
  $H_1(W_m;R_1[t^{\pm 1}])$.  Since $i_*(\beta_1)=0$, it follows that
  $\beta_1=0$ in $\cP^1\pi_1 W_m/\cP^2\pi_1 W_m$. Therefore
  $\beta_1\in \cP^2\pi_1W_m$.
\end{proof}

\section{Bi-filtrations and classical obstructions}
\label{section:bi-filtrations and classical obstructions}

In this section, we discuss relationships between our bi-filtrations
and previously known double sliceness obstructions.  We show that the
double sliceness obstructions in \cite{Gilmer-Livingston:1983-1}
vanish for knots in $\F_{2,2}$ and hence for knots in $\cG_{4,4}$
and~$\cW_{4,4}$.  We also show that the obstructions in
\cite{Friedl:2003-4, Livingston-Meier:2015-1} vanish for knots
in~$\F_{1.5,1.5}$, and hence for knots in $\cG_{3.5,3.5}$
and~$\cW_{3.5,3.5}$.  Details are given below.

\subsection*{Algebraic double sliceness}

Due to Sumners~\cite{Sumners:1971-1}, a doubly slice knot has a
hyperbolic Seifert matrix, that is, a Seifert matrix of the form
$\big[\begin{smallmatrix}0 & * \\ * & 0\end{smallmatrix}\big]$ with
the zero blocks of the same size.
Following~\cite{Gilmer-Livingston:1983-1}, we say that a knot $K$ is
\emph{algebraically doubly slice} if $K$ has a hyperbolic Seifert
matrix.  We remark that the following are equivalent:
\begin{enumerate}
\item $K$ has a hyperbolic Seifert matrix.
\item Every Seifert matrix of $K$ is hyperbolic.
\item $K$ has a stably hyperbolic Seifert matrix, that is, the
  orthogonal sum of a Seifert matrix of $K$ and some other hyperbolic
  Seifert matrix is hyperbolic.
\item Every Seifert matrix of $K$ is stably hyperbolic.
\item The Blanchfield form of $K$ is hyperbolic, that is, the
  Alexander module is an internal direct sum $P_1\oplus P_2$ for some
  submodules $P_1$ and $P_2$ which are self-annihilating with respect
  to the Blanchfield form.
\item The Blanchfield form of $K$ is stably hyperbolic, that is, the
  orthogonal sum of the Blanchfield form of $K$ and some other
  hyperbolic Blanchfield form is hyperbolic.
\end{enumerate}

The equivalence $(3) \Leftrightarrow (6) \Leftrightarrow (4)$ seems to
be a classical fact; an explicit proof can be found
in~\cite{Orson:2015-3}.  Equivalences $(1) \Leftrightarrow (3)$,
$(2) \Leftrightarrow (4)$ and $(5) \Leftrightarrow (6)$ are recent
results of Orson~\cite{Orson:2017-1}.

By Proposition~\ref{proposition:homology-splitting} with $R=\Z$, a
doubly 1-solvable knot is algebraically doubly slice. In particular,
our examples in Theorem~\ref{theorem:infinite-rank} are algebraically
doubly slice.

It is known that a knot is doubly 0-solvable (respectively doubly
0.5-solvable) if and only if it has vanishing Arf invariant
(respectively it is algebraically
slice)~\cite[Corollary~2.9]{Kim:2006-1}.

\subsection*{Gilmer-Livingston obstruction}

In \cite{Gilmer-Livingston:1983-1}, Gilmer-Livingston introduced
obstructions for a knot to being doubly slice, using the idea that a
prime power fold cyclic cover of $S^3$ branched over a doubly slice
knot embeds in~$S^4$.  To state their result, we need the following
notations.  For a space $X$ and a homomorphism
$\phi\colon H_1(X)\to \Z_m$ where $m=q^s$ for a prime $q$, let $X_m$
denote the associated $m$-fold cyclic cover of $X$ and $T$ a generator
of the group of covering transformations.  Let $\overline{H_k}(X_m)$
be the $e^{2\pi i/m}$-eigenspace for the action of $T$ on
$H_k(X_m;\C)$. We define
$\overline{\beta_k}(X_m):= \dim \overline{H_k}(X_m)$ and
$\rho_k(X):= \dim H_k(X;\Z_q)$.

Suppose $M$ is a closed 3-manifold and $\phi\colon H_1(M)\to \Z_m$ is
a homomorphism.  Then there is a 4-manifold $V$ with a homomorphism
$\tilde{\phi}\colon H_1(V)\to \Z_m$ such that
$\partial(V,\tilde{\phi})=r(M,\phi)$ for some nonzero integer~$r$.
Let $\sigma_1(V_m)$ be the signature of the intersection form on
$H_2(V_m;\C)$ restricted to $\overline{H_2}(V_m)$. We define
$\sigma(M,\phi):= \frac{1}{r} (\sigma_1(V_m) - \sigma_0(V))$ where
$\sigma_0(V)$ is the ordinary signature of $V$.

Let $n$ be a prime power.  Let $L^n$ be the $n$-fold cyclic cover of
$S^3$ branched over a knot $K$, which is a rational homology 3-sphere.
 
\begin{theorem}[{Gilmer-Livingston obstruction~\cite{Gilmer-Livingston:1983-1}}]
  \label{theorem:Gilmer-Livingston obstructions}
  If $K$ is a doubly slice knot, then the following hold.
  \begin{enumerate}
  \item $H_1(L^n)\cong G_1\oplus G_2$ where $G_1$ and $G_2$ are
    metabolizers of the linking form of~$L^n$.
  \item If $\phi\colon H_1(L^n)\to \Z_m$ is an epimorphism such that
    $\phi(G_1)=0$ or $\phi(G_2)=0$, then
    \[
      |\sigma(L^n,\phi)| + |d-1-\overline{\beta_1}(L^n_m)| \le d
    \]
    where $d=\frac12 \dim H_1(L^n;\Z_q)$.
  \end{enumerate}
\end{theorem}

\begin{proposition}
  \label{proposition:doubly-2-solvability-and-GL-obstructions}
  If a knot $K$ is doubly $2$-solvable, then $K$ has vanishing
  Gilmer-Livingston obstructions, that is, the conclusions (1) and
  (2) of Theorem~\ref{theorem:Gilmer-Livingston obstructions} hold.
\end{proposition}

\begin{proof}
  Suppose $M(K)$ is doubly 2-solvable via $(U^+, U^-)$. Let $V^+$ be
  the 4-manifold obtained by adding a 2-handle to $U^+_n$ along the
  meridian of the preimage of $K$ in~$M(K)_n$.  We define $V^-$ in a
  similar fashion using~$U^-$. Then $\partial V^\pm = L^n$.

  By Proposition~\ref{proposition:homology-splitting}, we have
  \[
    H_1(M(K);\Z[t^{\pm 1}])\cong H_1(U^+;\Z[t^{\pm 1}])\oplus
    H_1(U^-;\Z[t^{\pm 1}])
  \]
  Also, each of $H_1(U^\pm;\Z[t^{\pm 1}])$ is a self-annihilating
  submodule of the Blanchfield form.  Since $H_1(M(K);\Z[t^{\pm 1}])$
  is $\Z$-torsion free, so are $H_1(U^\pm;\Z[t^{\pm 1}])$.  By taking
  the quotients by the submodule generated by $t^n-1$, it follows that
  $H_1(L^n)\cong H_1(V^+)\oplus H_1(V^-)$ and that $H_1(V^\pm)$ are
  metabolizers of the linking form of $L^n$.

  For brevity, let $U:=U^+$ and $V:=V^+$. Suppose that
  $\phi\colon H_1(L^n)\to \Z_m$ is an epimorphism which
  vanishes on $H_1(V^-)$. Then the map $\phi$ factors through
  $H_1(V)$. Denote by $\tilde{\phi}$ the homomorphism
  $H_1(V )\to \Z_m$. Now we have
  $\partial (V,\tilde{\phi})=(L^n, \phi)$.

  Since $U$ is a 2-solution, for $r:=\frac12 \dim H_2(U;\Q)$, we have surfaces $x_i$ and $y_i$
  $(1\le i\le r)$ in $U$ whose homology classes $[x_i]$ and $[y_i]$
  satisfy the conditions (1) and (2) in
  Definition~\ref{definition:n-solution}. Let $x^j_i$ and $y^j_i$
  $(1\le i\le r, 1\le j\le n)$ be the lifts of $x_i$ and $y_i$ in $V$.
  Since $\chi(U_n)=2rn$, we have $\chi(V) = 1+ 2rn$.  This implies
  $\dim H_2(V;\Q) = 2rn$.  It follows that $[x^j_i]$ and $[y^j_i]$
  generate $H_2(V;\Q)$, since $x^j_i$ and $y^j_i$ are dual to each
  other with respect to the intersection form on $H_2(V;\Q)$. Since
  the intersection form on $H_2(V;\Q)$ vanishes on the span of
  $[x^j_i]$, we have $\sigma_0(V)=0$. It follows that
  $\sigma(L^n,\phi)=\sigma_1(V_m) - \sigma_0(V) = \sigma_1(V_m)$.

  Since $H_3(V;\Z_q)\cong H^1(V,L^n;\Z_q)=0$, we have $\rho_3(V)=0$.
  By \cite[Proposition~1.4]{Gilmer:1981-1}, which states
  $\overline{\beta_3}(V_m) \le \rho_3(V)$, it follows that
  $\overline{\beta_3}(V_m)=0$, and hence $\overline{H_3}(V_m)=0$.  By
  duality we have
  $\overline{H_1}(V_m,L^n_m)\cong \overline{H_3}(V_m)=0$. Then the
  sequence
  \[
    \overline{H_2}(V_m)\to \overline{H_2}(V_m,L^n_m)\to
    \overline{H_1}(L^n_m)\to \overline{H_1}(V_m)\to 0
  \]
  is exact. Therefore, letting $N$ be the nullity of the matrix for
  the intersection form restricted to $\overline{H_2}(V_m)$, we have
  $|\sigma_1(V_m)|+N\le \overline{\beta_2}(V_m)$ and
  $\overline{\beta_1}(L^n_m) = N + \overline{\beta_1}(V_m)$.

  We assert that $|\sigma_1(V_m)|+ N\le \overline{\beta_2}(V_m)-2rn$. The proof is as follows.  Regard
  $\C$ as a $\Z[\Z_m]$-module where $T$, a generator of the group of
  the covering transformations $\Z_m$, acts on $\C$ via multiplication
  by~$e^{2\pi i/m}$.  Define
  $H_*(V;\C^\omega):=H_*(C_*(V_m)\otimes_{\Z[\Z_m]} \C)$, where the
  coefficient group $\C$ is twisted as above. Let
  $\lambda\colon H_2(V;\C^\omega)\times H_2(V;\C^\omega)\to \C$ be the
  (twisted) intersection form. Since $\C$ is flat as a
  $\Z[\Z_m]$-module, $H_*(V;\C^\omega)$ is isomorphic to
  $H_*(V_m)\otimes_{\Z[\Z_m]} \C$. Note that $V_m$ is a metabelian
  cover of $U$. Therefore, the surfaces $x^j_i$ and $y^j_i$ can be
  lifted to $V_m$. Since $U$ is a 2-solution, by our choice of $x_i$
  and $y_i$, we have $\lambda([x^j_i], [x^\ell_k]) = 0$ and
  $\lambda ([x^j_i], [y^\ell_k]) = \delta_{ik}\delta_{j\ell}$ where
  $1\le i,k\le r$ and $1\le j,\ell \le n$. Let
  $\overline{\lambda}\colon
  \overline{H_2}(V_m)\times\overline{H_2}(V_m) \to \C$ be the
  intersection form on $H_2(V_m;\C)$ restricted to
  $\overline{H_2}(V_m)$. Then,
  $H_2(V;\C^\omega)\cong \overline{H_2}(V_m)$ and
  $\lambda(x,y) = m\overline{\lambda}(x,y)$ (see
  \cite{Casson-Gordon:1986-1}). It follows that $[x^j_i]$ and
  $[y^j_i]$ span a $2rn$-dimensional subspace of $\overline{H_2}(V_m)$
  such that $\overline{\lambda}$ vanishes on the span of $[x^j_i]$,
  which is a subspace of dimension $rn$. Moreover, this
  $2rn$-dimensional subspace has trivial intersection with the
  nullspace of $\overline{\lambda}$. Therefore,
  $|\sigma_1(V_m)|+ N\le \overline{\beta_2}(V_m)-2rn$, as asserted
  above.

  Combining the above assertion with $\sigma(L^n,\phi)=\sigma_1(V_m)$
  and $\overline{\beta_1}(L^n_m) = N + \overline{\beta_1}(V_m)$, we
  obtain that
  $|\sigma(L^n,\phi)|+ \overline{\beta_1}(L^n_m)\le
  \overline{\beta_2}(V_m) + \overline{\beta_1}(V_m)-2rn$

  We compute $\overline{\beta_2}(V_m)$.  We have
  $\overline{\chi}(V_m)=\chi(V)=1+2rn$, where the first equality is
  due to~\cite[Proposition~1.1]{Gilmer:1981-1}.  Since $\tilde{\phi}$
  is not trivial, $\overline{\beta_0}(V_m)=0$. Since $V_m$ is not
  closed, $\overline{\beta_k}(V_m) = 0$ for $k\ge 4$. We have shown
  $\overline{\beta_3}(V_m)=0$ above, and therefore
  $\overline{\beta_2}(V_m) = \overline{\beta_1}(V_m)+1+2rn$.

  Therefore, we have
  $|\sigma(L^n,\phi)|+ \overline{\beta_1}(L^n_m)\le
  2\overline{\beta_1}(V_m) + 1$.  Since $\overline{H_1}(V_m,L^n_m)=0$,
  we have $\overline{\beta_1}(V_m)\le \overline{\beta_1}(L^n_m)$.
  Also,
  \[
    \overline{\beta_1}(V_m)\le \rho_1(V)-1 = \dim H_1(V;\Z_q)-1 =
    \tfrac12 \dim H_1(L^n;\Z_q)-1 = d-1,
  \]
  where the inequality is due to
  \cite[Proposition~1.5]{Gilmer:1981-1}.  It follows that
  \[
  |\sigma(L^n,\phi)|+ \overline{\beta_1}(L^n_m)\le
  2\min\{\overline{\beta_1}(L^n_m), d-1\} + 1.
  \]
  It is straightforward to verify that this is equivalent to
  \[
    |\sigma(L^n,\phi)| + |d-1-\overline{\beta_1}(L^n_m)| \le d.\qedhere
  \]
\end{proof}

\subsection*{Friedl obstruction}

In \cite{Friedl:2003-4}, Friedl gave double silceness obstructions
using $\eta$-invariants.  He defined a collection, denoted by
$P^{irr}_k(H_1(M(K);\Z[t^{\pm 1}])\rtimes \Z)$, of certain
$k$-dimensional unitary representations of the group
$H_1(M(K);\Z[t^{\pm 1}])\rtimes \Z$, and considered $\eta$-invariants,
denoted by $\eta(M(K),\alpha)$, of the zero surgery manifold $M(K)$
associated to $\alpha\in P^{irr}_k(H_1(M(K);\Z[t^{\pm 1}])\rtimes \Z)$.
Since we do not need precise descriptions of them, we omit details.
Refer to~\cite[Sections~3 and~4]{Friedl:2003-4}.

\begin{theorem}[{Friedl obstruction~\cite[Theorem~8.4]{Friedl:2003-4}}]
  \label{theorem:friedl-obstruction}
  If $K$ is a doubly slice knot, then there exist self-annihilating
  submodules $P_1$ and $P_2$ of $H_1(M(K);\Z[t^{\pm 1}])$ for the
  Blanchfield form such that
  \begin{enumerate}
  \item $H_1(M(K);\Z[t^{\pm 1}])\cong P_1\oplus P_2$, and
  \item $\eta(M(K), \alpha) = 0$ for any
    $\alpha \in P^{irr}_k(H_1(M(K);\Z[t^{\pm 1}])\rtimes \Z)$ which
    vanishes on either $P_1$ or~$P_2$.
  \end{enumerate}
\end{theorem}

\begin{proposition}
  \label{proposition:doubly-1.5-solvability-and-friedl-obstruction}
  If a knot $K$ is doubly $1.5$-solvable, then $K$ has vanishing
  Friedl obstructions, that is, the conclusions of
  Theorem~\ref{theorem:friedl-obstruction} hold.
\end{proposition}

\begin{proof}
  Suppose that $K$ is doubly 1.5-solvable via $(W_1, W_2)$. Let $P_i$
  be the kernel of the map
  $H_1(M(K);\Z[t^{\pm 1}]) \to H_1(W_i; \Z[t^{\pm 1}])$. By
  Proposition~\ref{proposition:homology-splitting}, each $P_i$ is a
  self-annihilating submodule with respect to the Blanchfield form and
  $H_1(M(K);\Z[t^{\pm 1}])\cong P_1\oplus P_2$. Therefore the
  conclusion~(1) in Theorem~\ref{theorem:friedl-obstruction}
  holds. Let $k$ be a prime power and
  $\alpha\in P^{irr}_k(H_1(M(K);\Z[t^{\pm 1}])\rtimes \Z)$, following
  the notation in \cite{Friedl:2003-4}. Suppose $\alpha$ vanishes on
  either $P_1$ or~$P_2$. Let $M_k$ denote the $k$-fold cyclic cover of
  $M(K)$ and $L^k$ the $k$-fold cyclic cover of $S^3$ branched over
  $K$. Let $m$ be a prime power. By
  \cite[Proposition~5.4]{Friedl:2003-4},
  $\eta(M(K),\alpha) = \eta(M_k,\beta)$ for some 1-dimensional unitary
  representation $\beta:\pi_1M_k\to U(1)$ which is associated to a
  certain character $\chi\colon H_1(L^k)\to \Z_m$ and vanishes on
  $P_i/(t^k -1)$. Since $W_i$ are 1.5-solutions, the Casson-Gordon
  invariant $\tau(K, \chi) = 0$ (see
  \cite[Section~9]{Cochran-Orr-Teichner:1999-1}). By
  \cite[Proposition~5.3]{Friedl:2003-5}, we have
  $\eta(M_k, \beta) = 0$, and hence $\eta(M(K), \alpha)=0$. Therefore,
  the conclusion~(2) in Theorem~\ref{theorem:friedl-obstruction}
  holds.
\end{proof}

\subsection*{Livingston-Meier obstruction}

In \cite{Livingston-Meier:2015-1}, Livingston and Meier gave double
sliceness obstructions using twisted Alexander polynomials. Let $n$ be
a prime power and $p$ an odd prime. For a knot $K$, recall that $L^n$
denote the $n$-fold cyclic cover of $S^3$ branched over $K$. Let
$\zeta_p$ be a primitive $p$th root of unity and
$\Gamma_p:=\Q(\zeta_p)[t^{\pm 1}]$. For a homomorphism
$\rho\colon H_1(L^n)\to \Z_p$, let $\Delta_{K,\rho}(t)$ be the twisted
Alexander polynomial of $K$ associated to $\rho$, which is defined
in~\cite{Kirk-Livingston:1999-2}.

\begin{theorem}[{Livingston-Meier obstruction~\cite[Theorem~3.2]{Livingston-Meier:2015-1}}]
  \label{theorem:livingston-meier-obstruction}
  If $K$ is a doubly slice knot, then there exist subgroups $G_1$ and
  $G_2$ of $H_1(L^n)$ such that
  \begin{enumerate}
  \item $H_1(L^n)\cong G_1\oplus G_2$ where $G_i$ are invariant under
    the action of the covering transformations of $H_1(L^n)$,
  \item for every homomorphism $\rho\colon H_1(L^n)\to \Z_p$ which
    vanishes on $G_1$ or $G_2$, we have
    $\Delta_{K,\rho}(t)=af(t)\overline{f(t)}$ for some unit
    $a\in \Gamma_p$ and $f(t)\in \Gamma_p$.
  \end{enumerate}
\end{theorem}

\begin{proposition}
  \label{proposition:doubly-1.5-solvability-and-livingston-meier-obstruction}
  If a knot $K$ is doubly $1.5$-solvable, then $K$ has vanishing
  Livingston-Meier obstructions, that is, the conclusions of
  Theorem~\ref{theorem:livingston-meier-obstruction} hold.
\end{proposition}

\begin{proof}
  Suppose $K$ is doubly $1.5$-solvable via $(U^+,U^-)$. Using the same
  arguments as in the proof of
  Proposition~\ref{proposition:doubly-2-solvability-and-GL-obstructions},
  we can show that there are subgroups $G_i$ for $i=1,2$ such that
  $H_1(L^n)\cong G_1\oplus G_2$. Since $G_i$ are obtained by taking a
  quotient of a $\Z[t^{\pm 1}]$-module by $t^n-1$, $G_i$ are invariant
  under the action of the covering transformations.  This shows~(1).

  Suppose that $\rho\colon H_1(L^n)\to \Z_p$ vanishes on $G_1$ or
  $G_2$. Using the notations in the proof of
  Proposition~\ref{proposition:doubly-2-solvability-and-GL-obstructions},
  observe that $G_1$ and $G_2$ are the torsion subgroups of
  $H_1(U^+_n)$ and $H_1(U^-_n)$ respectively.  Then by
  \cite[Theorem~9.11]{Cochran-Orr-Teichner:1999-1} and its proof, the
  Casson-Gordon discriminant invariant $K$ associated to $\rho$
  vanishes. Then by \cite[Theorem~6.5]{Kirk-Livingston:1999-2},
  $\Delta_{K,\rho}(t)$ has a desired factorization required in~(2).
\end{proof}

\bibliographystyle{amsalpha}
\renewcommand{\MR}[1]{}
\bibliography{research}

\end{document}